\documentclass[11pt]{article}

\usepackage{amsmath,amssymb,amsthm,amsfonts,verbatim}
\usepackage{graphicx}
\usepackage{enumerate}
\usepackage[all,2cell]{xy}
\CompileMatrices

\usepackage[top=1.5in,bottom=1.5in,left=1.35in,right=1.35in]{geometry}

\title{Representation theory and homological stability}

\author{Thomas Church and Benson Farb \thanks{The authors
    gratefully acknowledge support from the National Science
    Foundation.}}

\theoremstyle{plain}
\newtheorem{theorem}{Theorem}[section]

\newtheorem{proposition}[theorem]{Proposition}

\newtheorem{lemma}[theorem]{Lemma}
\newtheorem{corollary}[theorem]{Corollary}

\newtheorem{conjecture}[theorem]{Conjecture}
\newtheorem{question}[theorem]{Question}

\newtheorem*{theorem:pc}{Theorem \ref{theorem:principal congruence}}
\newtheorem*{theorem:jk}{Theorem \ref{theorem:johnson}}
\newtheorem*{theorem:br}{Theorem \ref{theorem:brunnian}}
\theoremstyle{definition}
\newtheorem{remark}[theorem]{Remark}
\newtheorem{definition}[theorem]{Definition}

\newtheorem{xample}[theorem]{Example}
\newtheorem{nonexample}[theorem]{Non-example}

\newcommand{\nc}{\newcommand}
\nc{\dmo}{\DeclareMathOperator}

\nc{\I}{\mathcal{I}}
\nc{\K}{\mathcal{K}}
\nc{\U}{\mathcal{U}}
\renewcommand{\H}{\mathcal{H}}
\renewcommand{\L}{\mathcal{L}}
\nc{\Q}{\mathbb{Q}}
\nc{\R}{\mathbb{R}}
\nc{\Z}{\mathbb{Z}}
\nc{\C}{\mathbb{C}}
\nc{\N}{\mathbb{N}}
\dmo{\GL}{GL}
\dmo{\PSL}{PSL}
\nc{\gin}{i}
\nc{\ga}{\Gamma}
\dmo{\Out}{Out}
\dmo{\Aut}{Aut}
\dmo{\Stab}{Stab}

\dmo\im{im}
\dmo\id{id}
\dmo\SL{SL}
\dmo\Sp{Sp}
\dmo\Mod{Mod}
\dmo\PMod{PMod}
\dmo\fd{fd}
\dmo\F{\mathbb{F}}
\dmo\IA{IA}
\dmo\Schur{\mathbb{S}}
\dmo\Sym{Sym}
\dmo\Ind{Ind}
\dmo\Res{Res}
\dmo\tr{tr}
\dmo\gr{gr}
\dmo\Free{Free}
\dmo\spn{span}
\dmo\codim{codim}

\def\Fpbar{\overline{\mathbb{F}}_p}
\dmo\Tor{\mathcal{T}}
\dmo\Torsion{tor}
\def\fsl{\mathfrak{sl}}
\def\fh{\mathfrak{h}}
\def\fn{\mathfrak{n}}
\def\fp{\mathfrak{p}}

\def\fsp{\mathfrak{sp}}
\nc{\bwedge}{\textstyle{\bigwedge}}

\def\cF{\mathcal{F}}
\def\fgl{\mathfrak{gl}}

\renewcommand{\epsilon}{\varepsilon}
\nc{\coloneq}{\mathrel{\mathop:}\mkern-1.2mu=}
\nc{\margin}[1]{}
\nc{\para}[1]{\medskip\noindent\textbf{#1.}}

\begin{document}
\maketitle
\begin{abstract}
  We introduce the idea of \emph{representation stability} (and
  several variations) for a sequence of representations $V_n$ of
  groups $G_n$.  A central application of the new viewpoint we introduce here is
the importation of representation theory into the study of homological stability. 
This makes it possible to extend classical theorems of homological stability to a much broader variety of examples.   Representation stability also
  provides a framework in which to find and to predict patterns, from
  classical representation theory (Littlewood--Richardson and Murnaghan rules,
  stability of Schur functors), to cohomology of groups (pure braid,
  Torelli and congruence groups), to Lie algebras and their homology,
  to the (equivariant) cohomology of flag and Schubert varieties, to
  combinatorics (the $(n+1)^{n-1}$ conjecture).  The majority of this paper is devoted to 
exposing this phenomenon through examples.  In doing this we obtain 
applications, theorems and conjectures.  

Beyond the discovery of new phenomena, the viewpoint of representation stability can be useful in solving problems outside the theory.  In addition to the applications given in this paper, it is applied in \cite{CEF2} to counting problems in number theory and finite group theory.   Representation stability is also used in \cite{C} to give broad generalizations and new proofs 
  of classical homological stability theorems for configuration spaces on oriented manifolds.
  
  \end{abstract}

\maketitle

\newpage
\tableofcontents

\section{Introduction}
In this paper we introduce the idea of \emph{representation stability} (and
  several variations) for a sequence of representations $V_n$ of
  groups $G_n$. 
A central application of the new viewpoint we introduce here is
the importation of representation theory into the study of homological stability. 
This make it possible to extend classical theorems of homological stability to a much broader variety of examples. 
Representation stability also
  provides a framework in which to find and to predict patterns, from
  classical representation theory (Littlewood--Richardson and Murnaghan rules,
  stability of Schur functors), to cohomology of groups (pure braid,
  Torelli and congruence groups), to Lie algebras and their homology,
  to the (equivariant) cohomology of flag and Schubert varieties, to
  combinatorics (the $(n+1)^{n-1}$ conjecture).  The majority of this paper is devoted to 
exposing this phenomenon through examples.  In doing this we obtain 
applications, theorems and conjectures.  

Beyond the discovery of new phenomena, the viewpoint of representation stability can be useful in solving problems outside the theory.  In addition to the applications given in this paper, representation stability is used in \cite{C} to give broad generalizations and new proofs 
  of classical homological stability theorems for configuration spaces on oriented manifolds. In \cite{CEF2} representation stability is applied to counting problems in number theory. 
  
We begin with some context and motivation.

\para{Classical homological stability} Let $\{Y_n\}$ be a sequence of
groups, or topological spaces, equipped with maps (e.g.\ inclusions)
$\psi_n\colon Y_n\to Y_{n+1}$. The sequence $\{Y_n\}$ is
\emph{homologically stable} (over a coefficient ring $R$) if for each
$i\geq 1$ the map
\[(\psi_n)_\ast\colon H_i(Y_n,R)\to H_i(Y_{n+1},R)\] is an isomorphism
for $n$ large enough (depending on $i$).  Homological stability is
known to hold for certain sequences of arithmetic groups, such as
$\{\SL_n\Z\}$ and $\{\Sp_{2n}\Z\}$.  It is also known for braid groups,
mapping class groups of surfaces with boundary, and for (outer)
automorphism groups of free groups, by major results of many people
(including Borel, Arnol'd, Harer, Hatcher--Vogtmann and many others; 
see, e.g.\ \cite{Coh,Vo} and the references therein).  Further,
in many of these cases the stable homology groups have been computed.

In contrast, even for $\Q$ coefficients (or for $\F_p$ coefficients in the arithmetic examples), 
almost nothing is known about the homology of finite
index and other natural subgroups of the above-mentioned groups, even
in the simplest examples.  Indeed, homological stability is known to
fail in many cases, and it is not even clear what a closed-form description
of the homology might look like.  We now consider an example to
illustrate this point.

\para{A motivating example} Consider the set $X_n$ of ordered
$n$--tuples of distinct points in the complex plane:
\[X_n\coloneq \big\{(z_1,\ldots,z_n)\in \C^n\big|z_i\neq z_j\text{ for
  all }i\neq j\big\}.\] The set $X_n$ can be considered as a
hyperplane complement in $\C^n$.  The fundamental group of $X_n$ is
the \emph{pure braid group} $P_n$.  It is known that $X_n$ is
aspherical, and so $H_i(P_n;\Z)=H_i(X_n;\Z)$.  The symmetric group $S_n$
acts freely on $\C^n$ by permuting the coordinates, and this action
clearly restricts to a free action by homeomorphisms on $X_n$.  The
quotient $Y_n\coloneq X_n/S_n$ is the space of unordered $n$--tuples
of distinct points in $\C$.  The space $Y_n$ is aspherical, and so is
a classifying space for its fundamental group $B_n$, the \emph{braid
  group}.  We have an exact sequence:
\[1\to P_n\to B_n\to S_n\to 1\] 

Arnol'd \cite{Ar} and F.\ Cohen \cite{Co} proved that the sequence of
braid groups $\{B_n\}$ satisfies homological stability with integer
coefficients.  Over the rationals, they proved for $n\geq 3$ that
\[H_i(B_n;\Q)=
\begin{cases}
\Q&\text{if }i=0,1\\
0&\text{if }i\geq 2
\end{cases}
\]
and so stability holds in a trivial way.  In contrast, 
\[H_1(P_n;\Q)=\Q^{n(n-1)/2}\] and so the pure braid groups $\{P_n\}$
do not satisfy homological stability, even for $i=1$.

Arnol'd also gave a presentation for the cohomology algebra
$H^\ast(P_n;\Q)$ (see \S\ref{section:braids} for the description).
But we can try to extract much finer information, using representation
theory, as follows.  The action of $S_n$ on the space $X_n$ induces an
action of $S_n$ on the vector space $H^i(P_n;\Q)$, making it an
$S_n$--representation for each $i\geq 0$.  Each of these
representations is finite-dimensional, and so can be decomposed as a
finite direct sum of irreducible $S_n$--representations.  The question
of how many of these summands are trivial is already interesting: an
easy transfer argument gives that \[H^i(B_n;\Q)=H^i(P_n;\Q)^{S_n};\]
that is, $H^i(B_n;\Q)$ is the subspace of $S_n$--fixed vectors in
$H^i(P_n;\Q)$.  Thus we see that the ``trivial piece'' of
$H^i(P_n;\Q)$ already contains the Arnol'd--Cohen computation of
$H^i(B_n;\Q)$; the other summands evidently contain even deeper
information.

Now, the irreducible representations of $S_n$ are completely
classified: they are in bijective correspondence with partitions
$\lambda$ of $n$.  Which irreducibles (that is, which partitions)
occur in the $S_n$--representation $H^i(P_n;\Q)$? What are their
multiplicities?  There have been a number of results in this direction
(most notably by Lehrer--Solomon \cite{LS}), but an explicit count of
the multiplicity of a fixed partition $\lambda$ is known only for a
few $\lambda$; an answer for arbitrary $\lambda$ and arbitrary $i$
seems out of reach.

On the other hand, using the notation $V(a_1,\ldots ,a_r)$ to denote
the irreducible $S_n$--representation corresponding to the partition
$((n-\sum_{i=1}^r a_i), a_1, \ldots ,a_r)$ (see \S\ref{section:reps1}
below for more details), it is not hard to check that
\begin{equation}
\label{eq:h1pn}
H^1(P_n;\Q)=V(0)\oplus V(1)\oplus V(2) \qquad \text{for } n\geq 4.
\end{equation}
Note that, with our notation, the right-hand side of \eqref{eq:h1pn}
has a uniform description, independent of $n$ as long as $n\geq 4$.
More interestingly, using work of Lehrer--Solomon and the computer program Magma, we computed the following:
\begin{equation}
\label{eq:stabilizationinaction}
\begin{array}{l}
  H^2(P_4;\Q) = V(1)^{\oplus 2}\oplus V(1,1) \oplus V(2)\\  \\
  H^2(P_5;\Q) = V(1)^{\oplus 2} \oplus V(1,1)^{\oplus 2}
  \oplus V(2)^{\oplus 2} \oplus 
  V(2,1)\\  \\
  H^2(P_6;\Q) = V(1)^{\oplus 2} \oplus  V(1,1)^{\oplus 2}
  \oplus V(2)^{\oplus 2} \oplus 
  V(2,1)^{\oplus 2} \oplus V(3)\\  \\
  H^2(P_n;\Q) = V(1)^{\oplus 2} \oplus V(1,1)^{\oplus 2}
  \oplus V(2)^{\oplus 2} \oplus 
  V(2,1)^{\oplus 2} \oplus  V(3) \oplus V(3,1)
\end{array}
\end{equation}
where we carried out the computation in the last line for $n=7$, $8$, and $9$. We will see below that the last line of \eqref{eq:stabilizationinaction} in fact holds for all $n\geq 7$, so the irreducible 
decomposition of $H^2(P_n;\Q)$ \emph{stabilizes}.  These low-dimensional ($i=1$, $2$) cases are indicative of a more general
pattern.  The language needed to describe this pattern is given by the
main concept in this paper, which we now describe (in a special case).

\para{Representation stability}
Let $V_n$ be a sequence of $S_n$--representations, equipped with
linear maps $\phi_n\colon V_n\to V_{n+1}$, making the following diagram
commute for each $g\in S_n$:
\[\xymatrix{
  V_n\ar^{\phi_n}[r]\ar_{g}[d]&V_{n+1}\ar^{g}[d]\\
  V_n\ar_{\phi_n}[r]&V_{n+1} }\]
where $g$ acts on $V_{n+1}$ by its image under the standard inclusion
$S_n\hookrightarrow S_{n+1}$.  We call such a sequence of
representations \emph{consistent}.

We want to compare the representations $V_n$ as $n$ varies.  However,
since $V_n$ and $V_{n+1}$ are representations of different groups, we
cannot ask for an isomorphism as representations. But we can ask for
injectivity and surjectivity, once they are properly
formulated. Moreover, by using the decomposition into irreducibles, we
can formulate what it means for $V_n$ and $V_{n+1}$ to be the ``same
representation''.

\begin{definition}[Representation stability, special case] 
  Let $\{V_n\}$ be a consistent sequence of $S_n$--representations.  We
  say that the sequence $\{V_n\}$ is \emph{representation stable} if,
  for sufficiently large $n$, each of the following conditions holds:
\begin{enumerate}[I.]
\item \textbf{Injectivity:} The maps $\phi_n\colon V_n\to V_{n+1}$ are 
  injective.
\item \textbf{Surjectivity:} The span of the $S_{n+1}$--orbit of
  $\phi_n(V_n)$ equals all of $V_{n+1}$.
\item \textbf{Multiplicities:} Decompose $V_n$ into irreducible
  $S_n$--representations as
  \[V_n=\bigoplus_\lambda c_{\lambda,n}V(\lambda)\] with
  multiplicities $0\leq c_{\lambda,n}\leq \infty$. For each $\lambda$,
  the multiplicities $c_{\lambda,n}$ are eventually independent of
  $n$.
\end{enumerate}
\end{definition}

The idea of representation stability can be extended to other families
of groups whose representation theory has a ``consistent naming
system'', for example $\GL_n\Q$, $\Sp_{2n}\Q$ and the hyperoctahedral
groups; see \S\ref{section:repstab:def} for the precise definitions.  

As an easy example, let $V_n=\Q^n$ denote the standard representation
of $\GL_n\Q$.  Then the decomposition $V_n\otimes
V_n=\Sym^2V_n\oplus\bwedge^2V_n$ into irreducibles shows that the
sequence of $\GL_n\Q$--representations $\{V_n\otimes V_n\}$ is
representation stable; see Example~\ref{example:tensor}.  A natural
non-example is the sequence of regular representations $\{\Q S_n\}$ of
$S_n$.  These are not representation stable since, for any partition
$\lambda$, the multiplicity of $V(\lambda)$ in $\Q S_n$ is
$\dim(V(\lambda))$, which is not constant, and indeed tends to
infinity with $n$.

In \S\ref{section:reps1} and \S\ref{section:repsGLSp} we review the representation theory of all the groups we will be considering. In \S\ref{section:repstab:def} we develop the foundations of representation stability, in particular giving a number of useful examples, variations and refinements, such as uniform stability. In particular, we introduce strong stability, used
when one wishes to more finely control the $G_{n+1}$--span of the image of $V_n$ under 
$\phi_n$; this is important for applications.  We also develop the idea of ``mixed tensor stability'', which is meant to capture in certain cases subtle phenomena not detected by representation stability.  

With the above language in hand, we can state our first theorem.
``Forgetting the $(n+1)^{\text{st}}$ marked point'' gives a
homomorphism $P_{n+1}\to P_n$ and thus induces a homomorphism
$H^{i}(P_n;\Q)\to H^{i}(P_{n+1};\Q)$.  For each fixed $i\geq 1$ the
sequence of $S_n$--representations $\{H^i(P_n;\Q)\}$ is 
consistent in the sense given above.  While the exact multiplicities
in the decomposition of $H^i(P_n;\Q)$ into $S_n$--irreducible
subspaces are far from known, we have discovered the following.

\bigskip
\noindent
\textbf{Theorem~\ref{thm:pure}, slightly weaker version.}  {\it For each fixed
  $i\geq 0$, the sequence of $S_n$--representations $\{H^i(P_n;\Q)\}$
  is representation stable.  Indeed the sequence stabilizes once $n\geq 4i$.}

\bigskip See \S\ref{section:braids} for the proof.   Note that the example in \eqref{eq:stabilizationinaction} above shows that the ``stable range'' we give in 
Theorem \ref{thm:pure} is close to being sharp.  

The obvious explanation for the stability in Theorem~\ref{thm:pure}
would be that that each $V(\lambda)\subseteq H^i(P_n;\Q)$ includes
into $H^i(P_{n+1};\Q)$ with $S_{n+1}$--span equal to
$V(\lambda)$, at least for $n$ large enough.  But in fact this
coherence never happens, even for the trivial representation $V(0)$.  Thus the mechanism 
for stability of multiplicities in $\{H^i(P_n;\Q)\}$ must be more subtle, and indeed it is 
perhaps surprising that this stability occurs at all.  
See \S\ref{section:braid} for a discussion.

To prove Theorem~\ref{thm:pure}, we use work of
Lehrer--Solomon \cite{LS} to reduce the problem to a statement about
stability for certain sequences of induced representations of $S_n$.
We conjectured this stability to D.~Hemmer \cite{He}, who then proved
it (and more).  A.~Putman has informed us that he has a different
approach to Theorem~\ref{thm:pure}.  In \S\ref{section:braids} we
derive classical homological stability for $B_n$ with twisted
coefficients as a corollary of Theorem~\ref{thm:pure}. We extend these results to generalized braid groups in \S\ref{section:gbg}.

\para{Three applications}
In joint work \cite{CEF2} with Jordan Ellenberg, we use the
Grothendieck--Lefschetz trace formula to translate results proved here
on the representation-stable cohomology of spaces into counting theorems about points on varieties
over finite fields.  We then apply this to obtain statistics for 
polynomials over $\F_q$ and for maximal tori in certain finite groups of Lie type such as $\GL_n(\F_q)$.   

For each fixed partition $\lambda$, stability for  
the multiplicity of $V(\lambda)$ in $\{H^i(P_n;\Q)\}$ is
related to a different counting problem in $\F_q[T]$.  For example,
Theorem~\ref{thm:pure} for the sign representation $V(1,\ldots ,1)$
implies that the discriminant of a random monic squarefree polynomial
is equidistributed between residues and non-residues in $\F_q^{\times}$.
Theorem~\ref{thm:pure} for the standard representation $V(1)$ implies
that the expected number of linear factors of a random monic
squarefree polynomial of degree $n$ is
\[1-\frac{1}{q}+\frac{1}{q^2}-\frac{1}{q^3}+\cdots\pm\frac{1}{q^{n-2}}.\]
The stability of
$\{H^i(P_n;\Q)\}$ itself, even without knowing what the stable
multiplicities are, already implies that the associated counting
problems all have limits as the degree of the polynomials tends to
infinity.   One can also obtain the Prime Number Theorem for $\F_q[T]$, counting the number of
irreducible polynomials of degree $n$, this way. At present this approach 
reproduces results already known to analytic number theorists,
but our methods should generalize to wider classes of
examples, such as sections of line bundles on curves other than $\mathbf{P}^1$.  

In \cite{CEF2} we also give an application of representation stability of the cohomology of 
flag varieties (see Section~\ref{section:flags}), obtaining for each $V(\lambda)$ a counting theorem for maximal tori in $\GL_n\F_q$ and for Lagrangian tori in $\Sp_{2n}\F_q$. For the trivial representation
$V(0)$ we obtain Steinberg's theorem that the number of maximal tori in $\GL_n\F_q$ is $N=q^{n^2-n}$.  The standard representation $V(1)$ gives a formula for the expected number of eigenvectors of a random maximal torus in $\GL_n(\F_q)$ which are defined over $\F_q$.  The
sign representation gives a theorem of Srinivasan \cite[Lemma 5]{Sr}: when splitting a random maximal torus into irreducible factors, the number of factors is more likely to be even than odd, with bias exactly $\frac{1}{\sqrt{N}}$.

Another application of representation stability is given in \cite{C}. Thinking of  Theorem~\ref{thm:pure} as a statement about the configuration space of points in the plane, this is generalized to prove representation stability for the cohomology of ordered configuration spaces on an arbitrary orientable manifold.  Specializing to the case of stability for the trivial representation 
already gives new proofs and vast generalizations of classical homological stability 
theorems of McDuff and Segal for open manifolds. One reason these theorems were not known for general manifolds is that for closed manifolds, there are no maps connecting the unordered configuration spaces for different numbers of points, so it is hard to compare these different spaces, and indeed homological stability often fails integrally. Looking instead at representation stability for the ordered configuration spaces makes it possible to relate different configuration spaces, then push the results down to unordered configuration spaces by taking invariants.

\para{Representation stability in group homology} The example of pure
braid groups given above fits into a much more general framework.
Suppose $\Gamma$ is a group with normal subgroup $N$ and quotient
$A\coloneq \Gamma/N$.  The conjugation action of $\Gamma$ on $N$
induces a $\Gamma$--action on the group homology (and cohomology) of
$N$, with any coefficients $R$. This action factors through an
$A$--action on $H_i(N,R)$, making $H_i(N,R)$ into an $A$--module.

As with pure braid groups, the structure of $H_i(N,R)$ as an
$A$--module encodes fine information.  For example, the
transfer isomorphism shows that when $A$ is finite and $R=\Q$ the space
$H_i(\Gamma;\Q)$ appears precisely as the subspace of $A$--fixed
vectors in $H_i(N;\Q)$.  But there are typically many other summands,
and knowing the representation theory of $A$ (over $R$) gives us a
language with which to access these.

The following table summarizes some of the examples fitting in to this
framework.  Each example will be explained in detail later in this
paper: the first in \S\ref{section:braids}, the second and third in \S\ref{section:torelli}, and the fourth and fifth in \S\ref{section:congruence}.

\vspace{.2in}
\begin{tabular}{c|c|c|c|c}
  kernel $N$ & group $\Gamma$ & acts on & quotient $A$
  & $H_1(N,R)$ for big $n$\\
  \hline 
  & & & & \\
  $P_n$ & $B_n$ & $\{1,\ldots ,n\}$ & $S_n$ & ${\rm Sym}^2V/V$\\
  & & & & \\
  Torelli group $\I_n$ & mapping class&$H_1(\Sigma_n,\Z)$&$\Sp_{2n}\Z$& 
  $\bwedge^3V/V$ \\
  & group $\Mod_n$ & & & \\
  & & & & \\
  $\IA(F_n)$&$\Aut(F_n)$&$H_1(F_n,\Z)$&$\GL_n\Z$
  &$V^\ast\otimes \bwedge^2V$\\
  & & & & \\
  congruence&$\SL_n\Z$&$\F_p^n$&$\SL_n\F_p$&$\fsl_n\F_p$\\
  subgroup $\Gamma_n(p)$& & & &\\
  & & & & \\
  level $p$ subgroup&$\Mod_n$&$H_1(\Sigma_n;\F_p)$&$\Sp_{2n}\F_p$&
  $\bwedge^3 V/V\oplus\ \fsp_{2n}\F_p$\\
  $\Mod_n(p)$&&&&
  \end{tabular}

\vspace{.3in} Here $R=\Q$ in the first three examples, $R=\F_p$ in the
fourth and fifth, and $V$ stands in each case for the standard
representation of $A$. In the last example $p$ is an odd prime. 

In each of the examples given, the groups $\Gamma$ are known to
satisfy classical homological stability.  In contrast, the rightmost
column shows that none of the groups $N$ satisfies homological
stability, even in dimension 1.  In fact, except for the first
example, almost nothing is known about the $A$--module $H_i(N,R)$ for
$i>1$, and indeed it is not clear if there is a nice closed form
description of these homology groups.  However, the appearance of some
kind of ``stability'' can already be seen in the rightmost column, as
the names of the irreducible composition factors of these $A$--modules are
constant for large enough $n$; this is discussed in detail for each
example in the body of the paper.

A crucial observation for us is that each of the groups $A$ in the
table above has an inherent stability in the naming of its irreducible algebraic
representations (over $R$).  For example, an irreducible algebraic
representation of $\SL_n$ is determined by its highest weight vector,
and these vectors may be described uniformly without reference to
$n$. For example, for $\SL_n$ the irreducible representation
$V(L_1+L_2+L_3)$ with highest weight $L_1+L_2+L_3$ is isomorphic to
$\bigwedge^3 V$ regardless of $n$, where $V$ is the standard
representation of $\SL_n$ (see Section~\ref{section:repsGLSp} for the
representation theory of $\SL_n$).  This inherent stability can be
used, at least conjecturally, to give a closed form description for
$H_i(N,R)$ (for $n$ large enough, depending on $i$).  One idea is that
the growth in $\dim_R(H_i(N,R))$ should be fully accounted for by the
fact that each element of $H_i(N,R)$ brings along with it an entire
$A$--orbit.

\para{Homology of Lie algebras}  
In \S\ref{section:liealg} we develop representation stability for Lie algebras and their homology.  
The main theoretical result here,  Theorem~\ref{thm:equivhomLie}, proves the equivalence between stability for a family of Lie algebras and stability for its
  homology.  Both directions of this implication are applied to
  give nontrivial results. For example, in Corollary~\ref{corollary:nilp} we deduce stability for the homology of nilpotent Lie algebras, which is quite complicated, from stability for the homology of free Lie algebras, which is trivial to compute; the proof uses both directions of Theorem~\ref{thm:equivhomLie} in an essential way. We also give applications to the adjoint homology of free nilpotent Lie algebras   (Corollary \ref{cor:adjnil}) and the homology of Heisenberg Lie algebras (Examples~\ref{example:Heis} and \ref{example:adjHeis}).   
  
Although homological stability results for lattices in semisimple Lie groups has been known for some time, we emphasize that  there do not seem to have been any stability results for the homology of lattices in nilpotent Lie groups. Since Nomizu proved in 1954 that  the rational homology of a lattice in a nilpotent Lie group $N$ is isomorphic to the Lie algebra homology of the Lie algebra of $N$, such homological stability results follow from our theorems on nilpotent Lie algebras.

\para{The ubiquity of representation stability}
The phenomenon of representation stability occurs in a number of
different places in mathematics.   The majority of this paper is devoted to 
exposing this phenomenon through examples.  In doing this we obtain applications, theorems and conjectures. The examples include:

\begin{enumerate}

\item Classical representation theory
  (\S\ref{section:classicalstability}): stability of Schur functors;
  Littlewood--Richardson rule; Murnaghan's theorem on stability of
  Kronecker coefficients; other natural constructions.  These
  constructions arise in most other examples, and so their stability
  underlies the whole theory of representation stability.
  
\item Cohomology of moduli spaces (\S\ref{section:braids}, \S\ref{section:torelli}):  pure braid groups and generalized pure braid groups; conjecturally in the Torelli subgroups 
of mapping class groups $\Mod(S)$ and the analogue for automorphism groups $\Aut(F_n)$ of free groups.  
We prove representation stability for the homology of pure
  braid groups in Theorem~\ref{thm:pure} and pure generalized braid groups in Theorem~\ref{thm:genpure}.  In \S\ref{section:torelli} we give a number of conjectures about the stable homology of the Torelli groups and their analogues.  Previously there had been few (if any) general suggestions in this direction.

\item Lie algebras (\S\ref{section:liealg}): graded components of free 
Lie algebras; homology of various families of Lie algebras, for
  example free nilpotent Lie algebras and Heisenberg Lie algebras; Malcev
  completions of surface groups and (conjecturally) pure braid groups.  As discussed in the introduction, the main tool proved is Theorem~\ref{thm:equivhomLie}.  We apply it to prove 
  representation stability for various nilpotent Lie algebras.
  
\item (Equivariant) cohomology of flag and Schubert varieties
  (\S\ref{section:flags}).  As explained in \S\ref{section:flags}, the space $\cF_n$ of complete flags in $\C^n$ admits a nontrivial action of $S_n$, and the resulting representation on $H^i(\cF_n;\Q)$ is rather complicated.  Similarly, the hyperoctahedral group acts on the space $\cF'_n$ of complete flags on Lagrangian subspaces of $\C^{2n}$.   
  For each $i\geq 1$ the natural families $\{H^i({\cF_n};\Q)\}$ and 
  $\{H^i({\cF'_n};\Q))\}$ do not satisfy classical (co)homological stability.  However, we prove in each case (see Theorem \ref{thm:flag} and Theorem \ref{thm:flagsp}) that these sequences are representation stable.     
  
  Another class of well-studied families of varieties are the \emph{Schubert varieties}.  Each permutation $w$ of any finite set determines a family $\{X_w[n]\}$ of Schubert varieties 
  (see \S\ref{section:schubert}).  For each $i\geq 0$ the ($T$--equivariant, for a certain torus $T$) cohomology $H_T^i(X_w[n];\Q)$ admits a non-obvious action by $S_n$.  While the sequence 
  $\{H_T^i(X_w[n];\Q)\}$ does not satisfy homological stability in the classical sense, we prove in Theorem \ref{theorem:schubert} that this sequence is representation stable.

\item Algebraic combinatorics (\S\ref{section:flags}): Lefschetz
  representations associated to rank-selected posets and
  cross-polytopes.  Here Stanley's counts of multiplicities in
  terms of various Young tableaux are shown to give representation
  stability.  We also conjecture representation stability for the bigraded pieces of the diagonal
  coinvariant algebras.  This gives an ``asymptotic refinement'' of the famous 
  $(n+1)^{(n-1)}$ conjecture in algebraic combinatorics (proved by Haiman), as well as conjectures for ``higher coinvariant algebras'', where very little is known.

\item Homology of congruence subgroups and modular representations.  In 
\S\ref{section:congruence} we study congruence subgroups of certain arithmetic groups and their analogues for $\Mod(S)$ and for $\Aut(F_n)$.  Each of these groups $\Gamma$ admits an 
action by outer automorphisms by a finite group $G$ of Lie type, such as $G=\SL_n(\F_p)$.   This action makes each homology vector space $H_i(\Gamma,\F_p)$ a $G$--representation.    
As $p$ divides the order of $G$, this is a modular representation.  Thus, in order to obtain results and conjectures about these important representations, we need to develop a 
version of our theory using modular representation theory.  
Here a new phenomenon  occurs: \emph{stable periodicity} of a sequence of representations
  (see \S\ref{section:congruence}).  
  
  For each of the sequences of groups $\Gamma$ above we 
  state a ``stable periodicity conjecture'' for its homology with $\F_p$ coefficients.   The 
  few computations that have been completely worked out are almost all in degree $1$, and these 
  use deep mathematics (e.g.\ the congruence subgroup problem, work of Johnson, etc.).   These computations show that our conjectures are satisfied for $H_1$.  {See \S\ref{section:congruence} for details.}
  \end{enumerate}

\para{Historical notes} Various stability phenomena have been
known in representation theory at least as far back as the 1930s, when
formulas were given for the decomposition of tensor products of
irreducible representations of $\SL_n\Q$ (by Littlewood--Richardson,
see e.g.\ \cite[Appendix A]{FH}) and of the symmetric group $S_n$ (by
Murnaghan \cite{Mu}). Some aspects of representation stability can be
found in previous work on Lie algebras.  Related ideas appear in terms
of mixed tensor representations in the work of Hanlon \cite{Han} and
R.~Brylinski \cite{Bry} on Lie algebra cohomology of non-unital
algebras; in Tirao's description in \cite{Ti} of the homology of free
nilpotent Lie algebras; and in Hain's description in \cite{Ha} of the
associated graded Lie algebra of the fundamental group of a closed
surface. In an unpublished 1991 manuscript, Hain conjectured a phenomenon quite close to representation stability for the action of $\Sp_{2g}\Z$ on $H^i(\I_{g,1};\Q)$, as $g$ varies and $i\geq 0$ is fixed.

\para{Recent work} The first version of the present paper was posted in August 2010.  Since that time, a number of papers have appeared that relate to this paper or build on its ideas.  Representation stability has been shown to hold for certain polynomial algebras \cite{AAB}, the second homology of the Torelli group \cite{BD} (partially verifying Conjecture~\ref{conjecture:torelli} below), the homology of moduli spaces of $n$-pointed Riemann surfaces \cite{J1,J2}, the cohomology of certain hyperplane arrangements \cite{Mo}, syzygies of line bundles on Segre--Veronese varieties \cite{Ra}, the cohomology of pure string motion groups \cite{Wi1} (verifying Conjecture~\ref{conjecture:string} below), and 
the cohomology of configuration spaces of $n$ ordered points on any manifold \cite{C}.  

In \cite{Pu2} Putman devised for symmetric groups a similar theory of ``central stability'' and applied it to the homology of congruence subgroups.  Related applications have also been given in \cite{CE} and \cite{CEFN}. 

In joint work with Ellenberg \cite{CEF1}, we expanded the theory of representation stability for $S_n$-representations   via the study of ``FI-modules''.  This theory allows us to prove representation stability for many more examples, including diagonal co-invariant algebras (answering Question~\ref{question:coinv} below) and the Malcev Lie algebras of the pure braid groups (proving Conjecture~\ref{conjecture:malcevpn} below). It also has applications outside the theory of representation stability. Most notably, we prove that the characters of representation stable 
sequences are ``polynomial'' for large $n$; in particular, the dimension of a representation stable sequence is eventually polynomial. We also use FI-modules to reprove and extend (e.g.\ to finite field coefficients) the representation stability of the cohomology of configuration spaces originally proved in \cite{C}. In further joint work with Ellenberg and Nagpal~\cite{CEFN}, we showed that FI-modules make it possible to extend the theory to integral coefficients, where the methods of the current paper are not applicable. The theory of representation stability and FI-modules has been extended by Wilson~\cite{Wi2} to other sequences of Weyl groups.

FI-modules in turn can be considered as a special case of much more general structures called ``twisted commutative algebras'', a theory developed in \cite{Sn,SS1,SS2,SS3} by Sam and Snowden as a different approach to ``stable representation theory''.

\para{Acknowledgements} We thank Jon Alperin, Jordan
Ellenberg, Victor Ginzburg, David Hemmer, Rita Jimenez Rolland, Brian Parshall, Andy Putman, Steven Sam, 
Richard Stanley, and Paulo Tirao for helpful discussions.  We are grateful to Jenny Wilson for useful comments, and in particular for pointing out a simplification to the proof of Theorem~\ref{thm:genpure}.

\section{Representation stability}
\label{section:representationstability}

In order to define representation stability and its variants, we will
need to be precise in the labeling of the irreducible
representations of the various groups we consider. We begin by
reviewing the representation theory of the following families of
groups in order to establish uniform notation across the different
families.

In this section $G_n$ will always denote one of the following families of groups:
\begin{itemize}
\item $G_n=\SL_n\Q$, the special linear group.
\item $G_n=\GL_n\Q$, the general linear group.
\item $G_n=\Sp_{2n}\Q$, the symplectic group.
\item $G_n=S_n$, the symmetric group.
\item $G_n=W_n$, the hyperoctahedral group.
\end{itemize}

By a \emph{representation} of a group $G$ we mean a $\Q$--vector space
equipped with a linear action of $G$.  With the exception of Section~\ref{section:congruence}, throughout this paper we work
over $\Q$, but the definitions and results hold over any field of
characteristic $0$, in particular over $\C$.  In Section~\ref{section:congruence} we will extend the definition of representation stability to modular representations of $\SL_n\F_p$ and
$\Sp_{2n}\F_p$. 

\subsection{Symmetric and hyperoctahedral groups}
\label{section:reps1}

Our basic reference for representation theory is Fulton--Harris
\cite{FH}.  For hyperoctahedral groups, see Geck--Pfeiffer \cite[\S
1.4 and \S 5.5]{GP}.

\para{Symmetric groups} The irreducible representations of $S_n$ are
classified by the partitions $\lambda$ of $n$. A \emph{partition} of
$n$ is a sequence $\lambda=(\lambda_1\geq \cdots \geq \lambda_\ell\geq0)$
with $\lambda_1+\cdots+\lambda_\ell=n$; we write $|\lambda|=n$ or
$\lambda\vdash n$. These partitions are identified with
Young diagrams, where the diagram corresponding to $\lambda$ has
$\lambda_i$ boxes in the $i$th row. We identify partitions if their nonzero entries coincide; every partition then can be uniquely written with $\lambda_\ell>0$, in which case we say that $\ell=\ell(\lambda)$ is the \emph{length} of $\lambda$. The irreducible representation
corresponding to the partition $\lambda$ is denoted $V_\lambda$. This
irreducible representation can be obtained as the image $\Q S_n \cdot
c_\lambda$ of a certain idempotent $c_\lambda$ in the group algebra
$\Q S_n$. The fact that every irreducible representation of $S_n$ is
defined over $\Q$ implies that any $S_n$--representation defined over $\Q$
decomposes over $\Q$ into irreducibles. Since every representation of $S_n$ is defined over $\Q$, or alternately since $g$ is conjugate to $g^{-1}$ for all $g\in S_n$, every representation of $S_n$ is self-dual.

For example, the irreducible $V_{(n-1,1)}$ is the standard
representation of $S_n$ on $\Q^n/\Q$. The representation
$\bwedge^3 V_{(n-1,1)}$ is the irreducible representation
$V_{(n-3,1,1,1)}$. To remove the dependence of this notation on $n$,
we make the following definition. If $\lambda=(\lambda_1,\ldots,\lambda_\ell)\vdash k$ is any
partition of a fixed number $k$, then for any $n\geq
k+\lambda_1$ we may define the \emph{padded partition}
\[\lambda[n]=(n-k,\lambda_1,\ldots,\lambda_\ell).\]
The condition $n\geq k+\lambda_1$ is needed so that this sequence is
nonincreasing and defines a partition. For $n\geq k+\lambda_1$ we then
define $V(\lambda)_n$ to be the irreducible $S_n$--representation
\[V(\lambda)_n=V_{\lambda[n]}.\] Every irreducible representation of
$S_n$ is of the form $V(\lambda)_n$ for a unique partition
$\lambda$. When unambiguous, we denote this representation simply by
$V(\lambda)$. In this notation, the standard representation is $V(1)$,
and the identity $\bwedge^3 V(1)=V(1,1,1)$ holds whenever both sides
are defined.

\para{Hyperoctahedral groups} The hyperoctahedral group $W_n$ is the
wreath product $\Z/2\Z\wr S_n$; that is, the semidirect product
$(\Z/2\Z)^n\rtimes S_n$ where the action is by permutations.  $W_n$ can also
be thought of as the group of signed permutation matrices. General
analysis of wreath products shows that the irreducible representations
of $W_n$ are classified by \emph{double partitions}
$(\lambda^+,\lambda^-)$ of $n$, meaning that
$|\lambda^+|+|\lambda^-|=n$. Given any representation $V$ of $S_n$, we
may regard $V$ as a representation of $W_n$ by pullback.  The
irreducible representation $V_\lambda$ of $S_n$ yields the irreducible
representation $V_{(\lambda,0)}$ of $W_n$. Let $\nu$ be the
one-dimensional representation of $W_n$ which is trivial on $S_n$,
while each $\Z/2\Z$ factor acts by $-1$.  Then $V_{(\lambda,0)}\otimes
\nu=V_{(0,\lambda)}$. In general, if $\lambda^+\vdash k$ and
$\lambda^-\vdash n-k$, the irreducible $V_{(\lambda^+,\lambda^-)}$ is
obtained as the induced representation
\[V_{(\lambda^+,\lambda^-)}=\Ind_{W_k\times W_{n-k}}^{W_n}
V_{(\lambda^+,0)}\boxtimes V_{(0,\lambda^-)}\]
 where $V_{(\lambda^+,0)}\boxtimes V_{(0,\lambda^-)}$ denotes the vector space $V_{(\lambda^+,0)}\otimes V_{(0,\lambda^-)}$ considered as a representation of $W_k\times W_{n-k}$.  
 
 As before, for an arbitrary double partition
$\lambda=(\lambda^+,\lambda^-)$ with $|\lambda^+|+|\lambda^-|=k$, for
$n\geq k+\lambda^+_1$ we define the padded partition
\[\lambda[n]=((n-k,\lambda^+),\lambda^-),\]
and define $V(\lambda)_n$ to be the irreducible $W_n$--representation
\[V(\lambda)_n=V_{\lambda[n]}.\] Every irreducible representation of
$W_n$ is of the form $V(\lambda)_n$ for a unique double partition
$\lambda$.

\subsection{The algebraic groups $\SL_n$, $\GL_n$ and $\Sp_{2n}$}
\label{section:repsGLSp}

In this subsection we recall the representation theory of the
algebraic groups $\SL_n$, $\GL_n$ and $\Sp_{2n}$.  

\para{Special linear groups} We first review the representation theory
of $\SL_n\Q$ and $\GL_n\Q$. There is an interplay between two
important perspectives here, that of highest weight vectors and that
of Schur functors.

Every representation of $\SL_n\Q$ induces a representation of the Lie
algebra $\fsl_n\Q$. Fixing a basis gives a triangular decomposition
$\fsl_n\Q=\fn^-\oplus \fh\oplus \fn^+$, consisting of strictly lower
triangular, diagonal, and strictly upper triangular matrices
respectively. Given a representation $V$ of $\fsl_n\Q$, a
\emph{highest weight vector} is a vector $v\in V$ which is an
eigenvector for $\fh$ and is annihilated by $\fn^+$. Every irreducible
representation contains a unique highest weight vector and is
determined by the corresponding eigenvalue in $\fh^*$, called a
\emph{weight}.

Considering the obvious basis for the diagonal matrices, we obtain
dual functionals $L_i$; this
yields \[\fh^*=\Q[L_1,\ldots,L_n]/(L_1+\cdots+L_n=0).\] Every weight
lies in the \emph{weight lattice}
\[\Lambda_W=\Z[L_1,\ldots,L_n]/(L_1+\cdots+L_n).\]

The \emph{fundamental weights} are $\omega_i=L_1+\cdots+L_i$. A
\emph{dominant weight} is a weight that can be written as a
nonnegative integral combination $\sum c_i\omega_i$ of the fundamental
weights. A highest weight vector always has a dominant weight as its
eigenvalue, and every dominant weight is the highest weight of a
unique irreducible representation. If $\lambda=\sum c_i\omega_i$ is a
dominant weight, we denote by $V(\lambda)_n$ the irreducible
representation of $\SL_n\Q$ with highest weight $\lambda$. These
representations remain distinct and irreducible when restricted to
$\SL_n\Z$.

We now give another labeling of the irreducible $\SL_n\Q$--representations.  Let $\lambda=(\lambda_1\geq \cdots \geq \lambda_\ell)$ be a
partition of $d$.  Each such partition determines a \emph{Schur
  functor} $\Schur_\lambda$ which attaches to any vector space $V$ the
vector space
\[\Schur_\lambda(V)=V^{\otimes d}\otimes_{\Q S_d}V_\lambda,\]
where $\Q S_d$ acts on $V^{\otimes d}$ by permuting the factors. If
$\dim V$ is less than the number of rows $\ell(\lambda)$ of $\lambda$, then 
$\Schur_\lambda(V)$ is the zero representation.  If $V$ is a representation of a
group $G$, the induced action makes $\Schur_\lambda(V)$ a
representation of $G$ as well.

Consider the standard representation of $\SL_n\Q$ on $\Q^n$. For any
partition $\lambda=(\lambda_1\geq\cdots\geq \lambda_n\geq 0)$ with at
most $n$ rows, the resulting representation $\Schur_\lambda(\Q^n)$ is 
isomorphic to $V(\lambda_1L_1+\cdots+\lambda_nL_n)_n$ as
$\SL_n\Q$--representations. In particular, $\Schur_\lambda(\Q^n)$ is
irreducible, and all irreducible representations arise this way;  see
\cite[\S 6 and \S 15.3]{FH}.  For example, let $V=\Q^n$.  When $\lambda=d$ is the trivial partition
then $\Schur_\lambda(V)=\Sym^dV$, and when $\lambda=1+\cdots +1$ then
$\Schur_\lambda(V)=\bwedge^dV$. Note that since $L_1+\cdots+L_n=0$,
two partitions $\lambda$ and $\mu$ determine the same
$\SL_n\Q$--representation if and only if $\lambda_i-\mu_i$ is constant
for all $1\leq i\leq n$. Thus we may always take our partitions to
have $\lambda_n=0$ (see the ``important notational convention'' remark below).

If $\lambda=(\lambda_1\geq\ldots \geq \lambda_k\geq 0)$ is a partition
with $k$ rows, then for any $n>k$ we define
\begin{equation}
  \label{eq:Vlambda}
  V(\lambda)_{n}\coloneq \Schur_{\lambda}(\Q^n).
\end{equation}
With this convention, every irreducible representation of $\SL_n\Q$ is
of the form $V(\lambda)_n$ for a unique partition $\lambda$.  As
before, we will sometimes refer to $V(\lambda)_n$ as $V(\lambda)$ when
the dimension is clear from context. Note that with this terminology 
$V(3,1)_n$ has the same meaning as $V(3,1,0,0)_n$.

\para{Important notational convention} The right side of
\eqref{eq:Vlambda} makes sense even when $n=k$ or $n<k$. However, we
intentionally decline to define $V(\lambda)_n$ when $n\leq k$. The
reason is that as noted above, $V(\lambda_1,\ldots,\lambda_n)$
coincides with
$V(\lambda_1-\lambda_n,\ldots,\lambda_n-\lambda_n)$. This coincidence
causes confusion with intuitive expectations about
multiplicity. For example, we would expect that the multiplicity of
the irreducible $\SL_n\Q$--representation $\Schur_{(2,2,2,2)}(\Q^n)$
in the trivial representation $\Q$ is 0, and this is in fact true for all
$n>4$. However, when $n=4$ we have $\Schur_{(2,2,2,2)}(\Q^4)=\Q$, and
so the multiplicity in this case is 1. For $n<4$ the representation
$\Schur_{(2,2,2,2)}(\Q^n)$ is the zero representation, so the
multiplicity is not well-defined. Another benefit of this convention
is the important fact that every irreducible representation of
$\SL_n\Q$ is of the form $V(\lambda)_n$ for a \emph{unique}
$\lambda$. This notational convention is equivalent to requiring all
partitions to have $\lambda_n=0$, as mentioned above.

\para{General linear groups} Consider the standard representation of
$\GL_n\Q$ on $\Q^n$. If $\lambda=(\lambda_1\geq\cdots\geq
\lambda_n\geq 0)$ is a partition with at most $n$ rows, then
$V(\lambda)_n=\Schur_\lambda(\Q^n)$ is an irreducible representation
of $\GL_n\Q$. The partition $(1,\ldots,1)$ with $n$ rows yields the
representation $V(1,\ldots,1)_n=\bwedge^n\Q^n=D$, the one-dimensional
determinant representation of $\GL_n\Q$, and in general for any
positive $k$ we
have \[\Schur_{(\lambda_1+k,\ldots,\lambda_n+k)}(\Q^n)=
\Schur_{(\lambda_1,\ldots,\lambda_n)}(\Q^n) \otimes D^{\otimes k}.\] A \emph{pseudo-partition} is a sequence
$\lambda=(\lambda_1\geq \cdots\geq \lambda_\ell)$, where the integers
$\lambda_i$ are allowed to be negative. The length $\ell(\lambda)$ of a pseudo-partition is the largest $i$ for which $\lambda_i\neq 0$. We extend the definition of
$V(\lambda)$ to pseudo-partitions by the above formula. That is, for
any pseudo-partition $\lambda=(\lambda_1\geq \cdots\geq \lambda_k)$
and any $n\geq k$, we define
\[V(\lambda_1,\ldots,\lambda_k)_n\coloneq
\Schur_{(\lambda_1-\lambda_k,\ldots,\lambda_k-\lambda_k)}(\Q^n)\otimes
D^{\otimes\lambda_k}.\] Every irreducible representation of $\GL_n\Q$
is of the form $V(\lambda)_n$ for a unique pseudo-partition
$\lambda$. For example, the dual of $V(\lambda_1,\ldots,\lambda_n)$ is
$V(-\lambda_n,\ldots,-\lambda_1)$.  As before, the obvious basis for
the diagonal matrices yields dual functionals $L_i$. For a
pseudo-partition $\lambda=(\lambda_1\geq\cdots\geq\lambda_k)$, the
irreducible representation $V(\lambda)_n$ has heighest weight
$\lambda_1L_1+\cdots+\lambda_kL_k$. When restricted to $\GL_n\Z$, all
of these representations remain irreducible, and $D^{\otimes 2}$
becomes trivial; thus two pseudo-partitions $\lambda$ and $\mu$
determine the same representation of $\GL_n\Z$ if and only if
$\lambda_i-\mu_i$ is constant and even for all $1\leq i\leq n$.

\begin{remark}
 Note that for $\GL_n\Q$--representations, $V(3,1)_n$ has the same meaning as
 $V(3,1,1,1)_n$, while $V(3,1,0)_n$ has the same meaning as
 $V(3,1,0,0)_n$. The discrepancy between the terminology for
 representations of $\SL_n\Q$ and $\GL_n\Q$ comes from the fact that
 for $\SL_n\Q$ we always assume that $\lambda_n=0$.
\end{remark}

\para{Symplectic groups} We now review the representation theory of
$\Sp_{2n}\Q$. Every representation of $\Sp_{2n}\Q$ induces a
representation of the Lie algebra $\fsp_{2n}\Q$. Again we have a
decomposition $\fsp_{2n}\Q=\fn^-\oplus \fh\oplus \fn^+$, with
$\fh^*=\Q[L_1,\ldots,L_n]$. The fundamental weights are
$\omega_i=L_1+\cdots+L_i$, and so for any dominant weight
$\lambda=\sum c_i\omega_i$ there is a unique irreducible
representation $V(\lambda)_n$ of $\Sp_{2n}\Q$.

These can be identified explicitly as follows. Let $V=\Q^{2n}$ be the
standard representation of $\Sp_{2n}\Q$. For each $1\leq i<j\leq d$,
the symplectic form gives a contraction $V^{\otimes d}\to V^{\otimes
  d-2}$ as $\Sp_{2n}\Q$--modules. Define $V^{\langle d\rangle}\leq
V^{\otimes d}$ to be the intersection of the kernels of these
contractions. For any partition $\lambda\vdash d$, the representation
$\Schur_\lambda V$ is realized as the image of $c_\lambda\in \Q S_d$
acting on $V^{\otimes d}$. If $k$ is the number of rows of the
partition $\lambda$, for any $n\geq k$ we define $V(\lambda)_n$ to be
the intersection
\[V(\lambda)_n\coloneq\Schur_\lambda V\cap V^{\langle d\rangle}.\] The notation
$\Schur_{\langle\lambda\rangle}V$ is also used for the intersection
$\Schur_\lambda V\cap V^{\langle d\rangle}$. We
remark that this intersection is trivial if $n$ is less than the
number of rows of $\lambda$. Every irreducible representation of
$\Sp_{2n}\Q$ is of the form $V(\lambda)_n$ for a unique partition
$\lambda$. In particular, it follows that each irreducible
representation $V(\lambda)_n$ is self-dual. These representations
remain distinct and irreducible when restricted to $\Sp_{2n}\Z$.

\begin{remark}\label{rem:Spweights}
  There is one issue which can cause confusion when comparing weights
  for $\GL_{2n}\Q$ and $\Sp_{2n}\Q$. To clarify, we work out the
  comparison explicitly in terms of a basis. Let
  $\{a_1,b_1,\ldots,a_n,b_n\}$ be a symplectic basis for $\Q^{2n}$,
  meaning that the symplectic form satisfies $\omega(a_i,b_i)=1$ and
  $\omega(a_i,b_j)=\omega(a_i,a_j)=\omega(b_i,b_j)=0$. By abuse of
  notation, in this remark we also denote by
  $\{a_1,b_1,\ldots,a_n,b_n\}$ the corresponding basis for $\fh_\fgl$,
  the diagonal matrices in $\fgl_{2n}\Q$, with dual basis
  $\{a_1^*,b_1^*,\ldots,a_n^*,b_n^*\}$ for $\fh^*_\fgl$. These
  elements, in some order, will be the weights
  $\{L^{\fgl}_1,\ldots,L^{\fgl}_{2n}\}$, but we defer until later the
  explicit identification.

  If $\fh_\fsp$ denotes the diagonal matrices in $\fsp_{2n}\Q$, the
  dual $\fh^*_\fsp$ has basis $\{L^{\fsp}_i=a_i^*-b_i^*\}$. These
  weights are ordered so that $L^{\fsp}_1>\cdots>L^{\fsp}_n$. The
  restriction from $\fh^*_\fgl$ to $\fh^*_\fsp$ maps $a_i^*\mapsto
  L^{\fsp}_i$ and $b_i^*\mapsto -L^{\fsp}_i$. To correctly compare
  representations of $\GL_{2n}\Q$ with those of $\Sp_{2n}\Q$, this
  restriction should preserve the ordering on weights (for example, so
  that the notions of ``highest weight'' agree). This forces us to
  label the weights of $\GL_{2n}\Q$ as \[L^{\fgl}_1=a^*_1,\ \ldots,\
  L^{\fgl}_n=a^*_n,\ \ L^{\fgl}_{n+1}=b^*_n,\ \ldots,\
  L^{\fgl}_{2n}=b^*_1.\] Thus the restriction maps $L^{\fgl}_i\mapsto
  L^{\fsp}_i$ and $L^{\fgl}_{2n-i+1}\mapsto -L^{\fsp}_i$ for $1\leq
  i\leq n$. With respect to the ordered basis
  $\{a_1,\ldots,a_n,b_n,\ldots,b_1\}$, the subalgebra $\fn^+$ consists
  of exactly those matrices in $\fsp_{2n}\Q$ that are
  upper-triangular.
\end{remark}

\subsection{Definition of representation stability}
\label{section:repstab:def}

We are now ready to define the main concept of this paper.  Let $G_n$ be one of the families $\GL_n\Q$, $\SL_n\Q$, $\Sp_{2n}\Q$,
$S_n$, or $W_n$. In this section $\lambda$ refers to the datum
determining the irreducible representations of the corresponding
family, namely a pseudo-partition, a partition, or a double
partition. For each family we have natural inclusions
$G_n\hookrightarrow G_{n+1}$: for $S_n$ and $W_n$ we take the standard
inclusions, and for $\GL_n\Q$, $\SL_n\Q$, and $\Sp_{2n}\Q$ we take the
upper-left inclusions.

Let $\{V_n\}$ be a sequence of $G_n$--representations, equipped with
linear maps $\phi_n\colon V_n\to V_{n+1}$, making the following diagram
commute for each $g\in G_n$:
\[\xymatrix{
  V_n\ar^{\phi_n}[r]\ar_{g}[d]&V_{n+1}\ar^{g}[d]\\ V_n\ar_{\phi_n}[r]&V_{n+1}
}\]
On the right side we consider $g$ as an element of $G_{n+1}$ by the
inclusion $G_n\hookrightarrow G_{n+1}$.  This condition is equivalent 
to saying that $\phi_n$, thought of as a map  from $V_n$ to the restriction $V_{n+1}\downarrow{G_n}$, is a map of $G_n$--representations.  
We allow the vector spaces
$V_n$ to be infinite-dimensional, but we ask that each vector lies in
some finite-dimensional representation. This ensures that $V_n$
decomposes as a direct sum of finite-dimensional irreducibles. We call
such a sequence of representations \emph{consistent}.   

We want to compare the representations $V_n$ as $n$ varies.  However,
since $V_n$ and $V_{n+1}$ are representations of different groups, we
cannot ask for an isomorphism as representations. But we can ask for
injectivity and surjectivity, once they are properly
formulated. Moreover, using the uniformity of our labeling of irreducible representations,  we
can formulate what it means for $V_n$ and $V_{n+1}$ to be the ``same
representation''.

\begin{definition}[Representation stability] 
\label{definition:repstab1}
Let $\{V_n\}$ be a consistent sequence of $G_n$--representations.
The sequence $\{V_n\}$ is \emph{representation stable} if, for
sufficiently large $n$, each of the following conditions holds.
\begin{enumerate}[{\bf I.}]
\item \textbf{Injectivity:} The natural map $\phi_n\colon V_n\to
  V_{n+1}$ is injective.
\item \textbf{Surjectivity:} The span of the $G_{n+1}$--orbit of
  $\phi_n(V_n)$ equals all of $V_{n+1}$.
\item \textbf{Multiplicities:} Decompose $V_n$ into irreducible
  representations as
  \[V_n=\bigoplus_\lambda c_{\lambda,n}V(\lambda)_n\] with
  multiplicities $0\leq c_{\lambda,n}\leq \infty$. For each $\lambda$,
  the multiplicities $c_{\lambda,n}$ are eventually independent of
  $n$.
\end{enumerate}
\end{definition}

It is not hard to check that, given Condition I for $\phi_n$, Condition II for $\phi_n$ is equivalent to the following when $G_n$ is finite: $\phi_n$ is a composition of the inclusion $V_n\hookrightarrow \Ind_{G_n}^{G_{n+1}}V_n$ with a surjective $G_{n+1}$--module homomorphism $ \Ind_{G_n}^{G_{n+1}}V_n\to V_{n+1}$. 

By requiring Condition III just for the multiplicity of the single
irreducible representation $V(\lambda)_n$, we obtain the notion of
\emph{$\lambda$--representation stable} for a fixed $\lambda$.  In the
presence of Condition IV below, $\lambda$--representation stability is
exactly equivalent to representation stability for the
$\lambda$--isotypic components $V_n^{(\lambda)}$.

\begin{remark}
  Fix either $G_n=S_n$ or $W_n$ and take for each $n\geq 1$ an exact
  sequence of groups
  \[1\to A_n\to \Gamma_n\to G_n\to 1.\] Then an easy transfer argument
  shows that $\lambda$--representation stability of $\{H_i(A_n;\Q)\}$
  for the trivial representation ($\lambda=0$) is equivalent to
  classical homological stability for the sequence
  $\{H_i(\Gamma_n;\Q)\}$.
\end{remark}

\begin{remark}
It seems likely that many of the results in this paper can be extended to orthogonal groups and to the corresponding Weyl groups; it would be interesting to know what differences arise, if any.
\end{remark}

\para{Uniform stability} In Definition~\ref{definition:repstab1} we did not require the
multiplicities of all the irreducible representations to stabilize
simultaneously.  We will see that in many cases a stronger form of
stability holds, as follows.

\begin{definition}[Uniform representation stability] 
A consistent sequence $\{V_n\}$ of $G_n$--representations is
\emph{uniformly representation stable} if Conditions I and II hold for
sufficiently large $n$, and the following condition holds:
\begin{enumerate}
\item[{\bf III$'$.}] \textbf{Multiplicities (uniform):} There is some
  $N$, not depending on $\lambda$, so that for $n\geq N$ the
  multiplicities $c_{\lambda,n}$ are independent of $n$ for all
  $\lambda$. In particular, for any $\lambda$ for which $V(\lambda)_N$
  is not defined, $c_{\lambda,n}=0$ for all $n\geq N$.
\end{enumerate}
\end{definition}
For example, if $G_n=\GL_n\Q$, the latter condition applies to any
partition $\lambda$ with more than $N$ rows. We will see examples
below both of uniform and nonuniform representation stability.

\para{Multiplicity stability}
It sometimes happens that for a sequence $\{V_n\}$ of $G_n$--representations
there are no natural maps $V_n\to V_{n+1}$. For example, this is the
situation for the Torelli groups of closed surfaces (see
Section~\ref{section:torelli} below). In this case we can still ask
whether the decomposition of $V_n$ into irreducibles stabilizes in
terms of multiplicities.
\begin{definition}[(Uniform) multiplicity stability]
  A sequence of $G_n$--representations $V_n$ is called
  \emph{multiplicity stable} (respectively \emph{uniformly
    multiplicitly stable}) if Condition III (respectively Condition
  III$'$) holds.
\end{definition}

\para{Reversed maps} 
The definitions above capture the behavior of a sequence of representations, one including into the next.  In a number of contexts (see, e.g., \S~\ref{section:flags} below) we are given sequences of representations with the maps going the other way, from $V_{n+1}\to V_n$.  In this case we need to alter the definition of representation stability, in particular injectivity and surjectivity.


\begin{definition}[Representation stability with maps reversed]
\label{definition:repstabrev}
  A consistent sequence of $G_n$--representations $\{V_n\}$ with maps
  $\phi_n\colon V_n\leftarrow V_{n+1}$ is \emph{representation stable}
  if for sufficiently large $n$, Condition III holds, and the
  following conditions hold:
\begin{enumerate}[I$'$.]
\item \textbf{Surjectivity:} The map $\phi_n\colon V_n\leftarrow
  V_{n+1}$ is surjective.
\item \textbf{Injectivity:} There exists a subspace $V_{n+1}$ which
  maps isomorphically under $\phi_n$ to $V_n$, and whose
  $G_{n+1}$--orbit spans $V_{n+1}$.
\end{enumerate}
\end{definition}

We remark that Definition~\ref{definition:repstabrev} is \emph{not} equivalent to representation stability for the dual sequence $\phi_n^\ast:V_n^\ast\to V_{n+1}^\ast$.  For an explanation, and a way to handle dual sequences, see the discussion of  ``mixed tensor stability'' in \S\ref{section:strong}.

\para{Examples of representation stability} We will see many examples
of representation stability below; indeed much of this paper is an
exploration of such examples.  For now, we mention the following
simple examples.

\begin{xample}
  \label{example:tensor}
  Let $V_n=\Q^n$ be the standard representation of $\GL_n\Q$.  Then
  $\{V_n\otimes V_n\}$ is uniformly representation stable.  This
  follows easily from the decomposition \[V_n\otimes
  V_n=\Sym^2V_n\oplus \bwedge^2V_n\] of $V_n\otimes V_n$ into
  irreducibles. We will see in
  Section~\ref{section:classicalstability} that $\{V_n\otimes V_n\}$
  will be uniformly representation stable for any sequence $\{V_n\}$ of
  uniformly stable representations.
\end{xample}

In the other direction, we have the following.
\begin{nonexample}
  Let $V_n=\Q^n$ be the standard representation of $\SL_n\Q$, and let
  $W_n=\bwedge^* V_n$.  Then $\{W_n\}$ is a stable sequence of
  $\SL_n\Q$--representations, but it is not a uniformly stable
  sequence:  the multiplicity of the irreducible
  representation $V(1,\ldots ,1)_n$, with $k$ occurrences of
  $1$, does not stabilize until $n>k$.
\end{nonexample}

\begin{nonexample}
  \label{nonexample:regular}
  Let $G_n$ be either $S_n$ or $W_n$. Then the sequence of regular representations $\{\Q G_n\}$ is not representation stable, or even
  $\lambda$--representation stable for any partition or double
  partition $\lambda$.  This follows from the standard
  fact that the multiplicity of $V(\lambda)_n$ in the regular
  representation equals $\dim(V(\lambda)_n)$, which is not constant,
  and indeed tends to infinity with $n$.
\end{nonexample}

\subsection{Strong and mixed tensor stability}
\label{section:strong}

In this subsection we define two variations of
representation stability.  Both variations will be used later in the
paper in the analysis of certain examples.  The reader might want to
skip this subsection until encountering those examples and move to \S\ref{section:classicalstability}.

\para{Strong stability} Conditions I and II together give a kind of 
``isomorphism'' between representations of different groups, but
they give no information about the subrepresentations of the
$V_n$. Condition III better captures the internal structure of the
representations $V_n$, but ignores the maps between the
representations. For example, Condition III alone does not rule out
the possibility that the maps $\phi_n\colon V_n\to V_{n+1}$ are all
zero. The following condition combines these approaches to give
careful control over the behavior of a subrepresentation under
inclusion. We require that for every irreducible $V(\lambda)_n\subset
V_n$, the $G_{n+1}$--span of the image $\phi_n(V(\lambda)_n)$ is
isomorphic to $V(\lambda)_{n+1}$.

\begin{definition}[Strong representation stability]
\label{definition:repstab2}
A consistent sequence  $\{V_n\}$ of $G_n$--representations is
\emph{strong representation stable} if for sufficiently large $n$, not
depending on $\lambda$, Conditions I, II, and III$'$ hold (that is, 
$\{V_n\}$ is uniformly representation stable), and the following condition holds:
\begin{enumerate}
\item[{\bf IV.}] \textbf{Type-preserving:} For any subrepresentation
  $W\subset V_n$ so that $W\approx V(\lambda)_{n}$, the span of the
  $G_{n+1}$--orbit of $\phi_n(W)$ is isomorphic to $V(\lambda)_{n+1}$.
\end{enumerate}
\end{definition}
It is possible to embed the $\GL_n\Q$--module $\bwedge^i
\Q^n=V(\omega_i)_n$ into the $\GL_{n+1}\Q$--module
$\bwedge^{i+1}\Q^{n+1}=V(\omega_{i+1})_{n+1}$ by $v\mapsto v\wedge
x_{n+1}$. This embedding respects the group actions, but the
$\GL_{n+1}\Q$--span of the image is all of $\bwedge^{i+1}\Q^{n+1}$;
similar embeddings occur for other pairs of irreducible
representations. Condition IV rules out this type of phenomenon.  One 
example of a uniformly stable sequence of $S_n$--representations that
is not strongly stable is given by the cohomology of pure braid
groups; see \S\ref{section:braid}.

\begin{remark}
  \label{remark:strong}
  For applications, we will need the stronger statement that any
  subspace isomorphic to $V(\lambda)_n^{\oplus k}$ has $G_{n+1}$--span
  isomorphic to $V(\lambda)_{n+1}^{\oplus k}$, where the multiplicity
  $k$ may be greater than $1$. Fortunately, this stronger statement
  follows from Condition IV above. First, the maps
  $V_n\to V_{n+1}$ are injective (apply Condition IV to any $W$
  contained in the kernel). Furthermore, for a fixed $\lambda$,
  Condition IV implies that the inclusions $V_n\hookrightarrow
  V_{n+1}$ restrict to inclusions of $\lambda$--isotypic components
  $V_n^{(\lambda)}\hookrightarrow V_{n+1}^{(\lambda)}$.

  It is thus clear that the $G_{n+1}$--span of $V(\lambda)_{n}^{\oplus
    k}$ is $V(\lambda)_{n+1}^{\oplus \ell}$ with $\ell\leq k$. The
  potential problem is that two independent subrepresentations
  $W,W'\approx V(\lambda)_n\subset V_n$ could both map into the same
  $V(\lambda)_{n+1}\subset V_{n+1}$. This is ruled out by the
  following property, shared by each of our families of groups: the
  restriction $V(\lambda)_{n+1}\downarrow G_n$ contains the
  irreducible $G_n$--representation $V(\lambda)_n$ with multiplicity
  1. Thus the multiplicity of $V(\lambda)_n$ in
  $V(\lambda)_{n+1}^{\oplus \ell}\downarrow G_n$ is $\ell$. But as
  $G_n$--representations, we have an inclusion
  \[V(\lambda)_n^{\oplus k}\hookrightarrow
  \big(V(\lambda)_{n+1}^{\oplus \ell}\downarrow G_n\big),\] which
  implies $k\leq \ell$, verifying the stronger statement as desired.

  For $G_n=\SL_n\Q$ or $\GL_n\Q$, the property mentioned above can be
  seen from the formula \eqref{eq:SLres} given in the proof of
  Theorem~\ref{thm:classicalstability}(6) below. For $G_n=\Sp_{2n}\Q$,
  it follows from \eqref{eq:Spres} below. For $G_n=S_n$, this is the
  classical \emph{branching rule} \cite[Equation 4.42]{FH}:
  \[V(\lambda)_{n+1}\downarrow S_n=V(\lambda)_n\oplus
  \bigoplus_\mu V(\mu)_n\] where $\mu$ ranges over those partitions
  obtained by removing one box from $\lambda$. For $G_n=W_n$, the
  branching rule has the form \cite[Lemma 6.1.3]{GP}:
  \[V(\lambda^+,\lambda^-)_{n+1}\downarrow W_n
  =V(\lambda^+,\lambda^-)_n\oplus
  \bigoplus_{\mu^+}V(\mu^+,\lambda^-)_n\oplus
  \bigoplus_{\mu^-}V(\lambda^+,\mu^-)_n\] where $\mu^+$ is obtained
  from $\lambda^+$, and $\mu^-$ is obtained from $\lambda^-$, by
  removing one box.

  It follows that assuming surjectivity, Condition IV also implies
  Condition III$'$. Conversely, as long as the $V_n$ are
  finite-dimensional, or even have finite multiplicities $0\leq
  c_{\lambda,n}<\infty$, Conditions III$'$ and IV together imply
  Condition II.
\end{remark}

\para{An equivalent formulation of Condition IV}
\label{reformulation}
When $G_n$ is $\SL_n\Q$ or $\GL_n\Q$, Condition IV can be stated in a
more familiar basis-dependent form.  Let $P_{n+1}$ be the
$n$--dimensional subgroup of $\SL_{n+1}\Q$ preserving and acting
trivially on $\Q^n<\Q^{n+1}$; that is, agreeing with the identity outside the rightmost column. Then assuming uniform multiplicity
stability, Condition IV can be stated as follows.

\begin{proposition}
  \label{prop:SLstrong}
  For $G_n=\SL_n\Q$ or $\GL_n\Q$, let $\{V_n\}$ be a uniformly
  multiplicity stable sequence of $G_n$--representations. Assume that
  the maps $\phi_n\colon V_n\hookrightarrow V_{n+1}$ are injective. The
  sequence $\{V_n\}$ is type-preserving (satisfies Condition IV) for
  sufficiently large $n$ if and only if the following condition is
  satisfied for sufficiently large $n$.
  \end{proposition}
  \begin{enumerate}
  \item[{\bf IV$'$.}] $P_{n+1}$ acts trivially on the image
    $\phi(V_n)$ of $V_n$ in $V_{n+1}$.
  \end{enumerate}

Condition IV$'$ is in practice much easier to check than Condition
IV. As we will see in Theorem~\ref{thm:classicalstability}, Condition
IV$'$ is also preserved by many natural constructions. It is
equivalent to the statement that $\phi_n$ takes highest weight vectors in $V_n$ to 
highest weight vectors in $V_{n+1}$. 

\begin{proof}
  Within this proof, let $\fp_{n+1}$ be the Lie algebra of
  $P_{n+1}$. Explicitly, $\fp_{n+1}$ is the span of the elementary
  matrices $E_{i,n+1}$ with $1\leq i\leq n$. The subgroup $P_{n+1}$
  was chosen exactly so that $\fn^+_{n+1}$ is spanned by $\fn^+_n$
  together with $\fp_{n+1}$.

  \para{IV$'$ $\implies$ IV} In fact, we need only assume that $P_{n+1}$
  acts trivially on the image of each highest weight vector. Consider
  a highest weight vector $v\in V_n$, so $v$ is an eigenvector for
  $\fh_n$ with weight $\lambda\in\fh_n^*$, and $v$ is annihilated by
  $\fn^+_n$. By possibly rechoosing $v$, we may assume that $\phi(v)$
  is an eigenvector for $\fh_{n+1}$ with weight
  $\lambda'\in\fh_{n+1}^*$. The consistency of the map $V_n\to
  V_{n+1}$ implies that under the restriction map $\fh_{n+1}^*\to
  \fh_n^*$, the weight $\lambda'$ restricts to $\lambda$. The
  condition that $P_{n+1}$ acts trivially on $\phi(V_n)$ implies that
  $\fp_{n+1}$ annihilates $\phi(V_n)$. It follows that
  $\fn^+_{n+1}=\fn^+_n\oplus \fp_{n+1}$ annihilates $\phi(v)$, so
  $\phi(v)$ is a highest weight vector for $G_{n+1}$. By assumption
  $\{V_n\}$ is uniformly multiplicity stable. This implies that once
  $n$ is sufficiently large, the only weight $\lambda'$ occurring in
  $V_{n+1}$ which restricts to $\lambda\in \fh_n^*$ is the weight
  satisfying $V(\lambda)_{n+1}=V(\lambda')_{n+1}$. Thus we see that
  $\phi(v)$ spans the subrepresentation $V(\lambda)_{n+1}$, as
  desired. Since this holds for all highest weight vectors $v$, and
  each irreducible subrepresentation is the span of a highest weight
  vector, Condition IV follows.

  \para{IV $\implies$ IV$'$} Conversely, if $\phi\colon V_n\hookrightarrow
  V_{n+1}$ is type-preserving, let $v\in V_n$ be a highest weight
  vector for $G_n$ spanning $V(\lambda)_n$, and consider its image in
  $V_{n+1}$. Certainly $\phi(v)$ remains a highest weight vector for
  $G_n$ with weight $\lambda$. By Condition IV, its $G_{n+1}$--span is
  isomorphic to $V(\lambda)_{n+1}=V(\lambda')_{n+1}$. Let $w\in
  V(\lambda)_{n+1}$ be the $G_{n+1}$--highest weight vector with
  heighest weight $\lambda'$.  Then $w$ is evidently a highest weight vector
  for $G_n$ with weight $\lambda$ as well. But as noted in
  Remark~\ref{remark:strong}, the restriction
  $V(\lambda)_{n+1}\downarrow G_n$ contains $V(\lambda)_n$ with
  multiplicity 1, so $V(\lambda)_{n+1}$ contains a unique
  $G_n$--highest weight vector with weight $\lambda$. Thus $\phi(v)$
  must coincide with $w$, and in particular $\phi(v)$ is a highest
  weight vector for $G_{n+1}$.

  This implies that $P_{n+1}$ acts trivially on $\phi(v)$ for each
  highest weight vector $v$. It remains to show that $P_{n+1}$ acts
  trivially on the entire image of $V_n$, that is on the $G_n$--span
  of the highest weight vectors $\phi(v)$. This is a general fact of
  representation theory. Since $P_{n+1}$ is contained in
  $\SL_{n+1}\Q$, we may assume that $G_n=\SL_n\Q$.  For the rest of
  the argument we identify
  $\mathfrak{h}_{n+1}^*=\Z[L_1,\ldots,L_{n+1}]/(L_1+\cdots+L_{n+1})$
  with $\Z[L_1,\ldots,L_n]$ by setting $L_{n+1}=0$. Restrict to the
  inclusion of a single irreducible $V(\lambda)_n\subset
  V(\lambda)_{n+1}$, with highest weight vector $v$ of weight
  $\lambda=\lambda_1L_1+\cdots+\lambda_nL_n$. Let $k\coloneq \sum
  \lambda_i$ be the sum of the coefficients.

  The irreducible representation $V(\lambda)_{n+1}$ is the span of $v$
  under $\fn^-_{n+1}$, which is spanned by the elementary matrices
  $\{E_{j,i}|1\leq i<j\leq n+1\}$. If $j\leq n$ the
  matrix $E_{j,i}$ has weight $L_j-L_i$, while $E_{n+1,i}$ has weight
  $L_{n+1}-L_i=-L_i$. Adding the former does not change the sum of the
  coefficients, while adding the latter decreases the sum, so every
  weight $\mu=\mu_1L_1+\cdots+\mu_nL_n$ occurring in $V(\lambda)_{n+1}$
  has $\sum \mu_i\leq k$. The subspace $V(\lambda)_n$ is the span of
  $v$ under $\fn^-_n$, which is spanned by the $\{E_{j,i}|1\leq
  i<j\leq n\}$ with roots $\{L_j-L_i|1\leq i<j\leq n\}$. Thus the
  weights $\mu$ occurring in $V(\lambda)_n$ all have
  $\sum\mu_i=k$. Applying any matrix $E_{i,n+1}\in \fp_{n+1}$ with
  weight $L_i-L_{n+1}=L_i$ to such a vector would yield a vector with
  weight
  \[\mu_1L_1+\cdots+(\mu_i+1)L_i+\cdots+\mu_nL_n.\]
  The sum of the coefficients of such a weight is $k+1$, and we have
  already said that no such weight occurs in $V(\lambda)_{n+1}$. It
  follows that $\fp_{n+1}$ must annihilate every element of
  $V(\lambda)_{n}$, and thus $P_{n+1}$ acts trivially on
  $V(\lambda)_n\subset V(\lambda)_{n+1}$, as desired. We conclude that
  $P_{n+1}$ acts trivially on $\phi(V_n)\subset V_{n+1}$.
\end{proof}

\begin{remark}
  There is no such nice formulation of Condition IV for
  representations of $\Sp_{2n}\Q$. Indeed we will see in
  Theorem~\ref{thm:classicalstability} that if $\{V_n\}$ is a strongly
  stable sequence of $\SL_n\Q$--representations, then for any Schur
  functor $\Schur_\lambda$, the sequence $\{\Schur_\lambda(V_n)\}$ is
  strongly stable as well. The proof hinges upon the equivalence of
  Conditions IV and IV$'$.

  The corresponding fact for $\Sp_{2n}\Q$ is false. For example,
  $\{V_n=\Q^{2n}\}$ is certainly strongly stable. However,
  $\{\bwedge^2 V_n=\bwedge^2 \Q^{2n}\}$ is not strongly stable. The
  unique trivial subrepresentation of $\bwedge^2 \Q^{2n}$ is spanned
  by $a_1\wedge b_1+\cdots+a_n\wedge b_n$. However, this vector is not
  taken to a trivial subrepresentation of $\bwedge^2 \Q^{2n+2}$. In
  fact, the $\Sp_{2n+2}\Q$--span of $a_1\wedge b_1+\cdots+a_n\wedge
  b_n$ is all of $\bwedge^2 \Q^{2n+2}$. This failure is related to the
  fact, described in Remark~\ref{rem:Spweights}, that the upper-left
  inclusion $\Sp_{2n}\Q\subset \Sp_{2n+2}\Q$ does not respect the
  ordering of the roots. Of course there does exist some map
  $V(0)_n\oplus V(\lambda_2)_n\to V(0)_{n+1}\oplus V(\lambda_2)_{n+1}$
  which is type-preserving, but viewed as a map $\bwedge^2 \Q^{2n}\to
  \bwedge^2 \Q^{2n+2}$ this map appears wholly unnatural.
\end{remark}

\para{Mixed tensor representations}
There are certain natural families of representations with an inherent
``stability'', which indeed satisfy the definition of
representation stability given above, but for trivial reasons that do
not capture the real nature of their stability.  For example, the dual
of the standard representation of $\GL_n\Q$ has highest weight $-L_n$.
In terms of pseudo-partitions, the dual representation
$V(1,0,\ldots,0)_n^*$ is the representation $V(0,\ldots,0,-1)_n$,
which is given by a different pseudo-partition for each $n$. So in the
sequence of representations $\{V_n=(\Q^n)^*\}$, for each $\lambda$ the
irreducible $V(\lambda)_n$ appears in $V_n$ for at most one $n$, from
which it follows that the sequence $\{V_n\}$ does fit the definition of
representation stable given above.  However, the ``stable
representation'' is trivial, since each representation $V(\lambda)$
eventually has multiplicity 0.

To accurately capture the stability of this sequence, as well as other
natural sequences such as $\{V_n^*\otimes \bwedge^2 V_n\}$ and the adjoint
representations $\{\fsl_n\Q\}$,  we will use mixed tensor representations to 
define a stronger condition than representation stability. Given
two partitions $\lambda=(\lambda_1,\ldots,\lambda_k)$ and
$\mu=(\mu_1,\ldots,\mu_\ell)$, for $n\geq k+\ell$ the \emph{mixed
  tensor representation} $V(\lambda;\mu)_n$ is the irreducible
representation of $\GL_n\Q$ with highest
weight \[\lambda_1L_1+\cdots+\lambda_kL_k-\mu_\ell
L_{n-\ell+1}-\cdots-\mu_1L_n.\] Equivalently, $V(\lambda;\mu)_n$ is
the irreducible representation corresponding to the pseudo-partition
$(\lambda_1,\ldots,\lambda_k,0,\ldots,0,-\mu_\ell,\ldots,-\mu_1)$. Note
that when restricted to $\SL_n\Q$, this representation corresponds to
the partition \[(\mu_1+\lambda_1,\ldots,\mu_1+\lambda_k,\mu_1,
\ldots,\mu_1,\mu_1-\mu_\ell,\ldots,\mu_1-\mu_2,0).\]

\begin{definition}[Mixed tensor stable] 
  A consistent sequence of $\GL_n\Q$--representations or
  $\SL_n\Q$--representations $\{V_n\}$ is called \emph{mixed
    representation stable} if Conditions I and II are satisified for
   large enough $n$, and if in addition the following condition
  is satisfied:
  \begin{enumerate} 
  \item[{\bf MTIII.}] For all partitions $\lambda$ and
    $\mu$, the multiplicity of the mixed tensor representation
    $V(\lambda;\mu)_n$ in $V_n$ is eventually constant.
  \end{enumerate}
\end{definition}

Note that mixed tensor stability implies representation stability.  As an example of mixed tensor stability, consider the adjoint representation of $\SL_n\Q$ or
$\GL_n\Q$ on $\fsl_n\Q$. This corresponds to the partition
$(2,1,\ldots,1,0)$, or to the pseudo-partition
$(1,0,\ldots,0,-1)$. Thus $\fsl_n\Q$ is the mixed tensor
representation $V(1;1)_n$. Similarly, the dual $(\Q^n)^*$ of the
standard representation is $V(0;1)_n$. In general, the dual of
$V(\lambda;\mu)_n$ is $V(\mu;\lambda)_n$, so in particular if a
sequence $\{V_n\}$ is representation stable, the sequence of duals
$\{V_n^*\}$ is mixed tensor stable. We remark that a sequence
which is mixed representation stable is essentially never type-preserving.

Mixed representation stability is used in
Sections~\ref{section:torelli} and \ref{section:congruence}. This
notion was applied by Hanlon \cite{Han} to Lie algebra cohomology over
non-unital algebras, and futher applied by R. Brylinski \cite{Bry}.

\section{Stability in classical representation theory}
\label{section:classicalstability}

In this section we discuss examples of representation stability in
classical representation theory.  We remark that the definition of
representation stability itself already relies upon an inherent
stability in the classification of irreducible representations of the
groups $G_n$, in the sense that the system of names of representations
of the varying groups $G_n$ can be organized in a coherent way.

\subsection{Combining and modifying stable sequences}
The ubiquity of representation stability would be unlikely were it not
that many of the natural constructions in classical representation
theory preserve representation stability.  We now formalize this.
Many of the results follow from well-known classical theorems.

\begin{theorem}
\label{thm:classicalstability}
  Suppose that $G_n=\SL_n\Q$ and that $\{V_n\}$ and $\{U_n\}$ are
  multiplicity stable sequences of finite-dimensional
  $G_n$--representations. Fix partitions $\lambda$ and $\mu$.  Then
  the following sequences of $G_n$--representations are multiplicity
  stable.
  \begin{enumerate}
  \item \textbf{Tensor products: } $\{V_n\otimes
    U_n\}$.
  \item \textbf{Schur functors: } $\{\Schur_\lambda
    (V_n)\}$.
  \item \textbf{Schur functors of direct sums: } $\{\Schur_\lambda(V_n\oplus
    U_n)\}$.
  \item \textbf{Schur functors of tensor products: } $\{\Schur_{\lambda}(V_n\otimes
    U_n\})$.
  \item \textbf{Compositions of Schur functors: }
    $\{\Schur_\lambda(\Schur_\mu(V_n))\}$.\\ For example,
    $\{\Sym^r(\bwedge^s(V_n))\}$ for each fixed $r,s\geq 0$.
  \end{enumerate}
  If $G_n$ is $\SL_n\Q$, $\GL_n\Q$ or $\Sp_{2n}\Q$ and $\{V_n\}$ and
  $\{U_n\}$ are uniformly multiplicity stable sequences of
  finite-dimensional $G_n$--representations, then all the preceding
  examples are uniformly multiplicity stable, as are the following two
  examples.
  \begin{enumerate}
  \item[6.] \textbf{Shifted sequences: } The restrictions
    $\{V_n\downarrow G_{n-k}\}$ for any fixed $k\geq 0$.
  \item[7.] \textbf{Restrictions: } The restrictions $\{V_n\downarrow
    \SL_n\Q\}$ and $\{V_{2n}\downarrow\Sp_{2n}\Q\}$.
  \end{enumerate}
  If $G_n$ is $\SL_n\Q$ or $\GL_n\Q$ and $\{V_n\}$ and
  $\{U_n\}$ are strongly stable, then the resulting sequences in Parts
  1--5 are strongly stable.
\end{theorem}

We also have a version of Theorem~\ref{thm:classicalstability} for
$S_n$--representations.

\begin{theorem}
\label{thm:classicalSn}
  Let $\{V_n\}$ and $\{W_n\}$ be consistent sequences of
  $S_n$--representations that are uniformly multiplicity stable.  Then
  the following sequences of $S_n$--representations are uniformly
  multiplicity stable.
  \begin{enumerate}
  \item \textbf{Tensor products:} $\{V_n\otimes W_n\}$
  \item \textbf{Shifted sequences:} The restrictions $\{V_n\downarrow
    S_{n-k}\}$ for any fixed $k\geq 0$.
  \end{enumerate}
\end{theorem}

Before proving Theorem~\ref{thm:classicalstability} and
Theorem~\ref{thm:classicalSn}, we give a number of examples in order
to illustrate the necessity of various hypotheses in the theorems.

\begin{nonexample}
  In the final part of the statement of
  Theorem~\ref{thm:classicalstability}, we need to assume strong
  stability even to conclude standard representation stability of the
  various combinations of representations.  This strong assumption is
  used via the ``type-preserving'' condition, and without this
  assumption stability may not hold.  Perhaps surprisingly, the issue
  is not the stability of the multiplicities, but surjectivity. Here
  is a simple example which illustrates the problem.  This example is
  not injective, but it can easily be made so; we do this
  below. Let \[V_n=V(0)_n\oplus V(1)_n=\Q\oplus \Q^n,\] with maps
  $V_n\to V_{n+1}$ defined by
  \[\Q\oplus \Q^n\ni(a,v)\mapsto (a,ax_{n+1})\in\Q\oplus \Q^{n+1}\]
  where $x_{n+1}$ is the basis vector $x_{n+1}=(0,\ldots,0,1)$.  The tensor product 
  $V_n\otimes V_n$ decomposes into irreducibles as ${V(0)\oplus
    V(1)^{\oplus 2}\oplus V(2)\oplus V(1,1)}$, where the last two
  factors come from the decomposition $\Q^n\otimes
  \Q^n=\Sym^2\Q^n\oplus \bwedge^2 \Q^n$. It is easy to check that the
  $\SL_{n+1}\Q$--span of the image of $V_n\otimes V_n$ in
  $V_{n+1}\otimes V_{n+1}$ is exactly $V(0)\oplus V(1)^{\oplus
    2}\oplus V(2)$; the $\bwedge^2 \Q^n$ factor is inaccessible.

  For an example which is actually representation stable, let
  \[V_n=V(0)_n\oplus V(1)_n\oplus V(2)_n=\Q\oplus \Q^n\oplus
  \Sym^2\Q^n\] with $V_n\hookrightarrow V_{n+1}$ defined by $(a,v,w)
  \mapsto (a,ax_{n+1},w+v\cdot x_{n+1})$. The sequence $\{V_n\}$ is consistent and uniformly representation stable.  The tensor product $V_n\otimes V_n$
  contains $V(1,1)=\bwedge^2 \Q^n$ with multiplicity 1, but the
  $\SL_{n+1}\Q$ image of $V_n\otimes V_n$ does not contain this
  factor. Thus the sequence $\{V_n\otimes V_n\}$ is not surjective in
  the sense of Definition~\ref{definition:repstab1}.
  \end{nonexample}
  
\begin{nonexample}
  Even when the sequence $\{V_n\}$ is type-preserving (and so strongly
  stable), restrictions often fail to be surjective. Take
  $V_n=V(1)_n=\Q^n$, which certainly is strongly stable. The
  restriction $W_n=V_{n+1}\downarrow \GL_n\Q$ splits as $V(1)_n\oplus
  V(0)_n=\Q^n\oplus \Q$, which is multiplicity stable. But the image
  of $W_n$ in $W_{n+1}=\Q^{n+2}$ is invariant under $\GL_{n+1}\Q$, and
  so the sequence $\{W_n\}$ is not stable due to the failure of
  surjectivity.
\end{nonexample}

\begin{proof}[Proof of Theorem~\ref{thm:classicalstability}]
  In each case, injectivity is either trivial or follows from the
  functoriality of the Schur functor $\Schur_\lambda$. The proofs of
  multiplicity stability generally separate into two parts. First, we
  check stability when each $V_n$ is a single irreducible; the proof
  in this case often corresponds to a classical fact of representation
  theory. Second, we promote this to the general case when $\{V_n\}$
  is an arbitrary representation stable sequence. This sometimes
  requires bootstrapping off the first step of other parts.
  
  To reduce confusion, we have labeled the two steps of the proofs
  separately as (for example) Parts 1a and 1b. For simplicity, we
  refer to ``stability'' and ``uniform stability'' in the course of
  the proofs, but we reiterate that we are not claiming surjectivity,
  so these should properly be references to ``multiplicity
  stability''. We defer the discussion of strong stability until after
  the claims have been verified in the stable and uniformly stable
  cases. \\

  \textbf{1.} In general, the problem of decomposing the tensor
  product of two irreducible representations is called the
  \emph{Clebsch--Gordan problem}. The quintessential example of
  stability is the Littlewood--Richardson rule, which answers the
  Clebsch--Gordan problem for $\SL_n$ and shows that the
  multiplicities in the decomposition are independent of $n$. Given
  two partitions $\lambda$ and $\mu$, and a partition $\nu\vdash
  |\lambda|+|\mu|$, the \emph{Littlewood--Richardson coefficient}
  $C^{\nu}_{\lambda\mu}$ is the number of ways that $\nu$ can be
  obtained as a strict $\mu$--expansion of $\lambda$ (see
  \cite[Appendix A]{FH} for full definitions). The tensor product then
  decomposes as \cite[Equation 6.7]{FH}
  \begin{equation}
    \label{eq:LR}
    \Schur_\lambda(V)\otimes \Schur_\mu(V)=\bigoplus
    C^{\nu}_{\lambda\mu}\Schur_\nu(V).
  \end{equation}
  
  \textbf{1a.} We first verify the claim in the case of a single
  irreducible. We show that in each case, the tensor
  $V(\lambda)_n\otimes V(\mu)_n$ decomposes as $\bigoplus
  N_{\lambda\mu}^\nu V(\nu)_n$ for some constant $N_{\lambda\mu}^\nu$
  independent of $n$. For $\SL_n\Q$, by the Littlewood--Richardson
  rule $V(\lambda)_n\otimes V(\mu)_n$ decomposes as $\bigoplus
  C^{\nu}_{\lambda\mu}V(\nu)_n$, so we may take
  $N^{\nu}_{\lambda\mu}=C^{\nu}_{\lambda\mu}$. For $\GL_n\Q$, recall
  that $D$ denotes the determinant representation, and note that for fixed
  $\ell$, the representation $D^\ell\otimes V_n$ is stable if and only
  if $V_n$ is stable. Every irreducible $V(\lambda)_n$ can be written
  as $\Schur_{\overline{\lambda}}\Q^n\otimes D^\ell$ for a unique
  partition $\overline{\lambda}$ and integer $\ell$ (namely
  $\ell=\lambda_n$ and
  $\overline{\lambda}_i=\lambda_i-\lambda_n$). Then
  \[V(\lambda)_n\otimes V(\mu)_n
  =\Schur_{\overline{\lambda}}(\Q^n)\otimes D^\ell\otimes
  \Schur_{\overline{\mu}}(\Q^n)\otimes D^m=D^{\ell+m}\otimes \bigoplus
  C^{\overline{\nu}}_{\overline{\lambda}\overline{\mu}}\Schur_{\overline{\nu}}(\Q^n).\]
  The decomposition of the right side into irreducibles
  $V(\nu)_n=D^{\ell+m}\otimes \Schur_{\overline{\nu}}(\Q^n)$ is
  independent of $n$. Thus we may take
  $N^{\nu}_{\lambda\mu}=C^{\overline{\nu}}_{\overline{\lambda}\overline{\mu}}$
  for those $\nu$ with $\nu_n=\lambda_n+\mu_n$, and
  $N^{\nu}_{\lambda\mu}=0$ otherwise.

  For $\Sp_{2n}\Q$, the corresponding formula is \cite[Equation
  25.27]{FH}:
  \begin{equation}
    \label{eq:SpLR}
    \Schur_{\langle\lambda\rangle}(\Q^{2n}) \otimes
    \Schur_{\langle\mu\rangle}(\Q^{2n}) =\bigoplus
    \sum_{\zeta,\sigma,\tau}C^{\lambda}_{\zeta\sigma}
    C^{\mu}_{\zeta\tau}
    C^{\nu}_{\sigma\tau}\Schur_{\langle\nu\rangle}(\Q^{2n})
  \end{equation}
  where the sum is over all partitions $\zeta,\sigma,\tau$.  Thus
  $N^{\nu}_{\lambda\mu}=\sum_{\zeta,\sigma,\tau}
  C^{\lambda}_{\zeta\sigma} C^{\mu}_{\zeta\tau}C^{\nu}_{\sigma\tau}$.

  \textbf{1b.} Now consider arbitrary consistent sequences $\{V_n\}$ and
  $\{U_n\}$. If these sequences are uniformly representation stable, their decompositions
  $V_n=\bigoplus c_{\lambda,n}V(\lambda)_n$ and $U_n=\bigoplus
  d_{\mu,n}V(\mu)_n$ are eventually independent of $n$. Thus the
  decomposition of the tensor product as
  \[V_n\otimes U_n =\left(\bigoplus
    c_{\lambda,n}V(\lambda)_n\right)\otimes\left(\bigoplus
    d_{\mu,n}V(\mu)_n\right) =\bigoplus_\nu
  \sum_{\lambda,\mu}c_{\lambda,n}d_{\mu,n}N_{\lambda\mu}^\nu
  V(\nu)_n\] is eventually independent of $n$.

  For $\SL_n\Q$ the assumption of uniform stability is not
  necessary. In this case $N_{\lambda\mu}^\nu$ is the
  Littlewood--Richardson coefficient, which is nonzero only if
  $|\nu|=|\lambda|+|\mu|$. Thus for fixed $\nu$, only finitely many
  pairs $(\lambda,\mu)$ can contribute to the
  $\sum_{\lambda,\mu}c_{\lambda,n}d_{\mu,n}N_{\lambda\mu}^\nu
  V(\nu)_n$ term above. Thus we may take $n$ large enough that these
  finitely many coefficients $c_{\lambda,n}$ and $d_{\mu,n}$ are all
  independent of $n$, and the multiplicity
  $\sum_{\lambda,\mu}c_{\lambda,n}d_{\mu,n}N_{\lambda\mu}^\nu$ of
  $V(\nu)_n$ is eventually independent of $n$ as desired.\\

  \textbf{2a.} The classical \emph{plethysm} problem is to decompose the
  composition of two Schur functors:
  \begin{equation}
    \label{eq:SLpleth}
    \Schur_{\lambda}(\Schur_\mu V)
    =\bigoplus M_{\lambda\mu}^\nu \Schur_\nu V
  \end{equation}
  To compute the coefficients $M_{\lambda\mu}^\nu$ is difficult, but
  it is known that such coefficients exist, and are nonzero only when
  $|\nu|=|\lambda|\cdot|\mu|$ \cite[Exercise 6.17a]{FH}. It
  immediately follows that the sequence $\{\Schur_\lambda(V(\mu)_n\}$
  of $\SL_n\Q$--representations $\Schur_\lambda(V(\mu)_n)=\bigoplus
  M_{\lambda\mu}^\nu V(\nu)_n$ is representation stable. For
  $\GL_n\Q$, write $V(\mu)_n=\Schur_{\overline{\mu}}\Q^n\otimes
  D^{\ell}$ for some partition $\lambda$ and integer $\ell$. Note that
  in general, if $\rho$ acts on $V$ diagonally by multiplication by
  $R$, then the action of $\rho$ on $\Schur_{\lambda}V$ will be
  multiplication by $R^{|\lambda|}$. Since the center of $\GL_n\Q$
  acts diagonally, it follows that
  \[\Schur_\lambda(V(\mu)_n)
  =\Schur_\lambda(\Schur_{\overline{\mu}}(\Q^n)\otimes
  D^\ell)=\Schur_\lambda(\Schur_{\overline{\mu}}(\Q^n))\otimes
  D^{\ell|\lambda|}=\bigoplus
  M_{\lambda\overline{\mu}}^{\overline{\nu}}
  \Schur_{\overline{\nu}}(\Q^n)\otimes D^{\ell|\lambda|}\] and thus
  $\{\Schur_\lambda(V(\mu)_n)\}$ is representation stable. For the
  symplectic group the stability of the
  plethysm \[\Schur_\lambda(\Schur_{\langle\mu\rangle}(\Q^{2n}))=\bigoplus
  L_{\lambda\mu}^\nu\Schur_{\langle\nu\rangle}(\Q^{2n})\] was only
  proved recently by Kabanov \cite[Theorem 7]{Kab}. If $\mu$ has $\ell=\ell(\mu)$
  rows, the coefficients $L_{\lambda\mu}^\nu$ are independent of $n$
  once $n\geq \ell|\lambda|$ and are nonzero only for those $\nu$ with at
  most $\ell|\lambda|$ rows.

  \textbf{2b and 3.} We now verify Parts 2 and 3 in parallel by
  induction on total multiplicity. We do this first under the
  assumption of uniform stability, so that total multiplicity is
  well-defined. We then explain how to extend this to all
  representation stable sequences in the case of $\SL_n\Q$. In
  general, when a Schur functor is applied to a direct sum we have the
  decomposition \cite[Exercise 6.11]{FH}
  \begin{equation}
    \label{eq:Schursum}
    \Schur_\lambda(V\oplus U)
    =\bigoplus C_{\mu\nu}^\lambda(\Schur_\mu V\otimes \Schur_\nu U).
  \end{equation}
  We have already verified Part 2 when $V_n$ has total multiplicity 1,
  and Part 3 reduces to Part 2 when $V_n\oplus U_n$ has total
  multiplicity 1. We now prove Part 3 when $V_n\oplus U_n$ has total
  multiplicity $k$ by strong induction. Assume that $\{V_n\}$ and
  $\{U_n\}$ are uniformly representation stable sequences, and that neither is eventually
  zero. So we may assume that Part 2 of the theorem holds for
  $\{V_n\}$ and for $\{U_n\}$ by induction. By Part 2 we have that
  $\{\Schur_\mu V_n\}$ and $\{\Schur_\nu U_n\}$ are each uniformly
  stable.  By Part 1, the tensor product $\{\Schur_\mu V_n\otimes
  \Schur_\nu U_n\}$ is uniformly stable. Thus the sum
  \[\Schur_\lambda(V_n\oplus U_n) =\bigoplus
  C_{\mu\nu}^\lambda(\Schur_\mu V_n\otimes \Schur_\nu U_n)\] is uniformly
  stable, verifying Part 3. To verify Part 2 when $V_n$ has total
  multiplicity $k$, write $V_n=U_n\oplus W_n$ with each factor
  uniformly stable and apply Part 3. Although the splitting
  $V_n=U_n\oplus W_n$ might not respect the maps $V_n\to V_{n+1}$, we
  are only concerned with multiplicities at this point so this is not
  a problem. When we revisit this issue later, it will be under the
  assumption of strong stability, in which case $\{V_n\}$ does split
  as a sum of consistent, strongly stable sequences $\{U_n\}$ and $\{W_n\}$.

  We now consider the case when $G_n=\SL_n\Q$ and the sequences are
  not necessarily uniformly stable. For a fixed finite-dimensional
  $\SL_n\Q$--representation $V=\bigoplus c_\eta V(\eta)$, consider
  decomposing $\Schur_\lambda(V)=\Schur_{\lambda}(\bigoplus c_\eta
  V(\eta))$ by repeatedly applying the formula \eqref{eq:Schursum} for
  $\Schur_\lambda(V\oplus W)$. We obtain a decomposition of the form
  \[\Schur_\lambda(V)=\bigoplus X_\bullet
  \Schur_{\mu_\bullet}V(\eta_\bullet)\otimes \cdots\otimes
  \Schur_{\mu_\bullet}V(\eta_\bullet),\] where the $V(\eta_\bullet)$
  range over the irreducible summands of $V$. Consider the individual
  terms $\Schur_{\mu}V(\eta)$. As we noted above, the decomposition
  $\Schur_{\mu}V(\eta)=\bigoplus M_{\mu\eta}^\zeta V(\zeta)$ only
  contains those $V(\zeta)$ with $|\zeta|=|\mu|\cdot
  |\eta|$. Furthermore, recall that the coefficients
  $C_{\mu\nu}^\lambda$ are only nonzero if $|\mu|+|\nu|=|\lambda|$
  (this is where we use that $G_n=\SL_n\Q$). It follows that the
  irreducibles $V(\nu)$ appearing in a tensor $V(\zeta_1)\otimes
  \cdots V(\zeta_k)$ all satisfy
  $|\nu|=|\zeta_1|+\cdots+|\zeta_k|$. Combining this, we obtain the
  key point of the argument: when considering the multiplicity of
  $V(\nu)$ in $\Schur_\lambda(\bigoplus c_\eta V(\eta))$, \emph{we
    need only consider those $V(\eta)$ with $|\eta|\leq |\nu|$}.

  Using this observation, we reduce this case to the uniformly stable
  case as follows. For fixed $\nu$, replace $V_n=\bigoplus c_{\eta,n}
  V(\eta)$ with \[V_n^{\leq \nu}=\bigoplus_{|\eta|\leq |\nu|}
  c_{\eta,n} V(\eta).\] If the sequence $\{V_n\}$ is representation
  stable, then since only finitely many $\eta$ satisfy $|\eta|\leq
  |\nu|$, the sequence $\{V_n^{\leq \nu}\}$ is uniformly stable. Thus
  applying Part 2, we conclude that
  $\{\Schur_\lambda(V_n^{\leq\nu})\}$ is uniformly stable; in
  particular, the multiplicity of $V(\nu)$ is eventually constant. By
  the observation, this is the same as the multiplicity of $V(\nu)$ in
  $\Schur_\lambda(V_n)$. Thus $\{\Schur_\lambda(V_n)\}$ is
  multiplicity stable, as desired.\\

  \textbf{4.} In general, when a Schur functor is applied to a tensor
  product, we have the decomposition \cite[Exercise 6.11b]{FH}
  \[\Schur_\lambda(V\otimes W)=\bigoplus D_{\mu\nu}^\lambda(\Schur_\mu
  V\otimes \Schur_\nu W),\] where the sum is over partitions with
  $|\mu|=|\nu|=|\lambda|$ and the coefficients are defined as
  follows. Let $d=|\lambda|$; given a partition $\eta\vdash d$, let
  $C_\eta$ be the conjugacy class in $S_d$ whose cycle decomposition
  is encoded by $\eta$. Let $\chi_\lambda$ be the character of the
  irreducible $S_d$--representation $V(\lambda)$. Then 
  \[D_{\mu\nu}^\lambda=\sum_{\eta\vdash d}
  \frac{\chi_\lambda(C_\eta)\chi_\mu(C_\eta)\chi_\nu(C_\eta)}
  {|Z_{S_d}(C_\eta)|}\] where $Z_{S_d}(C_\eta)$ is the centralizer in
  $S_d$ of a representative of $C_\eta$. Now assume that $\{V_n\}$ and
  $\{U_n\}$ are uniformly stable, and consider
  $\Schur_\lambda(V_n\otimes U_n)$. By Part 2, $\{\Schur_\mu V_n\}$
  and $\{\Schur_\nu U_n\}$ are uniformly stable for each $\mu$ and
  $\nu$.  By Part 1 we have that $\{\Schur_\mu V_n\otimes \Schur_\nu
  U_n\}$ is uniformly stable.  Thus the sum
  \[\Schur_\lambda(V_n\oplus U_n) =\bigoplus
  D_{\mu\nu}^\lambda(\Schur_\mu U_n\otimes \Schur_\nu W_n)\] is
  uniformly stable. The case when $G_n=\SL_n\Q$ and uniform stability
  is not assumed proceeds exactly as in Part 2, since the coefficents
  $D_{\mu\nu}^\lambda$ are only nonzero if $|\mu|=|\nu|=|\lambda|$.\\

  \textbf{5.} Stability for the composition of Schur functors
  $\Schur_\lambda(\Schur_\mu(V_n))$ can be deduced from the plethysm
  decomposition in \eqref{eq:SLpleth}, or just by applying Part 2
  twice, first to $\{\Schur_\mu(V_n)\}$ and then to
  $\{\Schur_\lambda(\Schur_\mu(V_n))\}$.\\

  \textbf{6.} For the restriction of
  $V(\lambda)_n=\Schur_\lambda(\Q^n)$ from $\GL_n\Q$ to $\GL_{n-k}\Q$,
  the restriction decomposes as \cite[Exercise 6.12]{FH}
  \begin{equation}
    \label{eq:SLres}
    \Schur_\lambda(\Q^n)\downarrow \GL_{n-k}\Q
    =\bigoplus_\nu\big(\sum_\mu C_{\mu\nu}^\lambda
    \dim\Schur_\mu(\Q^k)\big)\Schur_\nu(\Q^{n-k}).
  \end{equation}
  Note that $\dim \Schur_\mu(\Q^k)$ does not depend on $n$. The claim
  for a single irreducible representation of $\SL_n\Q$ immediately
  follows: \[V(\lambda)_n\downarrow
  \SL_{n-k}\Q=\bigoplus_\nu\big(\sum_\mu C_{\mu\nu}^\lambda \dim
  \Schur_\mu(\Q^k)\big)V(\nu)_{n-k}\] For $\GL_{n-k}\Q$ it follows
  after noting that the determinant representation restricts to the
  determinant representation: $D\downarrow \GL_{n-k}\Q=D$, so if
  $V(\lambda)_n=\Schur_{\overline{\lambda}}(\Q^n)\otimes D^\ell$ we
  get \[\Schur_{\overline{\lambda}}(\Q^n)\otimes D^\ell\downarrow
  \GL_{n-k}\Q=\bigoplus_{\overline{\nu}}\big(\sum_\mu
  C_{\mu\overline{\nu}}^{\overline{\lambda}} \dim
  \Schur_\mu(\Q^k)\big)\Schur_{\overline{\nu}}(\Q^{n-k})\otimes
  D^\ell.\] The claim for a uniformly stable sequence $\{V_n\}$
  follows by taking $n$ large enough that the decomposition
  $V_n=\bigoplus c_{\lambda,n}V(\lambda)_n$ is independent of
  $n$. Note that uniform stability is necessary here even for
  $\SL_n\Q$. Indeed, from \eqref{eq:SLres} we see that for every
  partition $\lambda$ with $\ell(\lambda)\leq k$, the restriction
  $\Schur_\lambda(\Q^n)\downarrow \SL_{n-k}\Q$ contains the trivial
  representation with multiplicity $\dim \Schur_\lambda(\Q^k)$. Thus
  the multiplicity of $V(0)$ in $V_n\downarrow \SL_{n-k}\Q$ is at
  least the total multiplicity of subrepresentations $V(\lambda)_n$ of
  $V_n$ with $\ell(\lambda)\leq k$, which need not be eventually constant
  if we do not assume uniform stability.

  For $\Sp_{2n}\Q$, we consider the restriction to $\Sp_{2n-2}\Q$. For
  a single irreducible representation $V(\lambda)_n$ of $\Sp_{2n}\Q$,
  the restriction decomposes as \cite[Equation 25.36]{FH}
  \begin{equation}
    \label{eq:Spres}
    V(\lambda)_n\downarrow \Sp_{2n-2}\Q=\bigoplus_\nu N_\lambda^\nu
    V(\nu)_{n-1},
  \end{equation}
  where the sum is over partitions with $\nu_n=0$. The
  coefficient $N_\lambda^\nu$ is the number of sequences
  $p_1,\ldots,p_n$ satisfying:
  \begin{align*}
    \lambda_1\geq\ &p_1\geq \lambda_2\geq p_2\geq
    \cdots\geq \lambda_n\geq p_n\\
    &p_1\geq\nu_1\geq p_2\geq \cdots\geq\nu_{n-1}\geq p_n\geq\nu_n= 0
  \end{align*} Note that if $\lambda_k=0$, then $p_i=0$ for $i\geq k$,
  and thus any $\nu$ contributing to this sum has $\nu_i=0$ for
  $i>k$. It follows that for fixed $\lambda$, the collection of $\nu$
  that contribute to this sum is independent of $n$ once $n\geq \ell(\lambda)$, so the collection of
  $p_1,\ldots,p_n$ above and the multiplicities $N_\lambda^\nu$ are
  also eventually independent of $n$. Thus $\{V(\lambda)_n\downarrow
  \Sp_{2n-2}\Q\}$ is stable, and as above it follows that
  $\{V_n\downarrow \Sp_{2n-2}\Q\}$ is uniformly stable if $\{V_n\}$ is
  uniformly stable. Uniform stability for the restriction to
  $\Sp_{2n-2k}\Q$ now follows by induction.\\

  \textbf{7.} Every irreducible $\GL_n\Q$--representation
  $V(\lambda)_n$ remains irreducible when restricted to $\SL_n\Q$; the
  resulting representation is $V(\overline{\lambda})_n$, where
  $\overline{\lambda}$ is the partition defined by
  ${\overline{\lambda}_i=\lambda_i-\lambda_n}$. If
  $V_n=\bigoplus_\lambda c_{\lambda,n}V(\lambda)_n$, the restriction
  $V_n\downarrow \SL_n\Q$ is $\bigoplus_\mu\sum_\lambda
  c_{\lambda,n}V(\mu)$ where the sum is over those $\lambda$ with
  $\overline{\lambda}=\mu$. For fixed $\mu$, the collection of such
  $\lambda$ is independent of $n$. Thus if $\{V_n\}$ is uniformly
  stable and $c_{\lambda,n}$ are eventually independent of $n$, the
  same is true of the multiplicities $\sum_\lambda
  c_{\lambda,n}$ of $V(\mu)_n$.

  It thus suffices to consider the restriction from $\SL_{2n}\Q$ to
  $\Sp_{2n}\Q$. Littlewood proved that if $\lambda$ is a partition
  with at most $n$ rows, the restriction of the irreducible
  $V(\lambda)_{2n}=\Schur_\lambda(\Q^{2n})$ decomposes as
  \cite[Equation 25.39]{FH}
  \begin{equation}
    \label{eq:SLSpres}
    \Schur_\lambda(\Q^{2n})\downarrow \Sp_{2n}\Q
    =\bigoplus_\mu \sum_\eta C_{\eta\mu}^\lambda \Schur_{\langle\mu\rangle}(\Q^{2n}),
  \end{equation}
  where the sum is over all partitions $\eta=(\eta_1=\eta_2\geq
  \eta_3=\eta_4\geq \cdots)$ where each number appears an even number
  of times. Note that this formula is independent of $n$ once $n\geq \ell(\lambda)$ (so that the formula applies). Thus the sequence 
  $\{V(\lambda)_{2n}\downarrow \Sp_{2n}\Q\}$ is representation stable. If
  $\{V_n=\bigoplus c_{\lambda,n}V(\lambda)_n\}$ is uniformly stable, let
  $N$ be the largest number of rows of any partition $\lambda$ for
  which the eventual multiplicity of $V(\lambda)_n$ is
  positive. Assume that the decomposition of $V_n$ stabilizes once
  $n\geq N$. Then for $n\geq N$, we may apply Littlewood's rule to
  conclude that \[V_n\downarrow \Sp_{2n}\Q =
  \bigoplus_\mu\sum_{\lambda,\eta}c_{\lambda,n}C^{\lambda}_{\eta\mu}V(\mu)_n\]
  is independent of $n$. We conclude that the sequence 
  $\{V_n\downarrow \Sp_{2n}\Q\}$ is uniformly representation stable.\\

  \textbf{Strong stability.} We now consider the case when
  $G_n=\SL_n\Q$ or $\GL_n\Q$ and $\{V_n\}$ and $\{U_n\}$ are strongly
  stable, meaning they are not only uniformly stable but also
  type-preserving. Recall that for $G_n=\SL_n\Q$ or $\GL_n\Q$, this
  implies Condition IV$'$: that $P_{n+1}<G_{n+1}$ acts trivially on
  $V_n\subset V_{n+1}$. This property is preserved by direct sum and
  by tensor product: if $P_{n+1}$ acts trivially on $V_n\subset
  V_{n+1}$ and on $U_n\subset U_{n+1}$, it acts trivially on
  $V_n\oplus U_n\subset V_{n+1}\oplus U_{n+1}$ and $V_n\otimes
  U_n\subset V_{n+1}\otimes U_{n+1}$. The functoriality of
  $\Schur_\lambda$ implies that if $P_{n+1}$ acts trivially on
  $V_n\subset V_{n+1}$, it acts trivially on
  $\Schur_\lambda(V_n)\subset\Schur_\lambda(V_{n+1})$. The
  constructions in Parts 1--5 are obtained by composing these
  operations, so we conclude that Condition IV$'$ holds for each of
  the resulting sequences.

  We have already proved above that the resulting sequences in Parts
  1--5 are uniformly multiplicity stable (Condition III$'$). Thus we
  may apply Proposition~\ref{prop:SLstrong} to conclude that these
  sequences are type preserving (Condition IV). Finally, by
  Remark~\ref{remark:strong} Conditions III$'$ and IV together imply
  surjectivity (Condition II). This concludes the proof of strong
  stability in Parts 1--5, and thus completes the proof of
  Theorem~\ref{thm:classicalstability}.
\end{proof}

\begin{proof}[Proof of Theorem~\ref{thm:classicalSn}]\ 

  \textbf{1.} For irreducible representations $V(\lambda)_n$ and
  $V(\mu)_n$ of $S_n$, Murnaghan proved that the decomposition of
  the tensor product $V(\lambda)_n\otimes V(\mu)_n$ into irreducibles
  $V(\nu)_n$ is eventually independent of $n$ (see Section 1 of
  \cite{Mu}), and Briand--Orellana--Rosas have recently proved that the decomposition of $V(\lambda)_n\otimes V(\mu)_n$ stabilizes once $n\geq |\lambda|+|\mu|+\lambda_1+\mu_1$ \cite[Theorem 1.2]{BOR}. If $\{V_n=\bigoplus c_{\lambda,n}V(\lambda)_n\}$ and
  $\{W_n=\bigoplus d_{\mu,n}V(\mu)_n\}$ are uniformly multiplicity
  stable, taking $n$ large enough that the decomposition of
  $V(\lambda)_n\otimes V(\mu)_n$ stabilizes for all $\lambda$ and
  $\nu$ occurring in $V_n$ and $W_n$, it follows by distributivity that
  $\{V_n\otimes W_n\}$ is uniformly multiplicity stable.\\

  \textbf{2.} For $k=1$, we repeat from Remark~\ref{remark:strong} the
  branching rule for restrictions from $S_n$ to $S_{n-1}$:
  \[V(\lambda)_n\downarrow S_{n-1}=V(\lambda)_{n-1}\oplus
  \bigoplus_\mu V(\mu)_{n-1}\] where $\mu$ ranges over those
  partitions obtained by removing one box from $\lambda$. It is
  immediate that uniform multiplicity stability is preserved. For
  restrictions from $S_n$ to $S_{n-k}$ with $k>1$ a similar formula
  can be given explicitly \cite[Exercise 4.44]{FH}, but to
  conclude stability we can just inductively apply the result for
  $k=1$.
\end{proof}

\subsection{Reversing the Clebsch--Gordan problem} We conclude this
section by discussing the possibility of reversing the conclusion of
Theorem~\ref{thm:classicalstability}(1). This idea will play an 
important role in Section~\ref{section:liealg}.

\begin{theorem}\label{thm:LRinvert}
  Let $G_n=\SL_n\Q$, $\GL_n\Q$, or $\Sp_{2n}\Q$. If
  $\{W_n\}$ and $\{V_n\otimes W_n\}$ are nonzero and multiplicity stable as
  $G_n$--representations, then $\{V_n\}$ is multiplicity stable.
  This remains true if ``multiplicity stable'' is replaced by ``uniformly multiplicity
  stable''.
\end{theorem}

\begin{proof}
  We will prove the theorem in the following form: given the irreducible decompositions of $W$ and of $V\otimes W$, the irreducible decomposition of $V$ can be determined, and without reference to $n$. This will be formalized in the course of the proof. In the remark following the proof, we sketch a constructive way to determine the decomposition of $V$. But the theorem as stated follows from more general properties, as we now explain.

First, we will use that the representation ring is a domain for any such group $G_n$. Recall that the \emph{representation ring} $R_n$ consists of formal differences $V-U$ of representations of $G_n$, with addition given by direct sum and multiplication given by tensor product. Complete reducibility implies that as a group, $R_n$ is the free abelian group on the irreducible representations. The ring structure is more complicated, but we can in fact describe $R_n$ explicitly. Indeed, let $\Lambda_n$ be
the weight lattice in $\fh^*_n$. Any representation $V$ determines a ``character'' in the group ring $\Z[\Lambda_n]$, where the coefficient of the weight $L\in \Lambda_n$ is the dimension of the eigenspace $V^{(L)}$:
\[V\mapsto \sum_{L\in \Lambda_n}\dim V^{(L)}\cdot L\]

The highest weight decomposition as described in Section~\ref{section:repsGLSp} implies that a representation is determined by its character; that is, the induced ring homomorphism $R_n\hookrightarrow  \Z[\Lambda_n]$ is injective. This would suffice for our purposes, but we can say more: any such character is invariant under the Weyl group $W_n$, and in fact $R_n$ is exactly the subring $\Z[\Lambda_n]^{W_n}$ of invariants (\cite{FH}, Theorem 23.24, combined with Exercise 23.36(d) for $\GL_n\Q$).

Since $R_n$ is a domain, $V_n$ is the unique solution in $R_n$ to the equation $x\cdot [W_n]=[V_n\otimes W_n]$. It remains to see that given the decompositions of $W_n$ and $V_n\otimes W_n$, the solution to this equation does not depend on $n$. To do this, we need to relate the representation rings $R_n$ and $R_{n+1}$.
There is a natural homomorphism $R_{n+1}\to R_n$ given by restriction from $G_{n+1}$ to $G_n$, but this is \emph{not} the map we want, since restriction does not take irreducibles to irreducibles. Instead, assume first that $G_n=\SL_n\Q$; by identifying $V(\lambda)_n$ with $\lambda$, we get an identification of $R_n$ with the free abelian group $\Z[\{\lambda|\ell(\lambda)<n\}]$ on partitions with fewer than $n$ rows. The map we want is simply the projection
\[R_{n+1}=\Z[\{\lambda|\ell(\lambda)< n+1\}]\overset{\pi}{\twoheadrightarrow}  \Z[\{\lambda|\ell(\lambda)< n\}]= R_n\]
which sends $\lambda\mapsto 0$ if $\ell(\lambda)=n$ and $\lambda\mapsto \lambda$ otherwise. With respect to this basis of partitions, multiplication in the ring is given by the Littlewood--Richardson coefficients; in a sense, Theorem~\ref{thm:classicalstability}(1) is based on the fact that this map is a ring homomorphism. This projection has a right inverse $i\colon R_n\to R_{n+1}$ defined by the inclusion ${\{\lambda|\ell(\lambda)< n\}}\subset \{\lambda|\ell(\lambda)< n+1\}$; this is not a ring homomorphism, however. Note that uniform 
multiplicity stability is equivalent to $i([V_n]) = [V_{n+1}]$ for large enough $n$.

Assume that the sequences in question are uniformly stable, and that $n$ is large enough that the decompositions of $W_n$ and $V_n\otimes W_n$ have stabilized, meaning \[i([W_n])=[W_{n+1}]\ \ \text{ and }\ \ i([V_n\otimes W_n])=[V_{n+1}\otimes W_{n+1}].\] Then we have $\pi([W_{n+1}])=[W_n]$ and $\pi([V_{n+1}\otimes W_{n+1}])=[V_n\otimes W_n]$. Thus $[V_{n+1}]$ projects to a solution of $x\cdot [W_n]=[V_n\otimes W_n]$, so by uniqueness we have $\pi([V_{n+1}])=[V_n]$. 

We want to prove that $i([V_n])=[V_{n+1}]$.  Suppose not; that is, assume the difference $[V_{n+1}]-i([V_n])$ is not zero.  Since $\pi([V_{n+1}])=[V_n]$, this difference consists of all those irreducibles $V(\lambda)_{n+1}$ contained in $V_{n+1}$ having $\ell(\lambda)=n$.  It is easy to check from the definition of the Littlewood--Richardson coefficients that if a representation $V$ contains such a $V(\lambda)_{n+1}$, then for any nonzero representation $W$ the tensor $V\otimes W$ also contains some $V(\mu)_{n+1}$ with $\ell(\mu)=n$.  Applying this to $V_{n+1}\otimes W_{n+1}$ gives that $[V_{n+1}\otimes W_{n+1}]\neq i([V_n\otimes W_n])$, contradicting the uniform stability of $V_n\otimes W_n$. We conclude that $[V_{n+1}]=i([V_n])$ and so $\{V_n\}$ is uniformly stable, as desired.

For $G_n=\Sp_{2n}\Q$, the argument proceeds identically, except that $R_n$ is identified with the free abelian group $\Z[\{\lambda|\ell(\lambda)\leq n\}]$ on partitions with at most $n$ rows. We can deduce from \eqref{eq:SpLR} the desired property that if $\ell(\lambda)={n+1}$, $V(\lambda)_{n+1}\otimes W$ contains some $V(\mu)_{n+1}$ with $\ell(\mu)={n+1}$. For $G_n=\GL_n\Q$, $R_n$ is the free abelian group $\Z[\{\lambda|\ell(\lambda)\leq n\}]$ on \emph{pseudo-}partitions with at most $n$ rows, and we also have to modify the maps between $R_n$ and $R_{n+1}$. In this case the inclusion $i\colon R_n\to R_{n+1}$ takes $\lambda=(\lambda_1\geq \cdots\geq \lambda_n)$ to $i(\lambda)=(\lambda_1\geq \cdots\geq \lambda_n\geq \lambda_n)$; the projection $\pi\colon R_{n+1}\to R_n$ sends the pseudo-partition $\lambda=(\lambda_1\geq \cdots\geq \lambda_n\geq \lambda_{n+1})$ to $0$ if $\lambda_n\neq \lambda_{n+1}$, and to $\pi(\lambda)=(\lambda_1\geq \cdots\geq \lambda_n)$ if $\lambda_n=\lambda_{n+1}$. The argument above then goes through, with the role of the partitions with $\ell(\lambda)=n$ played by the pseudo-partitions having $\lambda_n\neq \lambda_{n+1}$.

We sketch a compactness argument to extend this to the case when the sequences are multiplicity stable but not uniformly so. Consider for example the ideal of $R_n$ spanned by partitions $\lambda$ with $|\lambda|> k$. The corresponding quotients have basis the partitions $\lambda$ with $|\lambda|\leq k$.  Since this set is finite, the corresponding subset of the multiplicities converges uniformly, and we can apply the argument above. Letting $k\to \infty$, we conclude that $\{V_n\}$ is multiplicity stable.
\end{proof}

\begin{remark}
A related theorem also holds: if $\{V_n\otimes V_n\}$ is multiplicity stable, then $\{V_n\}$ is multiplicity stable, and similarly for uniform stability.  Both this claim and Theorem~\ref{thm:LRinvert} can be proved constructively, by an algorithm which we now sketch. The Littlewood--Richardson coefficients have the following property with respect to the lexicographic order on partitions: given $\lambda$ and $\mu$, the largest partition occurring in $V(\lambda)\otimes V(\mu)$ is $\lambda+\mu$, with multiplicity 1. Thus the largest partition occurring in $V_n\otimes V_n$ will be $\lambda+\lambda$, where $\lambda$ is the largest partition occurring in $V_n$. The next largest must be $\lambda+\mu$, where $\mu$ is the next largest partition in $V_n$. Continue, at each stage finding the largest irreducible in $V_n\otimes V_n$ not yet accounted for by those irreducibles already found. The algorithm pivots on partitions of the form $\lambda+\mu$, so since $\ell(\lambda)\leq \ell(\lambda+\mu)$ and $\ell(\mu)\leq \ell(\lambda+\mu)$, the steps in the algorithm will not depend on $n$ (if the sequences are not uniformly stable, we also need a compactness argument as above).
\end{remark}

\section{Cohomology of pure braid and related groups}
\label{section:braids}

Let $P_n$ denote the pure braid group, as discussed in the
introduction.  As explained there, the action of $S_n$ on the
configuration space $X_n$ makes each cohomology group
$H^i(X_n;\Q)=H^i(P_n;\Q)$ into an $S_n$--representation.  Explicit
formulas for the multiplicity of an irreducible $V(\lambda)$ in
$H^i(P_n;\Q)$ are not known.  However, we do have the following.

\begin{theorem}
\label{thm:pure}
For each fixed $i\geq 0$, the sequence of $S_n$--representations
$\{H^i(P_n;\Q)\}$ is uniformly representation stable, and in fact stabilizes once $n\geq 4i$.
\end{theorem}
\noindent For example, for $n\geq 4$,
\[H^1(P_n;\Q)=V(0)\oplus V(1)\oplus V(2)\] and thanks to computations
by Hemmer, for $n\geq 7$ we have:
\begin{equation}
\label{eq:braid2stab}
H^2(P_n;\Q)= V(1)^{\oplus 2}\oplus V(1,1)^{\oplus 2}\oplus
V(2)^{\oplus 2}\oplus V(2,1)^{\oplus 2}\oplus V(3)\oplus V(3,1).
\end{equation}

As mentioned in the introduction, it is tempting to guess that the
reason for the stability in \eqref{eq:braid2stab} is that each factor
$V(\lambda)\subset H^2(P_n;\Q)$ has $S_{n+1}$--span inside
$H^2(P_{n+1};\Q)$ isomorphic to $V(\lambda)$.  In the terminology of
\S\ref{section:strong}, we would hope that the natural homomorphism
$H^2(P_n;\Q)\to H^2(P_{n+1};\Q)$ is type-preserving and so the
sequence $\{H^2(P_n;\Q)\}$ is strongly stable. However, this is false
for $\{H^2(P_n;\Q)\}$ and indeed is false for $\{H^i(P_n;\Q)\}$ for
every $i\geq 1$.

We can see this failure explicitly for $H^1(P_n;\Q)$ as
follows. We will see below that $H^1(P_n;\Q)$ has basis
$\{w_{ij}|1\leq i<j\leq n\}$; after identifying $w_{ji}=w_{ij}$, the group $S_n$
acts on this basis by permuting the indices. Thus the unique trivial
subrepresentation $V(0)\subset H^1(P_n;\Q)$ is spanned by the vector
\[v=\sum_{1\leq i<j\leq n}w_{ij}.\] This vector is, up to a scalar,
the sum $\sum_{\sigma\in S_n}\sigma\cdot w_{12}$, and thus is
certainly $S_n$--invariant.  But after including it into $H^1(P_{n+1};\Q)$ this 
vector is not invariant under $S_{n+1}$ (for example, it does not involve any basis elements of the form $w_{i,n+1}$). In fact, it is not too
hard to check that the $S_{n+1}$--span of this vector is $V(0)\oplus
V(1)\subset H^1(P_{n+1};\Q)$.\\

In trying to prove Theorem~\ref{thm:pure}, we were able to use work of
Orlik--Solomon \cite{OS} and Lehrer--Solomon \cite{LS} to reduce the
problem to a stability statement for certain induced representations
of symmetric groups. We conjectured the following theorem to D. Hemmer
in certain special cases. The theorem was then proved by Hemmer in
much greater generality than we had hoped. This result itself provides
another example of representation stability.

We begin by presenting Hemmer's result, which we will use in the proof
of Theorem~\ref{thm:pure}. Fix a subgroup $H$ of the symmetric group
$S_k$, and fix any representation $V$ of $H$. For $n\geq k$ we may
extend the action of $H$ on $V$ to the subgroup $H\times S_{n-k}<S_n$
by letting $S_{n-k}$ act trivially on $V$; this representation of
$H\times S_{n-k}$ is denoted $V\boxtimes \Q$. Finally, we may consider
the induced representation $\Ind_{H\times S_{n-k}}^{S_n}(V\boxtimes
\Q)$, which is a representation of $S_n$.

\begin{theorem}[Hemmer \cite{He}]
  \label{thm:hemmer}
  Fix $k\geq 1$, a subgroup $H<S_k$, and a representation $V$ of $H$.
  Then the sequence of $S_n$--representations $\{\Ind_{H\times
    S_{n-k}}^{S_n}(V\boxtimes \Q)\}$ is uniformly
  representation stable. The decomposition of this sequence stabilizes once $n\geq 2k$.
\end{theorem}
\noindent Injectivity and surjectivity are immediate from the
definition of induced representation; indeed $\Ind_{H\times
  S_{n-k}}^{S_n}(V\boxtimes \Q)$ sits inside $\Ind_{H\times
  S_{n-k+1}}^{S_{n+1}}(V\boxtimes \Q)$ as the $S_n$--span of
$V\boxtimes \Q$. For the proof of uniform multiplicity stability, see
Hemmer \cite{He}.

\subsection{Stability of the cohomology of pure braid groups}
\label{section:braid}

With the above tool in hand, we can now prove representation stability
for $\{H^i(P_n;\Q)\}$ for each fixed $i\geq 0$.

\begin{proof}[Proof of Theorem~\ref{thm:pure}] We continue with the
  notation given in the introduction. The projections of configuration
  spaces $X_{n+1}\to X_n$ given by forgetting the last coordinate give
  surjections $\psi_n\colon P_{n+1}\to P_n$.  These surjections induce
  maps \[\psi_n^\ast\colon H^*(P_n;\Q)\to H^*(P_{n+1};\Q).\]
  
  We will prove representation stability with respect to these maps.
  For each pair $j\neq k$, let $w_{jk}\in H^1(P_n;\Q)$ be the
  cohomology class represented by the differential form $\frac{1}{2\pi
    i}\frac{dz_j-dz_k}{z_j-z_k}$ on $X_n\subset\C^n$.  Note that
  $w_{jk}=w_{kj}$. The vector space $H^1(P_n;\Q)$ is spanned by
  the vectors $w_{jk}$, and the map $H^1(P_n;\Q)\to H^1(P_{n+1};\Q)$
  sends $w_{jk}\in H^1(P_n;\Q)$ to $w_{jk}\in
  H^1(P_{n+1};\Q)$. Furthermore, Arnol'd proved that $H^*(P_n;\Q)$ is
  generated as a $\Q$--algebra by $H^1(P_n;\Q)$, subject only to the
  relations
  \[R_{jkl}\colon \quad w_{jk}\wedge w_{kl}+w_{kl}\wedge
  w_{lj}+w_{lj}\wedge w_{jk}=0.\] This implies that, as a vector space, $H^i(P_n;\Q)$ has basis 
    \[\big\{w_{j_1k_1}\wedge\cdots \wedge
  w_{j_ik_i}\big|k_1<\cdots<k_i, \ \mbox{and $j_m<k_m$ for all
  }m\big\}.\]
  
  Injectivity of each $\psi_n^\ast$ is then immediate.  To prove
  surjectivity of $\psi_n^\ast$ (in the sense of the definition of
  representation stability), consider an arbitrary basis element for
  $H^i(P_{n+1};\Q)$. Note that for $n\geq 2i$, no basis element can
  involve all the numbers from $1$ to $n+1$ as indices. It follows
  that by applying some element of $S_{n+1}$, we may assume that our
  basis element can be written without $n+1$ as an index. But such an
  element is in the subalgebra of $H^*(P_{n+1};\Q)$ spanned by the
  image of $H^1(P_n)$, and thus is contained in the image of
  $H^i(P_n;\Q)$, as desired.
  
  We now prove uniform stability of multiplicities; we will defer the computation of the stable range until afterwards. The work of
  Orlik--Solomon on the cohomology of hyperplane complements implies
  that $H^*(P_n;\Q)$ splits into pieces ``supported on the top
  cohomology of Young subgroups'', as follows.   For details of what follows, 
  see Lehrer--Solomon
  \cite{LS}.  Any subset of
  $\{1,\ldots,n\}$ of cardinality $k$ determines a projection $P_n\to
  P_k$ by forgetting the other $n-k$ coordinates (strands). Given a
  partition $\mathcal{S}$ of $\{1,\ldots,n\}$ into subsets, the
  product over all these projections gives a projection of $P_n$ onto
  the group $P_{\mathcal{S}}$ defined as the product of the pure braid
  groups of sizes corresponding to elements of the partition. For
  concreteness we illustrate this explicitly for the partition of
  $\{1,\ldots,n\}$ into $\{1,\ldots,k\}$ and $\{k+1,\ldots,n\}$, which
  determines a projection $P_n\to P_{\mathcal{S}}=P_k\times
  P_{n-k}$. There is always a splitting $P_{\mathcal{S}}\to P_n$,
  given in this case for example by realizing $P_k$ and $P_{n-k}$
  disjointly. Note that the partition may contain subsets of size
  1. For example, the partition of $\{1,\ldots,n\}$ into
  $\{1,\ldots,k\}, \{k+1\},\ldots,\{n\}$ determines the group
  $P_k\times P_1\times\cdots P_1\approx P_k$.

  We refer to these groups $P_\mathcal{S}$ as \emph{Young subgroups}
  of $P_n$, by analogy with Young subgroups of symmetric groups, such
  as $S_k\times S_{n-k}<S_n$. This is a slight abuse of notation,
  since the embedding of $P_\mathcal{S}$ as a subgroup is not unique;
  the important thing is the projection $P_n\to P_\mathcal{S}$. The
  projection onto such a Young subgroup gives an inclusion
  $H^*(P_\mathcal{S};\Q)\to H^*(P_n;\Q)$. We now consider the image in
  $H^*(P_n;\Q)$ of the top cohomology of $P_{\mathcal{S}}$. For
  example, the cohomological dimension of $P_k\times P_{n-k}$ is
  $(k-1)+(n-k-1)=n-2$, and we consider the image of the top cohomology
  $H^{n-2}(P_k\times P_{n-k};\Q)$ inside $H^{n-2}(P_n;\Q)$. For each
  partition $\mathcal{S}$ of $\{1,\ldots,n\}$ into $i$ subsets, the
  corresponding Young subgroup $P_{\mathcal{S}}$ is the product of $i$
  pure braid groups, so the image of its top cohomology determines a
  subspace $H^\mathcal{S}(P_n)$ of $H^{n-i}(P_n;\Q)$.  Orlik--Solomon
  \cite[Proposition 2.10]{OS} implies that $H^*(P_n;\Q)$ splits as an
  $S_n$--module as a direct
  sum \[H^*(P_n;\Q)=\bigoplus_{\mathcal{S}}H^\mathcal{S}(P_n)\] over
  all partitions $\mathcal{S}$ of $\{1,\ldots,n\}$, and that $S_n$
  permutes the summands according to its action on $\{1,\ldots,n\}$.

  Every partition $\mathcal{S}$ of $\{1,\ldots,n\}$ determines a
  partition $\overline{\mathcal{S}}$ of $n$, listing the sizes of the
  subsets in $\mathcal{S}$. The term $H^{\mathcal{S}}(P_n)$
  contributes to $H^i(P_n;\Q)$ exactly if $|\mathcal{S}|=\ell(\overline{\mathcal{S}})=n-i$. The action of $S_n$ on $\{1,\ldots,n\}$
  induces an action on partitions $\mathcal{S}$ of $\{1,\ldots,n\}$,
  and the summands $H^{\mathcal{S}}(P_n)$ are permuted according to
  this action. In particular, for a fixed $\mu\vdash n$, the direct
  sum  $\bigoplus_{\overline{\mathcal{S}}=\mu}H^{\mathcal{S}}(P_n)$ is a
  subrepresentation of $H^i(P_n;\Q)$.  We will need explicit orbit representations, so for any partition $\mu\vdash
  n$, let $\mathcal{S}_\mu$ be the partition of $\{1,\ldots,n\}$ given
  by \[\{1,\ldots,\mu_1\},\{\mu_1+1,\ldots,\mu_1+\mu_2\},
  \ldots,\{\mu_1+\cdots+\mu_{n-1}+1,\ldots,n\}.\] This gives for each $\mu$ an orbit representative
  $\mathcal{S}_\mu$ with
  $\overline{\mathcal{S}_\mu}=\mu$. For a fixed $\mu$, the
  subrepresentation
  $\bigoplus_{\overline{\mathcal{S}}=\mu}H^{\mathcal{S}}(P_n)$ is
  generated by one summand $H^{\mathcal{S}_\mu}(P_n)$ and is the direct sum of its
  translates. Thus by the definition of induced representation we have
  \begin{equation}
    \label{eq:mupiece}
    \bigoplus_{\overline{\mathcal{S}}=\mu}H^{\mathcal{S}}(P_n)
    =\Ind_{\Stab(\mathcal{S}_\mu)}^{S_n}H^{\mathcal{S}_\mu}(P_n).
  \end{equation}

  We would like to apply Theorem~\ref{thm:hemmer} to the terms
  \eqref{eq:mupiece}. Consider the projection onto the Young subgroup
  $P_n\to P_k\times P_{n-k}$. Pulling back by the projection
  $P_{n+1}\to P_n$, this pulls back to the projection $P_{n+1}\to
  P_k\times P_{n-k}\times P_1$. In general, the Young subgroup
  $P_{\mathcal{S}}<P_n$ pulls back to $P_{\mathcal{S}\langle n+1\rangle}<P_{n+1}$,
  where $\mathcal{S}\langle n+1\rangle$ is the partition $\mathcal{S}\cup\{n+1\}$. Note
  that if $\overline{\mathcal{S}}=\mu=(\mu_1,\ldots,\mu_{n-i})$, then
  $\overline{\mathcal{S}\langle n+1\rangle}=\mu\langle n+1\rangle\coloneq (\mu_1,\ldots,\mu_{n-i},1)$. For larger $m\geq n$, we  define $\mathcal{S}\langle m\rangle\coloneq \mathcal{S}\cup\{n+1\}\cup \cdots\cup \{m\}$ and $\mu\langle m\rangle\coloneq  (\mu_1,\ldots,\mu_{n-i},1,\ldots,1)$ similarly.

Since
  $\mathcal{S}\langle n+1\rangle$ is a partition of $\{1,\ldots,n+1\}$ into $(n+1)-i$
  sets, $H^{\mathcal{S}\langle n+1\rangle}(P_{n+1})$ is contained in $H^i(P_{n+1};\Q)$,
  and in fact the natural map $H^*(P_n;\Q)\to H^*(P_{n+1};\Q)$ restricts to
  an isomorphism $H^{\mathcal{S}}(P_n)\to H^{\mathcal{S}\langle n+1\rangle}(P_{n+1})$.

 Certainly not every partition of $\{1,\ldots,n+1\}$ contains the
  singleton set $\{n+1\}$. But fixing $i$, every partition of $n+1$
  with $(n+1)-i$ entries must have some entry equal to 1 once $n\geq
  2i$. This means that any such partition is equal to $\mu\langle n+1\rangle$ for some
  $\mu\vdash n$. Note that we chose the definition of
  $\mathcal{S}_\mu$ so that
  $\mathcal{S}_{\mu\langle n+1\rangle}=\mathcal{S}_\mu\langle n+1\rangle$. Thus writing the
  decomposition \[H^i(P_n;\Q)= \bigoplus_{\substack{\mu\vdash n\\
      \ell(\mu)=n-i}}
  \bigoplus_{\overline{\mathcal{S}}=\mu}H^{\mathcal{S}}(P_n)=
  \bigoplus_{\substack{\mu\vdash n\\ \ell(\mu)=n-i}}
  \Ind_{\Stab(\mathcal{S}_\mu)}^{S_n}H^{\mathcal{S}_\mu}(P_n)\] we
  have for $n\geq 2i$ a decomposition of $H^i(P_{n+1};\Q)$ over the
  same partitions $\mu$:
\begin{align*}
H^i(P_{n+1};\Q)&=\bigoplus_{\substack{\nu\vdash n+1\\\ell(\mu)=n+1-i}}
  \Ind_{\Stab(\mathcal{S}_\nu)}^{S_{n+1}}H^{\mathcal{S}_\nu}(P_{n+1})\\
&=\bigoplus_{\substack{\mu\vdash n\\\ell(\mu)=n-i}}
  \Ind_{\Stab(\mathcal{S}_\mu\langle n+1\rangle)}^{S_{n+1}}H^{\mathcal{S}_\mu\langle n+1\rangle}(P_{n+1})
\end{align*}
  We already mentioned above that
  $H^{\mathcal{S}_\mu\langle n+1\rangle}(P_{n+1})\approx
  H^{\mathcal{S}_\mu}(P_n)$. There is a bijection between the set of partitions
  $\mu\vdash n$ with $\ell(\mu)=n-i$ rows and the set of partitions of $i$, given by subtracting one box from each of the $n-i$ rows. Thus the index set for the direct sum above is independent of $n$, so it suffices to prove for each
  $\mu$ that the sequence $\{\Ind_{\Stab(\mathcal{S}_\mu\langle m\rangle)}^{S_m}H^{\mathcal{S}_\mu}(P_m)\}$
  is uniformly representation stable as $m\to \infty$.

  Note that the stabilizer $\Stab(\mathcal{S})$ of a partition
  $\mathcal{S}$ of $\{1,\ldots,n\}$ need not preserve the individual
  subsets making up $\mathcal{S}$, only the overall decomposition into
  subsets. Thus if $\mathcal{S}$ has $m_j$ subsets of size $j$, the
  stabilizer $\Stab(\mathcal{S})$ will be a product of wreath products
  $S_j\wr S_{m_j}=(S_j)^{m_j}\rtimes S_{m_j}$, where the $(S_j)^{m_j}$
  factor acts on the subsets of size $j$, and the $S_{m_j}$ factor
  permutes them. In particular, the $S_1\wr S_{m_1}=S_{m_1}$ factor
  acts by permuting the singleton sets in $\mathcal{S}_\mu$. This
  corresponds to rearranging the $P_1\times \cdots P_1$ factors in the
  Young subgroup $P_\mathcal{S}$. From this we see that the $S_{m_1}$
  factor of $\Stab(\mathcal{S})$ acts trivially on
  $H^\mathcal{S}(P_n)$.

If we write $\Stab(\mathcal{S}_\mu)=H\times
  S_{m_1}$, we have $\Stab(\mathcal{S}_\mu\langle n+1\rangle)=H\times S_{m_1+1}$, and so
  on. Take $k=n-m_1$ and let $\nu\vdash k$ be the partition obtained
  from $\mu$ by deleting those entries equal to $1$. The subgroup $H$
  is exactly $H_\nu\coloneq \Stab(\mathcal{S}_\nu)<S_k$, and
  identifying $H^{\mathcal{S}_\mu}(P_n)$ with
  $H^{\mathcal{S}_{\nu}}(P_k)$, the sequence in question can be
  written as $\{\Ind_{H_\nu\times
    S_{n-k}}^{S_n}H^{\mathcal{S}_{\nu}}(P_k)\boxtimes \Q\}$. Thus Theorem~\ref{thm:hemmer} applies and gives that this
  sequence is uniformly multiplicity stable, as desired.

To compute the stable range, it suffices to bound the number $k=|\nu|$ which appears in the last paragraph of the proof. It is not hard to check that for a fixed $i$, the maximum $k$ occurs for the partition $\mu=(2,2,\ldots,2,1,\ldots,1)$, corresponding to Young subgroups isomorphic to $P_2\times \cdots\times P_2$. For such $\mu$ we have $\nu=(2,\ldots,2)$ with $\ell(\nu)=i$, and thus the maximal $k$ is $k=2i$. Since the stability range is Theorem~\ref{thm:hemmer} is $n\geq 2k$, we conclude that $\{H^i(P_n;\Q)\}$ stabilizes once $n\geq 4i$, as claimed.
\end{proof}

\begin{remark}
  By a careful analysis of the individual pieces $H^\mathcal{S}(P_n)$,
  Lehrer--Solomon \cite{LS} decompose $H^i(P_n;\Q)$ into a direct sum
  of representations induced from 1--dimensional representations of
  certain centralizers in $S_n$. Though we did not need this
  description to prove that $\{H^i(P_n;\Q)\}$ is representation
  stable, it is indispensable when actually computing multiplicities
  of irreducibles. We revisit these multiplicities in \cite{CEF2},
  where we explicitly compute the multiplicities of certain
  irreducible representations in $H^i(P_n;\Q)$ and give arithmetic
  consequences of their stable values.
\end{remark}
\bigskip

Arnol'd \cite{Ar} (see also F. Cohen \cite{Co}) established
homological stability for the integral homology groups $H_i(B_n;\Z)$.
He also showed that $H_i(B_n;\Z)$ is finite for $i\geq 2$, so that
$H_i(B_n;\Q)$ is trivial in this range. As a corollary of
Theorem~\ref{thm:pure}, we obtain homological stability for $B_n$ with
twisted coefficients. Any representation of $S_n$ can be regarded as a
representation of $B_n$ by composing with the standard projection
$B_n\to S_n$.

\begin{corollary}
  \label{corollary:twistedstability}
  For any partition $\lambda$ the sequence $\{H_*(B_n;V(\lambda)_n)\}$
  of twisted homology groups satisfies classical homological
  stability: for each fixed $i\geq 0$, once $n\geq 4i$ there is an isomorphism
  \begin{equation}
    \label{eq:twisted}
    H_i(B_n;V(\lambda)_n) \approx H_i(B_{n+1}; V(\lambda)_{n+1}).
  \end{equation}
\end{corollary}

\begin{proof}
  There are no natural maps between the homology groups in
  (\ref{eq:twisted}), but we show that their dimension is eventually
  constant.  Since $P_n$ has finite index in $B_n$ and our
  coefficients are vector spaces over $\Q$, the transfer map gives an
  isomorphism
  \[H_i(B_n;V(\lambda)_n)\approx H_i(P_n;V(\lambda)_n)^{S_n}\] with
  the $S_n$--invariants in $H_i(P_n;V(\lambda)_n)$.  Since the action of $B_n$ on $V(\lambda)_n$
  factors through $S_n$, the representation $V(\lambda)_n$ is trivial
  when restricted to $P_n$. Thus \[H_i(P_n;V(\lambda)_n)^{S_n}\approx
  \big(H_i(P_n;\Q)\otimes V(\lambda)_n\big)^{S_n}.\] Recall from Section~\ref{section:representationstability} that every representation of $S_n$ is self-dual. Schur's lemma thus gives that
  $V(\mu)\otimes V(\nu)$ contains the trivial representation if and
  only if $\mu=\nu$, in which case the trivial representation appears
  with multiplicity 1. It follows that the dimension of
  $\big(H_i(P_n;\Q)\otimes V(\lambda)_n\big)^{S_n}$ is exactly the
  multiplicity of $V(\lambda)_n$ in $H_i(P_n;\Q)$,  which is the same
  as the multiplicity of $V(\lambda)_n$ in $H^i(P_n;\Q)$. By
  Theorem~\ref{thm:pure}, this multiplicity is constant once $n\geq 4i$, as
  desired.
\end{proof}

\begin{remark}
  The space $X_n$ is the configuration space of $n$ ordered points in 
  $\C$, and Theorem~\ref{thm:pure} states that its
  cohomology groups $\{H^i(X_n;\Q)\}$ are representation
  stable. Similarly, the space $Y_n$ is the configuration space of $n$
  unordered points in $\C$, and
  Corollary~\ref{corollary:twistedstability} gives classical
  homological stability for the sequence $\{H_i(Y_n;\Q)\}$.  These results are extended to configuration spaces of arbitrary orientable manifolds in \cite{C}.
\end{remark}

\subsection{Generalized braid groups}
\label{section:gbg}

The \emph{generalized pure braid group} of type $B_n$ is the
fundamental group $WP_n\coloneq \pi_1(X'_n)$ of the configuration
space
\[X'_n\coloneq \big\{\mathbf{z}\in \C^{2n}\big|z_i\neq z_j, z_i\neq
-z_j, z_i\neq 0\big\}.\] This configuration space is aspherical, so
$H^*(WP_n;\Q)=H^*(X'_n;\Q)$. The hyperoctahedral group $W_n$ acts on $X'_n$ by permuting and negating the
coordinates; this induces an action of $W_n$ on $H^*(WP_n;\Q)$. The
quotient $Y'_n\coloneq X'_n/W_n$ is the space of unordered $n$--tuples
of distinct sets $\{z,-z\}\subset \C$ with $z\neq 0$.  Identifying each set
$\{z,-z\}$ with the point $z^2\in \C$, the space $Y'_n$ is identified with
the space of unordered $n$--tuples of distinct nonzero points. Thus
the \emph{generalized braid group} $WB_n\coloneq \pi_1(Y'_n)$ can be
identified with the subgroup $B_{1,n}<B_{n+1}$ which is the preimage
of the stabilizer $\Stab(1)<S_{n+1}$. See the survey by Vershinin
\cite{Ve} for an overview of generalized braid groups.

\begin{theorem}
\label{thm:genpure}
For each fixed $i\geq 0$, the sequence $\{H^i(WP_n;\Q)\}$ of
$W_n$--representations is uniformly representation
stable.
\end{theorem}

\begin{proof}
  The results of Orlik--Solomon are a bit more involved in this case,
  so we cover the necessary definitions in more detail. See Douglass
  \cite{Do} for an excellent exposition of these results in the case
  of $H^*(X'_n;\Q)$. Let $\mathcal{H}_n$ be the set of hyperplanes
  defined by the equations $z_i=z_j$, $z_i=-z_j$, and $z_i=0$ for $1\leq i,j\leq n$. To each
  $H\in \mathcal{H}_n$ defined by the linear equation $L=0$ we associate
  the element $w_H\in H^1(X'_n;\Q)$ represented by the $1$-form $\frac{1}{2\pi
    i}\frac{dL}{L}$. The action of $W_n$ on $X'_n$ induces an action on $\mathcal{H}_n$.  This action  
    in turn induces an action of $W_n$ on $\{w_H:H\in\mathcal{H}_n\}$ via $\sigma\cdot w_H:=w_{\sigma\cdot H}$.

  Brieskorn \cite{Bri} proved that $H^*(WP_n;\Q)=H^*(X'_n;\Q)$ is
  generated by the $w_H$, subject to certain relations \cite[Theorem
  5.2]{OS}. Injectivity in the definition of representation stability
  follows from the naturality of these relations; this is essentially
  the observation that any relation supported on the image of
  $H^*(X'_n;\Q)$ in $H^*(X'_{n+1};\Q)$ already holds in
  $H^*(X'_n;\Q)$. For surjectivity, simply note that each generator $w_H$ involves at most two coordinates. Therefore any monomial of
  length $i$ involves at most $2i$ coordinates, and thus up to the
  $W_{n+1}$--action it is contained in $H^i(X'_n;\Q)$ once $n\geq 2i$.

  Let $\mathcal{J}$ be the set of intersections of hyperplanes in
  $\mathcal{H}_n$. The \emph{support} of a monomial $w_{H_1}\cdots
  w_{H_k}$ is the intersection $H_1\cap \cdots\cap H_k\in
  \mathcal{J}$.   
  Orlik--Solomon \cite[Proposition 2.10]{OS} prove that
  $H^*(X'_n;\Q)$ splits as a direct sum over $J\in\mathcal{J}$
  \[H^*(X'_n;\Q)=\bigoplus_{J\in\mathcal{J}}\big\langle w_{H_1}\cdots
  w_{H_k} \big|H_1\cap \cdots\cap H_k=J\big\rangle\] of the subspace
  spanned by monomials with support $J$. The factors are permuted according to the action of
  $W_n$ on $\mathcal{J}$.
  
  The codimension of an intersection $H_1\cap \cdots \cap H_k$ in $\C^n$ is always $\leq k$; a monomial $w_{H_1}\cdots w_{H_k}$ is \emph{dependent} if $\codim(H_1\cap \cdots \cap H_k)<k$.
  As pointed out in \cite[Equation 2.2]{LS}, the Orlik--Solomon relations is that any dependent monomial $w_{H_1}\cdots w_{H_k}$ vanishes in $H^*(X_n';\Q)$. Therefore, the only nonzero monomials with support $J$ are those of length $k=\codim(J)$. For a given intersection $J\in \mathcal{J}$, let $k=\codim(J)$, and let $H^J$ be the summand of $H^k(X_n';\Q)$ spanned by all monomials $w_{H_1}\cdots
  w_{H_k}$ with $H_1\cap \cdots\cap H_k=J$. Then the
  splitting above becomes the decomposition \[H^k(X'_n;\Q)=\bigoplus_{\substack{J\in\mathcal{J}\\\codim(J)=k}}H^J.\] We are grateful to Jenny Wilson for pointing out this simplification. 
  
  To write this as a sum of induced representations, we need to understand the orbits of the $W_n$--action on $\mathcal{J}$, the set of subspaces $J$ which
  occur as intersections of the defining hyperplanes $\mathcal{H}_n$.  A
  representative example is the subspace defined by the equations
  \[z_1=-z_2=-z_4,\quad z_3=z_5=-z_6,\quad z_7=z_8=0.\] In general,
  any element $J\in \mathcal{J}$ determines a grouping of the indices $1\leq i\leq n$ into disjoint blocks, where two indices $i$ and $j$ are grouped together if $z_i=\pm z_j$ holds. For example, the nontrivial blocks for the subspace above are $\{1,2,4\}$, $\{3,5,6\}$, and $\{7,8\}$. One block (possibly empty) corresponds to the $\ell$ coordinates which
  are equal to 0, for some $0\leq \ell\leq n$. In the other blocks of size $>1$,
  the indices may be split into two parts as in $z_1=-z_2=-z_4$. However, since
  $W_n$ not only can permute coordinates but can also negate them, this
  internal division is not preserved by $W_n$. The only invariants preserved by the action of $W_n$ is the number $\ell$ of coordinates equal to 0, and the partition $\lambda$ recording the sizes of the remaining nontrivial blocks (those of size $>1$).

  The orbits of $W_n$
  acting on $\mathcal{J}$  are thus in bijection with pairs $(\lambda,\ell)$ where
  $\lambda$ is a partition not involving 1 and $|\lambda|+\ell\leq
  n$. We let $m=m(J)$ denote $|\lambda|+\ell$, the number of ``restricted'' coordinates. The codimension of any such subspace is \[\codim(J)=|\lambda|-\ell(\lambda)+\ell=m(J)-\ell(\lambda).\] For example, the
  subspace above corresponds to $(\lambda,\ell)=((3,3),2)$, with $m(J)=8$ and $\codim(J)=6$. This orbit has standard representative
  \[J=J_{((3,3),2)}\colon\qquad z_1=z_2=z_3,\quad z_4=z_5=z_6,\quad
  z_7=z_8=0.\] The stabilizer $\Stab_{W_n}(J)$ of this subspace
  is $(S_3\wr W_2)\times W_2\times W_{n-8}$. All we will need is that
  in general, the stabilizer $\Stab_{W_n}(J)$ splits as a product
  $G_J\times W_{n-m(J)}$, where $n-m(J)$ is the number of unrestricted
  coordinates and $G_J$ is a group depending only on $J$.

  The splitting \[H^k(X'_n;\Q)=\bigoplus_{\substack{J\in\mathcal{J}\\\codim(J)=k}}H^J\] can be
  rewritten as a sum over $W_n$--orbit representatives of induced
  representations \[H^k(X'_n;\Q)=\bigoplus_{\substack{J=J_{(\lambda,\ell)}\\m(J)-\ell(\lambda)=k}}
  \Ind_{\Stab_{W_n}(J)}^{W_n}H^J.\]

  For fixed $k$, the set of pairs $(\lambda,\ell)$ with
  $|\lambda|-\ell(\lambda)+\ell=k$ is finite and eventually
  independent of $n$; that is, the collection of orbit representatives
  $J_{(\lambda,\ell)}$   does not depend on $n$. Thus it  suffices to prove that for each
  such $(\lambda,\ell)$, the sequence of representations
  $\Ind_{\Stab_{W_n}(J_{(\lambda,\ell)})}^{W_n}H^{J_{(\lambda,\ell)}}$ is uniformly representation stable.

  We now mimic Hemmer's proof of Theorem~\ref{thm:hemmer} to finish
  the proof. Fix $J=J_{(\lambda,\ell)}$, set $m=m(J)$ and take $n>m(J)$. Recall that $\Stab_{W_n}(J)$ splits as
  $G_J\times W_{n-m}$, where $n-m=n-m(J)$ is the number of unrestricted
  coordinates. Since $H^{J}$ is spanned by monomials whose support is $J$, no
  unrestricted coordinate appears in any element of $H^{J}$, so
the $W_{n-m}$ factor above acts trivially on
  $H^{J}$. Thus we may consider $H^{J}$ as the representation
  $H^{J}\boxtimes \Q$ of $G_J\times W_{n-m}$. Factor the desired
  induction as
  \begin{align*}
    \Ind_{\Stab_{W_n}(J)}^{W_n}H^{J}&=
    \Ind_{W_m\times W_{n-m}}^{W_n}
    \Ind_{G_J\times W_{n-m}}^{W_m\times W_{n-m}}
    H^{J}\boxtimes \Q\\
    &=\Ind_{W_m\times W_{n-m}}^{W_n}
    \left(\big(\Ind_{G_J}^{W_m}H^{J}\big)
    \boxtimes \Q\right)
  \end{align*}

  Let $V_{J}$ be the representation $\Ind_{G_J}^{W_m}H^{J}$ of
  $W_m$; note that this makes no reference to $n$. Consider the
  decomposition of $V_{J}$ into irreducible representations
  $V_{(\lambda^+,\lambda^-)}$ of $W_m$. Since only finitely many
  irreducibles $V_{(\lambda^+,\lambda^-)}$ occur in $V_{J}$, it suffices
  to prove uniform stability for each factor $\Ind_{W_m\times
    W_{n-m}}^{W_n}\big(V_{(\lambda^+,\lambda^-)} \boxtimes
  \Q\big)$. The Littlewood--Richardson rule generalizes to
  hyperoctahedral groups as \cite[Lemma 6.1.3]{GP}, giving:
  \[\Ind_{W_m\times W_{n-m}}^{W_n}\big(V_{(\lambda^+,\lambda^-)}
  \boxtimes
  V_{(\mu^+,\mu^-)}\big)=\bigoplus_{\nu^+,\nu^-}
  C_{\lambda^+,\mu^+}^{\nu^+}
  C_{\lambda^-,\mu^-}^{\nu^-}
  V_{(\nu^+,\nu^-)}.\] Applying this to the trivial
  representation $\Q=V_{((n-m),(0))}$ yields
  \[\Ind_{W_m\times W_{n-m}}^{W_n}\big(V_{(\lambda^+,\lambda^-)}
  \boxtimes \Q\big)=\bigoplus_{\nu} C_{\lambda^+,(n-m)}^{\nu}
  V_{(\nu,\lambda^-)}=\bigoplus_{\nu} V_{(\nu,\lambda^-)}\] where the
  last sum is over those partitions $\mu$ obtained from $\lambda^+$ by
  adding $n-m$ boxes, no two in the same column. For fixed $\lambda^+$
  and large enough $n$, say $n-m>|\lambda^+|$, any such $\nu$ must
  have multiple boxes added to the first row. This yields a bijection
  between the partitions $\nu$ of $j\coloneq n-|\lambda^-|$ appearing
  in this decomposition and their stabilizations $\nu[j+1]$ appearing
  in the decomposition \[\Ind_{W_m\times
    W_{n+1-m}}^{W_{n+1}}\big(V_{(\lambda^+,\lambda^-)} \boxtimes
  \Q\big)=\bigoplus_{\nu} V_{(\nu[j+1],\lambda^-)},\] implying that
  this sequence of induced representations is uniformly multiplicity
  stable. This completes the proof that $H^i(WP_n;\Q)=H^i(X'_n;\Q)$ is
  uniformly representation stable.
\end{proof}

This gives the following corollary, by the same argument as Corollary~\ref{corollary:twistedstability}. An explicit stability range for Theorem~\ref{thm:genpure} and Corollary~\ref{cor:gen} can be extracted from the proof of Theorem~\ref{thm:genpure}.
\begin{corollary}
\label{cor:gen}
  For any double partition $\lambda=(\lambda^+,\lambda^-)$ the sequence $\{H_*(B_{1,n};V(\lambda)_n)=H_*(WB_n;V(\lambda)_n)\}$
  of twisted homology groups satisfies classical homological
  stability: for each fixed $i\geq 0$ and sufficiently large $n$ (depending on $i$), there is an isomorphism
    \[H_i(B_{1,n};V(\lambda)_n) \approx H_i(B_{1,n+1}; V(\lambda)_{n+1}).\]
\end{corollary}

\subsection{Groups of string motions}
A 3--dimensional analogue of the pure braid group is $P\Sigma_n$, the
group of pure string motions. Let $X''_n$ be the space of embeddings
of $n$ disjoint unlinked loops into 3--space, and define
$P\Sigma_n\coloneq \pi_1(X''_n)$. The hyperoctahedral group $W_n$ acts 
on $X''_n$ by permuting the labels and reversing the orientations,
inducing an $W_n$--action on $H^*(P\Sigma_n;\Q)$. We remark that
$P\Sigma_n$ can also be identified with McCool's \emph{pure symmetric
  automorphism group}, consisting of those automorphisms of the free
group $F_n$ sending each generator to a conjugate of itself. The
quotient $Y''_n\coloneq X''_n/W_n$ is the space of $n$ unordered
unoriented unlinked loops, and its fundamental group
$B\Sigma_n\coloneq \pi_1(Y''_n)$ is the group of \emph{string
  motions}.

The cohomology ring of $P\Sigma_n$ has been computed by
Jensen--McCammond--Meier \cite{JMM}, who prove that $H^*(P\Sigma_n;\Q)$
is generated by classes $\alpha_{ij}\in H^1(P\Sigma_n;\Q)$ for all $i\neq
j$, $1\leq i,j\leq n$, subject to the relations
\[\alpha_{ij}\wedge \alpha_{ji}=0\qquad\text{and}\qquad
\alpha_{ij}\wedge \alpha_{jk}+\alpha_{kj}\wedge
\alpha_{ik}+\alpha_{ik}\wedge \alpha_{ij}.\] The action of $W_n$ is as
follows: $S_n$ acts by permuting the indices, while negating the $j$th
coordinate negates generators of the form $\alpha_{ij}$ and fixes all
other generators. There is a natural embedding $P_n\hookrightarrow
P\Sigma_n$, and the induced surjection $H^*(P\Sigma_n;\Q)\to
H^*(P_n;\Q)$ maps $\alpha_{ij}\mapsto w_{ij}$. Based on the results of
the previous sections, it is natural to make the following conjecture.

\begin{conjecture}
\label{conjecture:string}
  For each fixed $i\geq 0$, the sequence of $W_n$--representations
  $\{H^i(P\Sigma_n;\Q)\}$ is uniformly representation
  stable.\footnote{Since this manuscript was originally posted, this conjecture has been proved by J.\ Wilson \cite{Wi1}.}
\end{conjecture}

Note that some element of $W_n$ negates $\alpha_{ij}$ and preserves
$\alpha_{ji}$, but both are mapped to $w_{ij}=w_{ji}\in H^1(P_n;\Q)$. Thus the action
of $W_n$ on $H^*(P\Sigma_n;\Q)$ does not descend to the action of
$S_n$ on $H^*(P_n;\Q)$, though of course the restricted action of
$S_n$ on $H^*(P\Sigma_n;\Q)$ does.

\section{Lie algebras and their homology}
\label{section:liealg}

In this section we show how the phenomenon of representation stability
occurs in the theory of Lie algebras.  Our main result,
Theorem~\ref{thm:equivhomLie} below, relates representation stability
for a sequence of Lie algebras to representation stability for their
homology groups. We then give a number of applications, some of which
were already known by other methods.

\subsection{Graded Lie algebras and Lie algebra homology}
\para{Lie algebra homology} Given a Lie algebra $\L$ over $\Q$, its
\emph{Lie algebra homology} $H_*(\L;\Q)$ is computed by the chain
complex

\begin{equation}
  \label{eq:HL}
  \cdots\longrightarrow \bwedge^3\L
  \overset{\partial_3}{\longrightarrow} \bwedge^2\L
  \overset{\partial_2}{\longrightarrow} \L
  \overset{\partial_1}{\longrightarrow} \Q,
\end{equation}
where the differential is given by
\[\partial_i(x_1\wedge\cdots\wedge x_i)=
\sum_{j<k}(-1)^{j+k+1}[x_j,x_k]\wedge x_1 \wedge\cdots\wedge
\widehat{x}_j\wedge\cdots \wedge \widehat{x}_k\wedge\cdots\wedge
x_i.\] Note that if $\GL(\L)$ denotes the group of Lie algebra
automorphisms of $\L$, the induced action of $\GL(\L)$ on $\bwedge^i
\L_n$ commutes with $\partial$. Thus an action of any group
$G$ on $\L$ by automorphisms induces an action of $G$ on
$H_i(\L;\Q)$ for each $i$.

\para{Homology with coefficients} If $M$ is an $\L$--module, the
\emph{homology} $H_*(\L;M)$ \emph{with coefficients in $M$} is the
homology of the complex
\[\cdots\to\bwedge^3\L\otimes M\to \bwedge^2\L\otimes M
\to\L\otimes M\to M\to 0,\] where the differential is the sum of the
previous differential on $\bwedge^*\L$, extended by the identity to
$\bwedge^*\L\otimes M$, plus $\partial'_i\colon
\bwedge^i\L\otimes M\to \bwedge^{i-1}\L\otimes M$ defined by
\[\partial'_i(x_1\wedge\cdots\wedge x_i\otimes m)=
\sum(-1)^{j+1}x_1\wedge\cdots\wedge\widehat{x}_j \wedge\cdots\wedge
x_i\otimes x_j\cdot m .\] A common example is the \emph{adjoint
  homology} $H_*(\L;\L)$. If $G$ acts on $\L$ by automorphisms and
acts $\L$--equivariantly on $M$, meaning that $(g\cdot x)\cdot (g\cdot
m)=g\cdot(x\cdot m)$, then $\partial'$ commutes with the action of
$G$, inducing an action of $G$ on $H_i(\L;M)$ for each $i$.

\para{Graded Lie algebras} A Lie algebra $\L$ is called a \emph{graded
  Lie algebra} if it decomposes into homogeneous components
$\L=\bigoplus_{j\geq 1} \L^j$ so that $[\L^j,\L^k]\subset
\L^{j+k}$. This induces a grading $\bwedge^i \L=\bigoplus_j
(\bwedge^i \L)^j$ under which, for example, the subspace $\bwedge^3
\L^2\subset\bwedge^3 \L$ has degree 6. From the definition above we
see that the differential $\partial$ preserves this grading. Thus it
descends to a grading $H_i(\L;\Q)=\bigoplus_j H_i(\L;\Q)^j$ of the Lie
algebra homology. If $M=\bigoplus_{j\geq 0} M^j$ is a graded
$\L$--module, meaning that $\L^j\cdot M^k\subset M^{j+k}$, then we
similarly obtain a grading $H_i(\L;M)=\bigoplus H_i(\L;M)^j$.

\begin{definition}[Consistent sequence of Lie algebras]
  Let $G_n$ be $\SL_n\Q$, $\GL_n\Q$, or $\Sp_{2n}\Q$. Consider a
  sequence of Lie algebras $\{\L_n\}$ with injections
  $\L_n\hookrightarrow \L_{n+1}$ and with each $\L_n$ equipped with an
  action of $G_n$ by Lie algebra automorphisms. We call the sequence
  $\{\L_n\}$ \emph{consistent} if each of the following holds:
\begin{enumerate}
\item $\L_n$ is consistent when considered as a sequence of
  $G_n$--representations.
\item Each $\L_n$ is graded, and both the maps $\L_n\to \L_{n+1}$ and
  the action of $G_n$ preserve the grading.
\item The graded components $\L_n^j$ are finite-dimensional.
\end{enumerate}
\end{definition}
\noindent
It will also be useful to allow our coefficient modules to vary with $n$.

\begin{definition}[Admissible coefficients]
A sequence $\{M_n\}$ of nonzero graded $\L_n$--modules with maps $M_n\to M_{n+1}$ and
  equivariant $G_n$--actions is \emph{admissible} if the following conditions hold:  
  for each $j\geq 0$  the sequence $\{M_n^j\}$ is strongly stable; each $M_n^j$ is 
  finite-dimensional; and $M_n^j$ is eventually nonzero for at least one $j\geq 0$.
  \end{definition}

Our main result in this section is the following.  It proves among
other things that strong representation stability for a sequence of
Lie algebras is actually equivalent to strong stability for its 
homology.  Each direction of this equivalence has applications.

\begin{theorem}[Stability of Lie algebras and their homology]
\label{thm:equivhomLie}
Let $G_n=\SL_n\Q$ or $\GL_n\Q$, and let $\{\L_n\}$ be a consistent 
sequence of graded Lie algebras with $G_n$--actions which is type-preserving
(satisfies Condition IV). The following are equivalent:
  \begin{enumerate}
  \item For each fixed $j\geq 0$ the sequence $\{\L_n^j\}$ is strongly
    stable.
  \item For each fixed $i,j\geq 0$ the sequence $\{H_i(\L_n;\Q)^j\}$
    is strongly stable.
  \item For each fixed $i,j\geq 0$ the sequence $\{H_i(\L_n;\L_n)^j\}$
    of graded adjoint homology groups is strongly stable.
  \item For one admissible sequence of coefficients $\{M_n\}$, the
    sequence $\{H_i(\L_n;M_n)^j\}$ is strongly stable for each fixed
    $i,j\geq 0$.
  \item For every admissible sequence of coefficients $\{M_n\}$, the
    sequence $\{H_i(\L_n;M_n)^j\}$ is strongly stable for each fixed
    $i,j\geq 0$.
  \end{enumerate}
\end{theorem}
\begin{proof}
  We will prove the equivalence for $G_n=\SL_n\Q$ and $\GL_n\Q$
  simultaneously. Note that by taking coefficients $M=\Q$ concentrated
  in grading 0 with trivial $\L$--action, (2) is a special case of (4), so 
  (2) $\implies$ (4). We will begin by proving that
  {(4) $\implies$ (1)}. We will then modify this argument slightly to
  prove that (3) $\implies$ (1). Note that under the assumption of
  (1), $\{\L_n\}$ is an admissible sequence of coefficients, so that
  under this assumption (3) follows from (5). Thus once we have proved
  that (1) $\implies$ (5), it immediately follows that (1) $\implies$
  (3). Since (5) also trivially implies (2) and (4), this will
  complete the proof of equivalence.
  
  \bigskip We first describe the complex computing graded
  homology. Let $\L=\bigoplus_{j\geq 1} \L^j$ be a graded Lie algebra,
  and $M=\bigoplus_{j\geq 0}M^j$ a graded $\L$--module. Since the
  differential preserves the grading, we can decompose the complex
  \eqref{eq:HL} computing $H_*(\L;M)$ into its graded pieces. The
  slice of this complex in grading $k$, which computes
  $H_*(\L;M)^k$, has the form:
  \begin{equation}
    \label{eq:gradedLhom}
    \begin{split}
      0\longrightarrow \bwedge^k\L^1\otimes
      M^0\longrightarrow\big(\bwedge^{k-2}\L^1\otimes \L^2\otimes
      M^0\big)\oplus\big(\bwedge^{k-1}\L^1\otimes
      M^1\big)\longrightarrow\cdots \\\cdots\longrightarrow
      \bigoplus_{1\leq j,j'<k}\L^j\otimes \L^{j'}\otimes
      M^{k-j-j'}\longrightarrow \bigoplus_{1\leq j\leq k}\L^j\otimes
      M^{k-j}\longrightarrow M^k\to 0
    \end{split}
  \end{equation}
Here the $\L^j\otimes \L^{j'}\otimes M^{k-j-j'}$ term is actually
  $\bwedge^2 \L^j\otimes M^{k-2j}$ when $j=j'$. Note the following key
  property: the graded piece $\L^k$ only appears in the second-to-last
  term, in the term $\L^k\otimes M^0$; all other terms involve $\L^j$
  only for smaller $j$ (for $j<k$).

  \bigskip \textbf{(4) $\implies$ (1).} Assume that $\{\L_n\}$ is a
  consistent $G_n$--sequence of graded Lie algebras, that $\{M_n\}$ is
  an admissible $G_n$--sequence of graded $\L_n$--modules, and that
  $\{H_i(\L_n;M_n)^j\}$ is strongly stable for each fixed $i,j\geq
  0$.   We first prove the result in the special case that $M_n^0$ is eventually nonzero; assume this for now.  
  
  We prove that
  $\{\L_n^j\}$ is strongly stable by induction. Since we have assumed
  that $\L_n\hookrightarrow \L_{n+1}$ is type-preserving, it suffices
  to prove that $\{\L_n^j\}$ is uniformly multiplicity stable.  In
  the next two sections of the proof only, we will
  abbreviate ``uniformly multiplicity stable'' to
  ``stable''. Furthermore, since stability is always taken over
  sequences with respect to $n$, we suppress the subscript $n$ for
  readability. To sum up, within the next two sections
  ``$\{H_1(\L;M)^1\}$ is stable'' means ``the sequence of
  $G_n$--representations $\{H_1(\L_n;M_n)^1\}$ is uniformly
  multiplicity stable''.

  We first prove that $\{\L^1\}$ is stable. The following sequence
  computes $H_*(\L;M)^1$:
  \[0\to \L^1\otimes M^0\overset{\partial}{\longrightarrow} M^1\to 0\]
  By assumption $\{M^1\}$ is stable, as are $\{\ker
  \partial=H_1(\L;M)^1\}$ and $\{M^1/\im \partial=H_0(\L;M)^1\}$. We see
  that $\{\im \partial\}$ is also stable, and thus $\{\L^1\otimes
  M^0=\ker\partial \oplus\im\partial\}$ is stable as well. We now appeal
  to Theorem~\ref{thm:LRinvert}, which states that if $\{M^0\}$ and
  $\{\L^1\otimes M^0\}$ are stable, then $\{\L^1\}$ is stable as well.

  The argument in the inductive step is similar. Assume that
  $\{\L^j\}$ is stable for each $j<k$. Consider the sequence
  \eqref{eq:gradedLhom} computing $H_*(\L;M)^k$.  Since $\{M^j\}$ is
  stable for each fixed $j\geq 0$, by repeatedly applying
  Theorems~\ref{thm:classicalstability}(1) and
  \ref{thm:classicalstability}(2) we conclude that each term of
  \eqref{eq:gradedLhom} is stable except possibly the term
  $\{\L^k\otimes M^0\}$. We now proceed along this complex from left
  to right, comparing the complex itself with its homology. Start with
  $\partial_k$, whose domain is $\bwedge^k\L^1\otimes M^0$.  Since
  $\{\bwedge^k\L^1\otimes M^0\}$ and $\{\ker\partial_k=H_k(\L;M)^k\}$
  are stable, so is $\{\im\partial_k\}$.  Since $\{\im\partial_k\}$
  and $\{\ker\partial_{k-1}/\im\partial_k=H_{k-1}(\L;M)^k\}$ are
  stable, so is $\{\ker\partial_{k-1}\}$.  Since
  $\{\ker\partial_{k-1}\}$ and the domain of $\partial_{k-1}$ are
  stable, so is $\{\im\partial_{k-1}\}$. Continuing along the complex, we have by induction that
  $\{\im\partial_2\}$ is stable, as is
  $\{\ker\partial_1/\im\partial_2=H_1(\L;M)^k\}$, so
  $\{\ker\partial_1\}$ is stable. Now moving to the right side,
  $\{M^k\}$ and $\{M^k/\im\partial_1=H_0(\L;M)^k\}$ are stable, so
  $\{\im\partial_1\}$ is stable. Combining these claims, we see that
  $\{\ker\partial_1\oplus\im\partial_1=\bigoplus_{1\leq j\leq
    k}\L^j\otimes M^{k-j}\}$ is stable. Since all but one term in this
  sum is stable, the remaining term $\{\L^k\otimes M^0\}$ is stable as
  well.  Applying Theorem~\ref{thm:LRinvert}, we conclude that
  $\{\L^k\}$ is stable.

  In the previous two paragraphs we assumed that $M^0_n$ was
  eventually nonzero, but a similar argument applies in general. For
  example, consider the case when $M^0_n$ is zero for all $n$, but
  $M^1_n$ is eventually nonzero. Then every term containing $M^0$
  vanishes in the complex \eqref{eq:gradedLhom} computing
  $H_*(\L;M)^k$. Among the remaining terms, $\L^k$ no longer appears,
  and $\L^{k-1}$ appears only in the term $\L^{k-1}\otimes M^1$. Thus
  assuming that $\{L^j\}$ is stable for $j<k-1$, an argument like the
  one above shows that $\{\L^{k-1}\}$ is stable. This completes the
  proof that (4) $\implies$ (1).

  \bigskip \textbf{(3) $\implies$ (1).} This proof is exactly like the
  proof that (4)$\implies$(1), except that we do not know at the
  beginning that $\L_n$ is an admissible sequence of coefficients. To
  prove this, first note that the complex computing $H_1(\L;\L)^1$ is
  just $0\to \L^1\to 0$, so $\L^1$ must be stable. Since $\L$ has
  positive grading, the complex computing $H_k(\L;\L)^k$ has the form
  \begin{equation*}
    0\longrightarrow
    {\bwedge^{k-1}\L^1}\otimes\L^1\to\cdots\to
    \bigoplus_{0<j<k}\L^j\otimes \L^{k-j}\longrightarrow
    \L^k\longrightarrow 0
  \end{equation*} In particular, $\L^k$
  appears only in the last term. By induction, every term except
  possibly the last is stable, and the homology in each dimension is
  stable, so as above we can conclude that $\L^k$ is stable, as
  desired. Note that we do not need a separate argument for the case
  when $\L^1$ is trivial.

  \bigskip \textbf{(1) $\implies$ (5).} Assume that $\{\L_n^j\}$ is
  strongly stable for each $j\geq 0$. Let $N_n^{i,j}$ be the piece of
  $\bwedge^i\L_n\otimes M_n$ in grading $j$, so that the complex
  \eqref{eq:gradedLhom} computing $H_*(\L_n;M_n)^k$ has the form
  \[0\to N_n^{k,k}\to N_n^{k-1,k}\to \cdots\to N_n^{2,k}\to
  N_n^{1,k}\to N_n^{0,k}\to 0.\] We have already encountered these
  subspaces; for example, $N_n^{k,k}=\bwedge^k\L_n^1\otimes M_n^0$ and
  $N_n^{1,k}=\bigoplus_{1\leq j\leq k}\L_n^j\otimes M_n^{k-j}$. We
  have already assumed that $\{\L_n\}$ and $\{M_n\}$ are strongly
  stable. If both are finite-dimensional, then by
  Theorems~\ref{thm:classicalstability}(1) and
  \ref{thm:classicalstability}(2), the sequence
  $\{\bwedge^i\L_n\otimes M_n\}$ is strongly stable for each fixed
  $i\geq 0$. Even if $\L_n$ is not finite-dimensional, for fixed $i,j$
  the term $N_n^{i,j}$ only involves finitely many graded pieces
  $\L_n^\bullet$ and $M_n^\bullet$, as in the example
  $N_n^{1,k}=\bigoplus_{1\leq j\leq k}\L_n^j\otimes M_n^{k-j}$
  above. Since each graded piece is assumed finite-dimensional, we may
  repeatedly apply Theorem~\ref{thm:classicalstability} to conclude
  that $\{N_n^{i,j}\}$ is strongly stable, and in particular satisfies
  Condition IV.

  Let $\partial_i^n$ be the differential $\bwedge^i\L_n\otimes
  M_n\to \bwedge^{i-1}\L_n\otimes M_n$, and let $(\partial_i^n)^j\colon N_n^{i,j}\to N_n^{i-1,j}$ be the restriction to $N_n^{i,j}$. The
  commutativity of
  \begin{equation}
    \label{eq:wedgeLnsquare}
    \xymatrix{
      \bwedge^i\L_n\otimes M_n\ar^{\partial_i^n}[r]\ar[d]&\bwedge^{i-1}\L_n\otimes M_n\ar[d]\\
      \bwedge^i\L_{n+1}\otimes M_{n+1}\ar_{\partial_i^{n+1}}[r]&\bwedge^{i-1}\L_{n+1}\otimes M_{n+1}
    }
  \end{equation}
  implies that under the vertical inclusions, $\ker \partial_i^n$ maps to $\ker
  \partial_i^{n+1}$, and that $\im\partial_i^n$ maps to $\im \partial_i^{n+1}$. Restricting to grading $j$, we similarly conclude that $\ker (\partial_i^n)^j$ maps to $\ker
  (\partial_i^{n+1})^j$ and that $\im(\partial_i^n)^j$ maps to $\im(\partial_i^{n+1})^j$ under the inclusions $N_n^{i,j}\hookrightarrow N_{n+1}^{i,j}$.
  
  Recall that Condition IV for $\{N_n^{i,j}\}$ says that for any
  subspace isomorphic to $V(\lambda)_n^k$ in $N_n^{i,j}$, its
  $G_{n+1}$--span in $N_{n+1}^{i,j}$ is isomorphic to
  $V(\lambda)_{n+1}^k$. Applying this to $\ker(\partial_i^n)^j$ and $\im(\partial_i^n)^j$, the observation above implies that for fixed $i$, $j$ and
  $\lambda$, the multiplicity of $V(\lambda)_n$ in
  $\ker(\partial_i^n)^j$ and in $\im(\partial_i^n)^j$ is nondecreasing in $n$. The sum of these representations is $N_n^{i,j}$, whose decomposition is eventually constant by uniform multiplicity stability. Once the decomposition of $N_n^{i,j}$ has stabilized, an increase in $\ker(\partial_i^n)^j$ would necessitate a corresponding decrease in $\im(\partial_i^n)^j$, contradicting the observation for $\im(\partial_i^n)^j$, and vice versa. We conclude that $\{\ker(\partial_i^n)^j\}$ and $\{\im(\partial_i^n)^j\}$ are uniformly multiplicity stable for each $i$ and $j$, stabilizing once $N_n^{i,j}$ does. Thus for each $i$ and $j$ the
  quotient $\{H_i(\L_n;M_n)^j=\ker
  (\partial_{i-1}^n)^j/\im(\partial_i^n)^j\}$ is uniformly multiplicity stable, as desired.
  
  Since $\{N_n^{i,j}\}$ is uniformly
  multiplicity stable, for fixed $i,j\geq 0$ and sufficiently large
  $n$ we have the following property: only finitely many partitions
  $\lambda$ occur in $N_n^{i,j}$ (meaning the multiplicity of
  $V(\lambda)_n$ in $N_n^{i,j}$ is nonzero). This property passes to
  the subquotient $H_i(\L_n;M_n)^j$. But a sequence
  $\{H_i(\L_n;M_n)^j\}$ which is multiplicity stable yet involves only
  finitely many irreducibles is necessarily uniformly multiplicity
  stable, and so we can promote Condition III to Condition III$'$.

  By assumption $\L_n$ and $M_n$ are strongly stable. By
  Proposition~\ref{prop:SLstrong}, this implies that $P_{n+1}$ acts
  trivially on the image of $\L_n$ in $\L_{n+1}$ and of $M_n$ in
  $M_{n+1}$. As noted above, this implies that $P_{n+1}$ acts
  trivially on the image of $\bwedge^i\L_n\otimes M_n$ in
  $\bwedge^i\L_{n+1}\otimes M_{n+1}$ for each $i$. But this condition
  passes to subquotients, so $P_{n+1}$ acts trivially on the image of
  $H_i(\L_n;M_n)$ in $H_i(\L_{n+1};M_{n+1})$, verifying Condition IV
  for the sequences $\{H_i(\L_n;M_n)\}$ and
  $\{H_i(\L_n;M_n)^j\}$. Since the $N_n^{i,j}$ are finite-dimensional,
  the same is true of their subquotients $H_i(\L_n;M_n)^j$. By
  Remark~\ref{remark:strong}, for a finite-dimensional sequence
  Conditions III$'$ and IV together imply Conditions I and II.
This concludes the proof of strong stability of $\{H_i(\L_n;M_n)^j\}$.
\end{proof}

\para{Symplectic Lie algebras}
In the proof of (1) $\implies$ (5) of Theorem~\ref{thm:equivhomLie} 
we used the assumption that
$\{\L_n\}$ is type-preserving. For representations of $\Sp_{2n}\Q$ we
do not have the appropriate analogue of
Proposition~\ref{prop:SLstrong}, so the argument does not work in this
case. But examining the proof above, we did not use that $\{\L_n\}$ is
type-preserving in the implications (4) $\implies$ (1) or (3)
$\implies (1)$. We needed only Theorem~\ref{thm:LRinvert} and Theorem~\ref{thm:classicalstability} for uniform multiplicity stable
sequences, and these theorems apply to $\Sp_{2n}\Q$--representations
as well. Thus we deduce the following from the proof of
Theorem~\ref{thm:equivhomLie}.

\begin{theorem}
\label{thm:equivSpLie}
Let $\{\L_n\}$ be a consistent $\Sp_{2n}\Q$--sequence of graded Lie
algebras. If the sequence $\{H_i(\L_n;\Q)^j\}$ is uniformly
multiplicity stable for each fixed $i,j\geq 0$, then the sequence
$\{\L_n^j\}$ is uniformly multiplicity stable for each fixed $j\geq
0$.
\end{theorem}

\subsection{Simple representation stability}
Certain classical families of representations satisfy a stronger form
of stability, which is in some sense as close to actual stability as a
sequence of $\GL_n\Q$--representations can be. Consider a partition
$\lambda$ with $\ell=\ell(\lambda)$ rows. As noted above, $\Schur_\lambda(\Q^n)$ is
trivial for $n<\ell$, and for such $n$ there is no irreducible
representation which could be called $V(\lambda)_n$. A sequence is
called simply representation stable if this is the only obstruction to
having constant multiplicities.
\begin{definition}[Simple representation stability] A consistent
  sequence $\{V_n\}$ of $\GL_n\Q$--representations is called
  \emph{simply representation stable} if for all $n\geq 1$ it
  satisfies Conditions I and II, and if in addition it satisfies the
  following:
  \begin{enumerate}
  \item[{\bf SIII.}] For each partition $\lambda$ with $\ell=\ell(\lambda)$ nonzero
    rows, the multiplicity of the irreducible representation
    $\Schur_\lambda(\Q^n)$ in $V_n$ is constant for all $n\geq \ell$. For
    any pseudo-partition $\lambda=(\lambda_1\geq \cdots\geq
    \lambda_\ell)$ which is not a partition (meaning $\lambda_\ell<0$), the
    multiplicity of $V(\lambda)_n$ in $V_n$ is 0.
  \item[{\bf SIV.}] For any subrepresentation $W\subset V_n$ so that
    $W\approx \Schur_\lambda(\Q^n)$, the span of the $\GL_{n+1}\Q$--orbit of $\phi_n(W)$ is
    isomorphic to $\Schur_\lambda(\Q^{n+1})$.
  \end{enumerate}
\end{definition}
If we interpret $V(\lambda)_n$ as being trivial when $n$ is less than
the number of rows of $\lambda$, then simple representation stability
says there is a decomposition
\[V_n=\bigoplus c_{\lambda}\Schur_\lambda(\Q^n)=\bigoplus c_{\lambda}V(\lambda,0)_n\] over partitions $\lambda$
which is totally independent of $n$ and preserved by the maps
$V_n\hookrightarrow V_{n+1}$. Then Theorem~\ref{thm:equivhomLie} has
the following strengthening.
\begin{theorem}
  \label{thm:simplehomLie}
  Let $\{\L_n\}$ be a consistent $\GL_n\Q$--sequence of graded Lie
  algebras which is type-preserving, 
  and $\{M_n\}$ an admissible sequence of coefficients which is simply stable. Then
  Theorem~\ref{thm:equivhomLie} remains true if ``strongly stable'' is
  replaced everywhere by ``simply stable''.
\end{theorem}
\begin{proof} We sketch the proof. The characterization of
  Proposition~\ref{prop:SLstrong} still holds: given Conditions I, II,
  and SIII, Condition SIV is equivalent to Condition IV$'$. Examining
  the proof of Theorem~\ref{thm:classicalstability}, and in particular
  that the formulas \eqref{eq:LR} and \eqref{eq:SLpleth} are
  independent of $n$, we conclude that if $\{V_n\}$ and $\{U_n\}$ are
  simply stable, the same is true of $\{V_n\otimes U_n\}$ and
  $\{\Schur_\lambda(V_n)\}$. Similarly, from the proof of
  Theorem~\ref{thm:LRinvert}, we conclude that if $\{U_n\}$ and
  $\{V_n\otimes U_n\}$ satisfy Condition SIII, the same is true of
  $\{V_n\}$. This has the folllowing implications for the proof of
  Theorem~\ref{thm:equivhomLie}.

  In the proofs of (4) $\implies$ (1) and of (3) $\implies$ (1) we can
  replace ``stable'' with ``simply stable'' everywhere, and the
  argument remains valid. For (1) $\implies$ (5), if the sequence
  $\{\L_n\}$ is simply stable, the same is true of
  $\{\bwedge^i\L_n\otimes M_n\}$ and of $\{N_n^{i,j}\}$. As before,
  the multiplicity of $\Schur_\lambda(\Q^n)$ in $\ker
  (\partial_i^n)^j$ and $\im(\partial_i^n)^j$ is nondecreasing. Their sum is the multiplicity of
  $\Schur_\lambda(\Q^n)$ in $N_n^{i,j}$, which by simple stability of
  $\{N_n^{i,j}\}$ is finite and constant, so the same is true for
  $\ker\partial_i^n$ and $\im\partial_i^n$. Since this holds for each
  $i$, we conclude that the multiplicity of $\Schur_\lambda(\Q^n)$ in
  $H_i(\L_n;M_n)^j$ is constant, as desired. Condition SIV follows as
  before, and we conclude that $\{H_i(\L_n;M_n)^j\}$ is simply stable.
\end{proof}

\subsection{Applications and examples}
\label{section:applications}
In this subsection we give a number of applications of
Theorem~\ref{thm:equivhomLie}, Theorem~\ref{thm:equivSpLie}, and
Theorem~\ref{thm:simplehomLie}.

\para{Free Lie algebras} Let $V_n$ be a
$\Q$--vector space with basis $x_1,\ldots,x_n$.  Let
$\L(V_n)=\L(x_1,\ldots,x_n)$ be the free Lie algebra on $V_n$. The
action of $\GL(V_n)\approx \GL_n\Q$ on $V_n$ induces an action of
$\GL_n\Q$ on $\L(V_n)$. The Lie algebra $\L(V_n)$ has a natural
grading
\[\L(V_n)=\bigoplus_{i\geq 1} \L_i(V_n)\] which is preserved by the
action of $\GL_n\Q$. The obvious inclusion of $V_n \hookrightarrow
V_{n+1}$ induces natural maps $\L(V_n)\hookrightarrow \L(V_{n+1})$ and $\L_i(V_n)\hookrightarrow
\L_i(V_{n+1})$.  These inclusions are respected by the inclusion of $\GL_n\Q \hookrightarrow \GL_{n+1}\Q$.

The free Lie algebra $\L(V_n)$ has the following homology groups for
all $n\geq 1$:
\[H_i(\L(V_n);\Q)=
\begin{cases}
 \Q &i=0\\
 V_n &i=1\\
 0 & i\geq 2
\end{cases}\] This follows from \eqref{eq:freepn} below, and can also
be checked directly. As $\GL_n\Q$--representations, these are
$\Q=\Schur_{(0)}(\Q^n)$, with grading 0, and $V_n=\Schur_{(1)}(\Q^n)$,
with grading 1. Thus $\{H_*(\L(V_n);\Q)\}$ is simply
stable, so Theorem~\ref{thm:simplehomLie} gives the following
corollary.

\begin{corollary}\label{cor:freeLie}
  For each fixed $m\geq 1$, the sequence of $\GL_n\Q$--representations
  $\{\L_m(V_n)\}$ of degree $m$ components of the free Lie algebras
  $\L(V_n)$ is simply representation stable.
\end{corollary}

In fact, the multiplicities of $V(\lambda)_n$ in $\L_m(V_n)$ are
known, at least in some sense. For the irreducible representation
$V(\lambda)_n$ to appear in $\L_m(V_n)$ it is necessary that $\lambda$
be a partition of $m$. For such $\lambda$, Bakhturin \cite[Proposition
3, \S3.4.5]{Ba} gives the following formula for the multiplicity.  Let
$\chi_{\lambda}$ denote the character of the irreducible
representation of $S_m$ associated to $\lambda$, and let $\tau$ be the
$m$--cycle $(1\,2\,\ldots\, m)$. Then the multiplicity of
$V(\lambda)_n$ in $\L_m(V_n)$ is
\[c_\lambda\coloneq \frac{1}{m}\sum_{d|m}\mu(d)\chi_\lambda(\tau^{m/d}).\]
Despite this formula, due to the dependence on the values of irreducible
characters $\chi_\lambda$ of the symmetric group, explicitly
calculating these multiplicities remains an active area of research.

\para{Free nilpotent Lie algebras}
Let $\mathcal{N}_k(n)$ be the level $k$ truncation of the free Lie
algebra of rank $n$, meaning:
\[\mathcal{N}_k(n)=\L(V_n)/\L_{k+1}(V_n)=\bigoplus_{i\leq k} \L_i(V_n).\]
This is the free $k$--step nilpotent Lie algebra on $V_n$. Since
$\mathcal{N}_k(n)$ is a truncation of $\L(V_n)$,
Corollary~\ref{cor:freeLie} tells us that for each fixed $i$ the
sequence of $i^{\text{th}}$ graded pieces
$\{\mathcal{N}_k(n)^i=\L_i(V_n)\}$ is simply stable. Thus as a
corollary of Theorem~\ref{thm:simplehomLie}, we obtain the following
theorem of Tirao \cite[Theorem 2.9]{Ti}.

\begin{corollary}[Tirao]
  \label{corollary:nilp}
  Fix any $k\geq 1$.  Then for each fixed $i\geq 0$ the sequence of
  $\GL_n(\Q)$--representations $\{H_i(\mathcal{N}_k(n);\Q)\}$ is
  simply stable.
\end{corollary}

The novelty of our deduction of Corollary~\ref{corollary:nilp} is that we start with the homology of the free 
Lie algebra, which is quite easy to compute, then we move to the free Lie algebra itself using one direction of Theorem~\ref{thm:equivhomLie},  
then to its nilpotent truncation, then to the homology of that truncation using the reverse implication in Theorem~\ref{thm:equivhomLie}.

We note that, while the $\mathcal{N}_k(n)$ are not themselves
complicated, their homology is quite complicated.  For $k=2$, it
follows from work of Kostant (see \cite[Theorem 3.1]{CT1}) that the
multiplicity of $V(\lambda)_n$ in $H_i(\mathcal{N}_2(n);\Q)$ is 0
unless $\lambda$ is self-conjugate (i.e., its Young diagram is
symmetric under reflection across the diagonal) and has exactly $i$
boxes above the diagonal, in which case the multiplicity is 1. No
formula is known in general, but for some small values of $i$ and $k$,
Tirao \cite{Ti} computes the decomposition of this homology
explicitly. For example, he proves that (in our terminology):
\begin{align*}
  H_3(\mathcal{N}_3(2);\Q)&=V(4,2)\\
  H_3(\mathcal{N}_3(3);\Q)&=V(4,2)\oplus V(2,2,2)\oplus V(3,1,1)\oplus
  V(3,3,1)\oplus V(4,2,1)\oplus V(5,1,1)\\
  H_3(\mathcal{N}_3(4);\Q)&=V(4,2)\oplus V(2,2,2)\oplus V(3,1,1)\oplus
  V(3,3,1)\oplus V(4,2,1)\oplus
  V(5,1,1)\\&\quad\qquad\oplus V(3,1,1,1)\oplus V(3,2,1,1)\\
  H_3(\mathcal{N}_3(n);\Q)&=V(4,2)\oplus V(2,2,2)\oplus V(3,1,1)\oplus
  V(3,3,1)\oplus V(4,2,1)\oplus V(5,1,1)\\&\quad\qquad\oplus
  V(3,1,1,1)\oplus V(3,2,1,1)\oplus V(2,2,1,1,1) \qquad\quad\text{for
  }n\geq 5
\end{align*}
Note that, as guaranteed by simple stability, each representation
$V(\lambda)_n$ first appears in $H_3(\mathcal{N}_3(n))$ when $n$ is the
number of rows of $\lambda$, and persists with the same multiplicity
thereafter.

Also from Theorem~\ref{thm:simplehomLie}, we obtain the following
corollary on the homology of $\mathcal{N}_k(n)$ with coefficients
in the adjoint representation.

\begin{corollary}\label{cor:adjnil}
  Fix $k\geq 1$.  Then for each fixed $i\geq 0$ the sequence of
  $\GL_n(\Q)$--representations
  $\{H_i(\mathcal{N}_k(n);\mathcal{N}_k(n))\}$ is simply stable.
\end{corollary}

The adjoint homology of $\mathcal{N}_2(n)$ was studied in
Cagliero--Tirao \cite{CT1}, and in this case
Corollary~\ref{cor:adjnil} can be deduced from \cite[Theorem 4.4]{CT1}
combined with the description of $H_i(\mathcal{N}_2(n);\Q)$ above. For
$k\geq 3$, to the best of our knowledge this result was not previously
known.

\para{Continuous cohomology and pseudo-nilpotent groups}  The
\emph{continuous cohomology} of a group $\Gamma$ is the direct limit
\[H^*_{\text{cts}}(\Gamma;\Q)=\lim_{\longrightarrow} H^*(\Gamma/K;\Q)\]
of the cohomology of all its finitely generated nilpotent quotients
$\Gamma/K$. The basic properties of continuous cohomology are
established in Hain \cite{Ha2}. There is an obvious comparison map
$H^*_{\text{cts}}(\Gamma;\Q)\to H^*(\Gamma;\Q)$, which is always an
isomorphism on $H^0$ and $H^1$, and is always injective on $H^2$. A
finitely generated group $\Gamma$ is called \emph{pseudo-nilpotent} if
this map is an isomorphism in every degree.

Nomizu's theorem \cite{No} implies that for finitely generated groups,
$H^*_{\text{cts}}(\Gamma;\Q)$ coincides with the continuous cohomology
$H^*_{\text{cts}}(\mathfrak{g};\Q)$ of the Malcev Lie algebra
$\mathfrak{g}$ of $\Gamma$. The Malcev Lie algebra is a certain
pronilpotent $\Q$--Lie algebra associated to $\Gamma$.  We will only
need the following property.  Recall that the \emph{lower central
  series}
\[\Gamma=\Gamma_1 >\Gamma_2 > \cdots\]
of a group $\Gamma$ is defined inductively by $\Gamma_1\coloneq
\Gamma$ and $\Gamma_{n+1}\coloneq [\Gamma,\Gamma_n]$.  The
\emph{associated graded (rational) Lie algebra} $\gr(\Gamma)$ is the
$\Q$--Lie algebra defined by \[\gr(\Gamma)\coloneq
\bigoplus_{n=1}^\infty (\Gamma_n/\Gamma_{n+1})\otimes\Q\] where the
Lie bracket is induced by the group commutator.  The Malcev Lie
algebra $\mathfrak{g}$ of $\Gamma$ has the property that the graded
Lie algebra associated to its lower central series is isomorphic to
$\gr(\Gamma)$. For the groups we consider, we have an isomorphism
$H^*_{\text{cts}}(\mathfrak{g};\Q)\approx H^*(\gr(\Gamma);\Q)$. Note
that the automorphism group $\Aut(\Gamma)$ acts on $\Gamma$, preserves
each $\Gamma_i$, and so acts on $\gr(\Gamma)$.  It is well known that this action 
factors through the representation
\begin{equation}
  \label{eq:autrep}
  \Aut(\Gamma)\to \Aut(\Gamma/[\Gamma,\Gamma]).
\end{equation}

For the free group $F_n$ it is well known that
$\gr(F_n)=\L(H_1(F_n;\Q))=\L(V_n)$. Since free groups are
pseudo-nilpotent (see \cite[Corollary 5.10]{Ha2}), we conclude that
\begin{equation}
  \label{eq:freepn}
  H^*(\gr(F_n);\Q)=H^*(F_n;\Q)=\Q\oplus
  \Q^n.
\end{equation}
In this case the representation \eqref{eq:autrep} gives the natural
action of $\GL_n\Z$ on $\L(V_n)$, which extends to a representation of
$\GL_n\Q$. In Corollary~\ref{cor:freeLie}, we noted that the
isomorphism $H^*(\gr(F_n);\Q)=H^*(F_n;\Q)$ implies that
$\{H^i(\gr(F_n);\Q)\}$ is simply stable for each $i$, and so we concluded
that the sequence of graded components $\{\gr(F_n)^j\}$ is simply
stable for each $j\geq 0$ as well.

We would like to mimic this argument for surface groups. Let
$\pi_g=\pi_1(S_g)$ be the fundamental group of the closed, connected, 
orientable surface of genus
$g\geq 2$. Labute \cite{La} proved that $\gr(\pi_g)$ is the quotient of the
free Lie algebra $\L(H_1(S_g;\Q))$ by the ideal generated by the
symplectic form:
\[\gr(\pi_g)\approx \L(H_1(S_g;\Q))/([a_1,b_1]+\cdots+[a_g,b_g])\]
Further, in this case the representation \eqref{eq:autrep} is known to
factor through the integral symplectic group $\Sp_{2g}\Z$.  Hain
proved \cite[Proposition 5.11]{Ha2} that $\pi_g$ is pseudo-nilpotent,
and that the continuous cohomology of its Malcev Lie algebra coincides
with $H^*(\gr(\pi_g);\Q)$. Thus $H^*(\gr(\pi_g);\Q)\approx
H^*(S_g;\Q)$, which as an $\Sp_{2g}\Q$--representation decomposes as
follows for all $g\geq 1$:
\[H^i(S_g;\Q)=
\begin{cases}
  V(0) &i=0,2\\
  V(1) &i=1\\
  0 & i\geq 3
\end{cases}\]

We do not have maps $\pi_g\to \pi_{g+1}$, but we do have surjections
$\pi_{g+1}\to \pi_g$ inducing surjections $\gr(\pi_{g+1})\to
\gr(\pi_{g})$, which induce maps $H^*(\gr(\pi_g))\to
H^*(\gr(\pi_{g+1}))$. For each $i$, this makes the cohomology
groups $\{H^i(\gr(\pi_g);\Q)=H^i(S_g;\Q)\}$ into a uniformly stable
sequence of
$\Sp_{2g}\Q$--representations. Theorem~\ref{thm:equivSpLie}  
works just as well for cohomology, so we obtain as a corollary the
following result of Hain \cite[Corollary 8.5]{Ha}.
\begin{corollary}[Hain]\label{cor:hain}
  For each fixed $i\geq 0$, the sequence of
  $\Sp_{2g}\Q$--representations given by the graded components
  $\{\gr(\pi_g)^i\}$ are uniformly representation stable for all $j$.
\end{corollary}

\para{Homology of symplectic Lie algebras}
Many sequences of Lie algebras $\L_n$ are naturally
$\Sp_{2n}\Q$--representations and in fact are uniformly stable, such
as the Heisenberg Lie algebras considered below. We would like to
conclude stability for the homology groups of this sequence of Lie
algebras $\{H_i(\L_n;\Q)\}$ as we did in
Corollary~\ref{corollary:nilp} above. But without a good notion of
strong stability for $\Sp_{2n}\Q$--representations, the proof of the
necessary implication in Theorem~\ref{thm:equivhomLie}, namely (1)
$\implies$ (5), does not work. However, in specific cases the argument
can be successfully modified.

We will give a concrete example of such a modification, but first we
extract from the proof of (1) $\implies$ (5) in
Theorem~\ref{thm:equivhomLie} exactly where strong stability was
used. Ignoring the grading for the moment, the homology $H_*(\L_n;\Q)$
is computed by the rows of the complex
\begin{equation}\label{eq:Heis}
  \xymatrix{
    \cdots\ar[r]&\bwedge^3\L_n\ar^{\partial_3}[r]\ar[d]&
    \bwedge^2\L_n\ar^{\partial_2}[r]\ar[d]&
    \L_n\ar^{\partial_1}[r]\ar[d]&
    \Q\\
    \cdots\ar[r]&\bwedge^3\L_{n+1}\ar^{\partial_3}[r]&
    \bwedge^2\L_{n+1}\ar^{\partial_2}[r]&
    \L_{n+1}\ar^{\partial_1}[r]&
    \Q
  }
\end{equation}
We know stability holds for each term $\{\bwedge^i\L_n\}$, but the
differentials between them introduce a possible source of
instability. To draw conclusions about stability for
$H_i(\L_n;\Q)=\ker\partial_i/\im\partial_{i+1}$, we need control over
how the differentials $\partial_i$ interact with the vertical maps
$\bwedge^i\L_n\hookrightarrow \bwedge^i\L_{n+1}$. For example, we
do not know that $\{\ker\partial_i\}$ is stable.

In Theorem~\ref{thm:equivhomLie}, the type-preserving assumption
guaranteed that the vertical maps preserved isotypic components. Then
since the vertical maps are injective, the commutativity of
\eqref{eq:Heis} implied that even if $\{\ker\partial_i\}$ and $\{\im\partial_i\}$ did not
stabilize immediately, their decompositions were nondecreasing in $n$. Thus once the terms, which here would be $\{\bwedge^i\L_n=\ker \partial_i\oplus \im\partial_i\}$, stabilized, the summands $\{\ker\partial_i\}$ and $\{\im\partial_i\}$ were forced to stabilize as well. However, for
$\Sp_{2n}\Q$--representations the vertical maps for $\bwedge^i\L_n$
are almost never type-preserving, as we will see in detail below. Thus
a new idea is needed.

Let $H_n\coloneq V(\lambda_1)_n=\Q^{2n}$ be the standard
representation of $\Sp_{2n}\Q$. For $i\leq n$, we have the
decomposition into irreducibles
\[\bwedge^i H_n=V(\lambda_i)_n\oplus V(\lambda_{i-2})_n\oplus \cdots
V(\lambda_\epsilon)_n\] where $\epsilon=0$ or $1$ if $i$ is even or
odd respectively. The inclusion $\bwedge^i H_n\hookrightarrow
\bwedge^i H_{n+1}$ does \emph{not} respect this decomposition. In
fact, we have the following:

\begin{lemma}
  \label{lem:wedgeH}
  For $n>i$ and any irreducible representation
  $V(\lambda_k)_n\subset\bwedge^i H_n$ with $i>k$, the
  $\Sp_{2n+2}\Q$--span of $V(\lambda_k)_n$, considered as a subspace
  of $\bwedge^i H_{n+1}$, is isomorphic to
  \[V(\lambda_{k+2})_{n+1}\oplus
  V(\lambda_k)_{n+1}.\]
\end{lemma}

\begin{proof}
  For an overview of the symplectic representation theory used here,
  see \cite[\S 17]{FH}.  For each $i\leq n$ there is a unique
  contraction $C\colon \bwedge^iH\to \bwedge^{i-2}H$, with $\ker
  C\approx V(\lambda_i)$ generated by $a_1\wedge \cdots \wedge
  a_i$. This induces a filtration of $\bwedge^i H$ by \[\ker
  C^j\approx V(\lambda_i)\oplus \cdots\oplus V(\lambda_{i-2j+2}).\] A
  complement to $\ker C$ is given by the image of $\cdot\wedge
  \omega_n\colon \bwedge^{i-2}H\to \bwedge^i H$, where
  \[\omega_n=a_1\wedge b_1+\cdots+a_n\wedge b_n\] spans the trivial
  subrepresentation of $\bwedge^2 H$. It follows that
  $V(\lambda_k)_n\subset \bwedge^iH_n$ is the $\Sp_{2n}\Q$--span of
  $v_{k,n}\coloneq a_1\wedge\cdots\wedge a_k\wedge (\omega_n)^j$,
  where $j=(i-k)/2$. Let $W\subset \bwedge^i H_{n+1}$ be the desired
  representation, the $\Sp_{2n+2}\Q$--span of $v_{k,n}$. The
  contractions $C$ commute with the inclusion $\bwedge^i
  H_n\hookrightarrow \bwedge^i H_{n+1}$. Thus since $v_{k,n}$ is
  contained in $\ker C^{j+1}$ but not $\ker C^j$, we know that $W$ is
  contained in $V(\lambda_i)_{n+1}\oplus \cdots\oplus
  V(\lambda_k)_{n+1}$ but not in $V(\lambda_i)_{n+1}\oplus
  \cdots\oplus V(\lambda_{k+2})_{n+1}$.

  Under the inclusion $\bwedge^2 H_n\hookrightarrow \bwedge^2
  H_{n+1}$, $\omega_n$ is not mapped to $\omega_{n+1}$. Instead we
  have $\omega_n=\omega_{n+1}-a_{n+1}\wedge b_{n+1}$, and so
  $(\omega_n)^j=(\omega_{n+1})^j-(\omega_{n+1})^{j-1}\wedge
  a_{n+1}\wedge b_{n+1}.$ Writing \[v_{k,n}=a_1\wedge \cdots\wedge
  a_k\wedge (\omega_{n+1}-a_{n+1}\wedge b_{n+1})\wedge
  (\omega_{n+1})^{j-1},\] we see that $v_{k,n}$ is in the image of
  $\cdot\wedge (\omega_{n+1})^{j-1}\colon \bwedge^{k+2}H_{n+1}\to
  \bwedge^iH_{n+1}$. Combined with the above bound on $W$, this
  implies that $W$ is contained in $V(\lambda_{k+2})_{n+1}\oplus
  V(\lambda_k)_{n+1}$.

  Since we know that $W$ is not contained in $V(\lambda_{k+2})_{n+1}$,
  it remains to show that $W$ is not contained in
  $V(\lambda_k)_{n+1}$. Note that
  $C^j(v_{k,n})=\frac{(n-k)!}{(n-k-j)!}a_1\wedge \cdots\wedge
  a_k$. Then if $A\in \Sp_{2n+2}\Q$ is any element fixing
  $a_1\wedge\cdots\wedge a_k$ but not $v_{n,k}$ (for example, the
  permutation matrix exchanging $a_n$ with $a_{n+1}$ and $b_n$ with
  $b_{n+1}$), we have $C^j(A\cdot v_{k,n})=A\cdot
  C^j(v_{k,n})=C^j(v_{k,n})$. In particular, the vector
  $v_{k,n}-A\cdot v_{k,n}$ lies in $\ker C^j$. As explained above,
  this shows that $W\approx V(\lambda_{k+2})_{n+1}\oplus
  V(\lambda_k)_{n+1}$.
\end{proof}

\begin{remark}
  \label{remark:substitute}
  Lemma~\ref{lem:wedgeH} can serve as a substitute for the
  type-preserving assumption in the argument outlined before the
  lemma.  Take any sequence of maps $f_n\colon
  \bwedge^i H_n\to V_n$ commuting with the inclusions $\bwedge^i
  H_n\hookrightarrow \bwedge^i H_{n+1}$ and $V_n\hookrightarrow
  V_{n+1}$. If $V(\lambda_k)_n\subset \ker f_n$,
  Lemma~\ref{lem:wedgeH} implies that $V(\lambda_{k+2})_{n+1}\oplus
  V(\lambda_k)_{n+1}\subset \ker f_{n+1}$. Thus the multiplicity of
  $V(\lambda_k)_n$ in $\ker f_n$ is nondecreasing. (In
  fact, by induction we see there is some $k\leq i$ so that $\ker f_n$ is
  always exactly $V(\lambda_i)_n\oplus \cdots\oplus V(\lambda_k)_n$
  for sufficiently large $n$.) The same applies to $\im f_n$.
\end{remark}

\para{Heisenberg Lie algebras} As an explicit example to which Remark~\ref{remark:substitute} applies,
we consider the Heisenberg Lie algebras $\H_{2n+1}$, defined as
follows. Given a symplectic form $\omega$ on $\Q^{2n}$, this is the
central extension
\[0\to \Q\to \H_{2n+1}\to \Q^{2n}\to 0\] classified by $\omega\in
H^2(\Q^{2n};\Q)\approx \bwedge^2\Q^{2n}$.  Since the natural action of
$\Sp_{2g}\Q$ on $\Q^{2n}$ preserves $\omega$, it extends to an action
of $\Sp_{2g}\Q$ on $\H_{2n+1}$.  There is an obvious inclusion
$\H_{2n+1}\hookrightarrow\H_{2n+3}$ which makes $\{\H_{2n+1}\}$ into a
consistent sequence of Lie algebras. As an
$\Sp_{2n}\Q$--representation $\H_{2n+1}\approx V(0)\oplus
V(\lambda_1)$, so the sequence $\{\H_{2n+1}\}$ is uniformly representation 
stable.

Note that \[\bwedge^i \H_{2n+1}\approx \bwedge^i(H_n\oplus \Q)\approx
\bwedge^i H_n\oplus \bwedge^{i-1}H_n\] and that this decomposition is
respected by the maps
$\bwedge^i\H_{2n+1}\hookrightarrow\bwedge^i\H_{2n+3}$. So $\{\bwedge^i \H_{2n+1}\}$
is uniformly stable for each $i\geq 0$ and we can apply
Remark~\ref{remark:substitute} to the complex
\[ \cdots\longrightarrow \bwedge^3\H_{2n+1}
  \overset{\partial_3}{\longrightarrow} \bwedge^2\H_{2n+1}
  \overset{\partial_2}{\longrightarrow} \H_{2n+1}
  \overset{\partial_1}{\longrightarrow} \Q,
\]
to conclude that $\{\ker \partial_i\}$ and $\{\im\partial_i\}$ are uniformly stable for each
$i$. This shows that the homology of $\H_{2n+1}$ is
uniformly stable. (It turns out that in the complex above at most one
differential is nonzero in each grading, so it is easy to compute by
hand that $H_i(\H_{2n+1};\Q)=V(\omega_i)$ for $i\leq n$ and see
uniform stability directly; see, e.g., \cite[Theorem 4.2]{CT2}.)  We
therefore have the following.
\begin{xample}
\label{example:Heis}
  For each $i\geq 0$ the sequence of
  $\Sp_{2n}\Q$--representations $\{H_i(\H_{2n+1};\Q)\}$ is uniformly
  representation stable.
\end{xample}
To extend this argument to adjoint and exterior coefficients, all that
would be needed is to duplicate Lemma~\ref{lem:wedgeH} for
$\bwedge^iH_n\otimes H_n$ and $\bwedge^k H_n\otimes \bwedge^\ell
H_n$. However, we do not do this here, since Cagliero--Tirao have
already computed these homology groups. The adjoint homology
$H_i(\H_{2n+1};\H_{2n+1})$ is $V(\omega_1+\omega_i)\oplus
V(\omega_{i+1})$ for $i<n$ \cite[Corollary 4.15]{CT2}. For the
exterior homology, the irreducibles appearing in
$H_i(\H_{2n+1};\bwedge^k \H_{2n+1})$ always correspond to the sum of
two fundamental weights $V(\omega_j+\omega_\ell)$, with multiplicity
either 1 or 0, independent of $n$ for $i\leq n$ \cite[Theorem
4.13]{CT2}. In both cases we have uniform stability.
\begin{xample}[Cagliero--Tirao]
\label{example:adjHeis}
  For each fixed $i\geq 0$ and $k\geq 0$ the sequences of
  $\Sp_{2n}\Q$--representations $\{H_i(\H_{2n+1};\H_{2n+1})\}$ and
  $\{H_i(\H_{2n+1};\bwedge^k \H_{2n+1})\}$ are uniformly
  representation stable.
\end{xample}

\subsection{The Malcev Lie algebra of the pure braid group}
\label{section:malcev:braid}

In this subsection we describe a conjecture which can be thought of as
a ``infinitesimal'' version of Theorem~\ref{thm:pure}. Let $\Gamma=P_n$,
the pure braid group on $n$ strands, and let $\fp_n\coloneq\gr(P_n)$.
The Lie algebra $\fp_n$ occurs, among other places, in the theory of
Vassiliev invariants.  Drinfeld and Kohno (see \cite{Ko}) gave a
finite presentation for $\fp_n$, as follows.  Let
$\L(\{X_{ij}\})$ denote the free Lie
algebra on the set of formal symbols $\{X_{ij}: 1\leq i,j\leq n, i\neq
j\}$.  Then for $n\geq 4$,
\[\fp_n=\L(\{X_{ij}\})/R\]
where $R$ is the ideal generated by the quadratic relations:
\[[X_{ij},X_{kl}] \ \ \text{with $i,j,k,l$ distinct}\]
\[[X_{ij},X_{ik}+X_{jk}] \ \ \text{with $i,j,k$ distinct}\] Consider
the action of $S_n$ on $\fp_n$. Since the relations above are
homogeneous, the grading on $\L(\{X_{ij}\})$ descends to a grading on
$\fp_n$, which is clearly preserved by the action of $S_n$. Let
$\fp_n^i$ denote the $i^{\text{th}}$ graded component of $\fp_n$.

\begin{conjecture}[Representation stability for $\fp_n$]
\label{conjecture:malcevpn}
For each fixed $i\geq 1$, the sequence $\{\fp_n^i\}$ is a
uniformly representation stable sequence of $S_n$--representations.\footnote{Since this manuscript was originally posted, this conjecture has been proved by the authors and Ellenberg in \cite[Theorem~5.8]{CEF1}.}
\end{conjecture}

As evidence for this conjecture, we point out that $P_n$ is
pseudo-nilpotent \cite[Example 5.12]{Ha2}. Furthermore, by
Theorem~\ref{thm:pure}, for each fixed $i\geq 0$ the cohomology
$H^i(\fp_n;\Q)=H^i(P_n;\Q)$ is a uniformly stable sequence of
$S_n$--representations. Thus Conjecture~\ref{conjecture:malcevpn}
would follow as in Corollary~\ref{cor:hain} if we had a version of
Theorem~\ref{thm:equivhomLie} for representations of $S_n$. We remark
that not all aspherical hyperplane complements have pseudo-nilpotent
fundamental group (see e.g.\ Falk \cite[Proposition 5.1, Example
5.3]{Fa}), so we do not expect Conjecture~\ref{conjecture:malcevpn} to
extend to all such groups.

\section{Homology of the Torelli subgroups of $\Mod(S)$ and
  $\Aut(F_n)$}
\label{section:torelli}

In this section we discuss representation stability in the context of
the homology of the Torelli groups associated with mapping class
groups and automorphism groups of free groups.  Most of the picture
here is conjectural.  However, before the idea of representation
stability, even a conjectural picture of these homology groups was
lacking.

\subsection{Homology of the Torelli group}  
Let $S_{g,1}$ be a connected, compact, oriented surface of genus
$g\geq 2$ with one boundary component.  Let $H\coloneq
H_1(S_{g,1};\Q)$ and let $H_\Z\coloneq H_1(S_{g,1},\Z)$.  The
\emph{mapping class group} $\Mod_{g,1}$ is the group of homotopy
classes of homeomorphisms of $S_{g,1}$, where both the homeomorphisms
and the homotopies fix $\partial S_{g,1}$ pointwise.  The action of
$\Mod_{g,1}$ on $H_\Z$ preserves algebraic intersection number, which
is a symplectic form on $H_\Z$, yielding a symplectic representation
which fits into the exact sequence
\[1\to\I_{g,1}\to\Mod_{g,1}\to\Sp_{2g}\Z\to 1,\] where $\I_{g,1}$ is
the \emph{Torelli group}, consisting of those $f\in\Mod_{g,1}$ acting
trivially on $H_\Z$.  The conjugation action of $\Mod_{g,1}$ on
$\I_{g,1}$ descends to an action of $\Sp_{2g}\Z$ by outer
automorphisms, which gives each $H_i(\I_{g,1};\Q)$ the structure of an
$\Sp_{2g}\Z$--module.  The natural inclusion of surfaces
$S_{g,1}\hookrightarrow S_{g+1,1}$ induces an inclusion
$\I_{g,1}\hookrightarrow \I_{g+1,1}$, by extending by the
identity. For each $i\geq 0$ the induced homomorphism
$H_i(\I_{g,1};\Q)\to H_i(\I_{g+1};\Q)$ respects the action of
$\Sp_{2g}\Z$.

\bigskip If $G$ is any group and $V$ is any (perhaps infinite
dimensional) $G$--representation, we define the
\emph{finite-dimensional part} of $V$, denoted $V^{\fd}$, to be the
subspace of $V$ consisting of those vectors whose $G$--orbit spans a
finite-dimensional subspace of $V$.  Note that $V^{\fd}$ may itself be
infinite dimensional.  Our first conjecture about $H_i(\I_{g,1};\Q)$
makes a prediction about its finite-dimensional part.  It is a slight
refinement of a conjecture we first stated in \cite{CF}.
\begin{conjecture}[Homology of the Torelli group]
  \label{conjecture:torelli}
  For each fixed $i\geq 1$, each of the following statements holds.

  \bigskip
  \noindent \textbf{Preservation of finite-dimensionality: }The
  natural map \[H_i(\I_{g,1};\Q)^{\fd}\to H_i(\I_{g+1,1};\Q)\] induced
  by the inclusion $\I_{g,1}\hookrightarrow \I_{g+1,1}$ has image
  contained in $H_i(\I_{g+1,1};\Q)^{\fd}$.

  \medskip
  \noindent \textbf{Rationality:} Every irreducible
  $\Sp_{2g}\Z$--subrepresentation in $H_i(\I_{g,1};\Q)^{\fd}$ is the
  restriction of an irreducible $\Sp_{2g}\Q$--representation.

  \medskip
  \noindent \textbf{Stability: }The sequence of
  $\Sp_{2g}\Q$--representations $\{H_i(\I_{g,1};\Q )^{\fd}\}$ is
  uniformly representation stable.
\end{conjecture}

\para{Remarks}
\begin{enumerate}
\item Along with Conjecture~\ref{conjecture:torelli} for
  $H_i(\I_{g,1};\Q)$, we have a corresponding, equivalent conjecture
  for the cohomology $H^i(\I_{g,1};\Q)$, with stability in the sense of
  Definition~\ref{definition:repstabrev}. In an unpublished 1991 manuscript, Hain made a conjecture similar to the stability part of Conjecture 6.1 for 
$H^i(\I_{g,1};\Q)$.

\item A form of the Margulis Superrigidity Theorem (see \cite[Theorem
  VIII.B]{Ma}) gives that any finite-dimensional representation (over
  $\C$) of $\Sp_{2g}\Z$ virtually extends to a rational representation
  of $\Sp_{2g}\Q$.\footnote{One can also use the solution to the
    congruence subgroup property for $\Sp_{2g}\Z$, $g>1$ here; see
    \cite[Theorem~16.2]{BMS}.}  Thus the Rationality part of
  Conjecture~\ref{conjecture:torelli} is meant to ensure that we can extend
  to $\Sp_{2g}\Q$ without passing to a finite index subgroup.  We also
  remark that, by a statement close to the Borel Density Theorem
  (namely Proposition 3.2 of \cite{Bo2}), a representation of
  $\Sp_{2g}\Q$ is irreducible if and only if its restriction to
  $\Sp_{2g}\Z$ is irreducible, so we can (and will) ignore this
  distinction.  Similar statements apply to $\GL_n\Q$ as well.

\item It is known that $H_i(\I_{g,1};\Q )^{\fd}$ is not all of $H_i(\I_{g,1};\Q )$, but it 
is possible that they do coincide for $g\gg i$;  see the examples discussed after Conjecture~\ref{conjecture:finitegen} below.
  
\item The Torelli group is often defined for closed surfaces or for
  surfaces with punctures. In this case there are no maps connecting
  the Torelli groups for different $g$, so the strongest statement one
  could hope for is multiplicity stability for the homology of the corresponding
  Torelli groups.  We conjecture this to be true.
\end{enumerate}

Some evidence for Conjecture~\ref{conjecture:torelli} in each
dimension $i\geq 1$ is given and discussed in detail in \cite{CF}.
Further, we note that Conjecture~\ref{conjecture:torelli} is true for
$i=1$, by Johnson's computation that \[H_1(\I_{g,1})\approx
H_1(\I_{g,\ast})\approx V(\omega_3)\oplus V(\omega_1) \ \ \mbox{for
  each $g\geq 3$}.\]

\noindent
Since this paper was originally posted, Boldsen-Dollerup \cite{BD} have proved the surjectivity part of stability in Conjecture~\ref{conjecture:torelli} for $i=2$.  Finally, we note the well-known analogy of $\I_{g,1}$ and $\Sp_{2g}\Z$
with $P_n$ and $S_n$.  Since representation stability holds for the
latter example (Theorem~\ref{thm:pure}), one is led to believe it
holds for the former.

\para{Malcev Lie algebra of $\I_{g,1}$} There is a kind of
``infinitesimal'' version of Conjecture~\ref{conjecture:torelli},
parallel to Conjecture~\ref{conjecture:malcevpn} for the pure braid
group.  Let $\gr(\I_{g,1})$ denote the graded rational Lie algebra
associated to the lower central series of $\I_{g,1}$ (see
\S\ref{section:applications}), and let $\gr(\I_{g,1})^i$ denote its
$i^{\text{th}}$ graded piece. Hain computed $\gr(\I_g)$ in \cite{Ha}, and this was extended by Habegger--Sorger \cite{HS} to the case of
surfaces with boundary.  To state their result, let
$H=H_1(S_{g,1};\Q)$ as usual, and for any vector space $V$ let $\L(V)$
denote the free Lie algebra on $V$ as in \S\ref{section:applications}.
The extension by Habegger--Sorger of Hain's theorem states that, for all
$g\geq 6$, the rational Lie algebra $\gr(\I_{g,1})$ has a
presentation:
 
 \[\gr(\I_{g,1})=\L(\bwedge^3H)/(R_1,R_2)\]
 where $(R_1,R_2)$ denotes the ideal generated by the
 $\Sp_{2g}(\Q)$--span of the two elements
\begin{align*}
  R_1&=(a_1\wedge a_2\wedge b_2)\wedge (a_3\wedge a_4\wedge b_4)\\
  R_2&=(a_1\wedge a_2\wedge b_2)\wedge (a_g\wedge \omega)
\end{align*}
where $w\coloneq\sum_{i=1}^ga_i\wedge b_i$.  One nontrivial consequence of Hain's theorem is that the natural $\Sp_{2g}\Z$--action on $\gr(\I_{g,1})$ extends to an
$\Sp_{2g}\Q$--action.  As previously mentioned, $\gr(\I_{g,1})$ is the
associated graded associated to the Malcev Lie algebra of $\I_{g,1}$,
and $\gr(\I_{g,1})^i$ denotes the $i^{\text{th}}$ graded component of
$\gr(\I_{g,1})$.
 
\begin{conjecture}[Stability of the Malcev Lie algebra
  of $\I_{g,1}$]
  \label{conjecture:malcev:torelli}
    For each fixed $i\geq 1$ the sequence
  $\{\gr(\I_{g,1})^i\}$ of $\Sp_{2g}\Q$-representations is uniformly representation
  stable.
\end{conjecture}
  
As evidence for Conjecture~\ref{conjecture:malcev:torelli}, we remark
that the conjecture is true when $i=1$ and when $i=2$, as follows.
Johnson proved \cite{Jo2} that \[\gr(\I_{g,1})^1\approx
\bwedge^3H\approx V(1,1,1)\oplus V(1)\] as $\Sp_{2g}\Z$--modules, and Habegger--Sorger
\cite[Theorem 2.2]{HS} use the work of Hain \cite{Ha} to deduce that 
(in our terminology) :
 \[\gr(\I_{g,1})^2\approx V(2,2)\oplus V(1,1)\oplus V(0)^{\oplus 2}\] 
 as $\Sp_{2g}\Z$--modules.  

\subsection{Homology of $\IA_n$}
The above discussion has an analogy in the case of free groups and
their automorphisms.  Let $F_n$ denote the free group of rank $n$, and
let $\Aut(F_n)$ denote its automorphism group. The action of
$\Aut(F_n)$ on $H_1(F_n;\Z)$ gives the well-known exact
sequence
\[1\to \IA_n\to\Aut(F_n)\to\GL_n\Z\to1.\] The conjugation action of
$\Aut(F_n)$ on $\IA_n$ descends to an outer action of $\GL_n\Z$,
giving each $H_i(\IA_n;\Q)$ the structure of a $\GL_n\Z$--module. The
standard inclusion $F_n\hookrightarrow F_{n+1}$ induces an inclusion
$\IA_n\hookrightarrow\IA_{n+1}$ by extending by the identity.  Thus
for each $i\geq 0$ we have an induced homomorphism $H_i(\IA_n;\Q)\to
H_i(\IA_{n+1};\Q)$ of $\GL_n\Q$--representations.

It is natural to conjecture the analogue of Conjecture
\ref{conjecture:torelli} for $\IA_n$, and in particular that
$\{H_i(\IA_n;\Q)\}$ is representation stable.  However such a conjecture would not capture 
what is going on, even in dimension 1: a computation of Andreadakis,
Farb, Kawazumi, and Cohen-Pakianathan (see e.g.\ \cite{Ka}) gives:
\begin{equation}
\label{eq:ia1}
H_1(\IA_n;\Q)\approx \bwedge^2\Q^n\otimes (\Q^n)^\ast\approx
V(L_1+L_2-L_n)\oplus V(L_1)
\end{equation}
from which we see that the decomposition of the sequence
$\{H_1(\IA_n;\Q)\}$ does not stabilize (except in the trivial sense, observing that \emph{no} representation ever appears twice as the first summand). However, the notion of
\emph{mixed} tensor stability, defined in \S\ref{section:strong}, suffices to capture 
the stability here. Indeed, the computation in
\eqref{eq:ia1} shows that the sequence $\{H_1(\IA_n;\Q)\}$ is mixed
representation stable, since
\[H_1(\IA_n;\Q)=V(1,1;1)\oplus V(1)\]
for sufficiently
large $n$. With this alteration, we give the analogue of Conjecture
\ref{conjecture:torelli} for $\IA_n$.
\begin{conjecture}[Homology of $\IA_n$]
\label{conjecture:ian}
For each fixed $i\geq 1$, each of the
  following statements hold.

  \bigskip
  \noindent \textbf{Preservation of finite-dimensionality: }The
  natural map \[H_i(\IA_n;\Q)^{\fd}\to H_i(\IA_{n+1};\Q)\] induced by
  the inclusion $\IA_n\hookrightarrow \IA_{n+1}$ has image contained
  in $H_i(\IA_{n+1};\Q)^{\fd}$.

  \medskip
  \noindent \textbf{Rationality:} Every irreducible
  $\GL_n\Z$--subrepresentation in $H_i(\IA_n;\Q)^{\fd}$ is the
  restriction of an irreducible $\GL_n\Q$--representation.

  \medskip
  \noindent \textbf{Stability: }The sequence of
  $\GL_n\Q$--representations $\{H_i(\IA_n;\Q )^{\fd}\}$ is 
  uniformly \emph{mixed} representation stable.
\end{conjecture}

As for the ``infinitesimal'' version of Conjecture~\ref{conjecture:ian}, we
conjecture that each of the $\GL_n\Z$--representations $\gr(\IA_n)^i$
extend to $\GL_n\Q$--representations, and that these form a uniformly
stable sequence.  However, we would like to point out
that the Lie algebra $\gr(\IA_n)$ is still not known.

\subsection{Vanishing and finiteness conjectures for the (co)homology
  of $\I_{g,1}$ and $\IA_n$}

We now make a few other natural conjectures concerning the
(co)homology of $\I_{g,1}$ and $\IA_n$.  Our goal is to give as much
of a conjectural picture as possible where there was none before.

\para{A Morita-type conjecture for $\IA_n$} Let $e_i\in
H^i(\I_{g,1};\Q)$ denote the $i^{\text{th}}$ Morita--Mumford--Miller
class restricted to $\I_{g,1}$.  The following is Conjecture 3.4 of
\cite{Mo1}.

\begin{conjecture}[Morita's Conjecture]
  The $\Sp_{2g}\Z$--invariant stable rational cohomology of $\I_{g,1}$ is
  generated as a $\Q$--algebra by $\{e_2,e_4,e_6,\ldots \}$.
\end{conjecture}

Note that all the $e_i$ generate the stable rational cohomology of
$\Mod_{g,1}$, by Madsen--Weiss \cite{MW}, and the odd classes
$e_1,e_3,e_5,\ldots$ vanish when restricted to $\I_{g,1}$. Morita's
Conjecture predicts the trivial representations that can occur in
$H^i(\I_{g,1};\Q)$.  However, it is not known which of the
even Morita--Mumford--Miller classes $e_i$, or combinations thereof, are
nonzero in $H^*(\I_{g,1};\Q)$. Thus even an affirmative answer to
Morita's Conjecture would not imply
Conjecture~\ref{conjecture:torelli} for the trivial representation.

Since Galatius \cite{Ga} has proven that $H^i(\Aut(F_n);\Q)=0$ for
$n\gg i$, it is natural to make the following conjecture.

\begin{conjecture}[Vanishing conjecture]
\label{conjecture:invpart}
  The $\GL_n\Z$--invariant part of the stable rational cohomology of
  $\IA_n$ vanishes.
\end{conjecture}

By the computation $H_1(\IA_n;\Q)\approx\bwedge^2\Q^n\otimes
(\Q^n)^\ast$ for $n\geq 3$, which has no trivial subrepresentations,
Conjecture~\ref{conjecture:invpart} is true for cohomology in
dimension $1$.

\para{Two finiteness conjectures} Bestvina--Bux--Margalit proved that both $H_{3n-2}(\I_{n,1};\Q)$ and
$H_{2n-3}(\IA_n;\Q)$ contain infinite dimensional permutation representations of $\Sp_{2g}\Z$ and $\GL_n\Z$ respectively (see below for details and references). One
might hope that stably such representations do not arise, and that all
irreducible $\Sp_{2n}\Z$--submodules of $H_i(\I_{n,1};\Q)$ and
$\GL_n\Z$--submodules of $H_i(\IA_n;\Q)$ are finite-dimensional for
$n\gg i$.  The limited evidence we have seems to point to the
following.

\begin{conjecture}[Stable finite-dimensionality]
  \label{conjecture:stablyfinite}
  For each $i\geq 1$ and each $n$ sufficiently large (depending on $i$), the
  natural maps 
  \[H_i(\I_{n,1};\Q)^{\fd}\hookrightarrow
  H_i(\I_{n,1};\Q)\]
  and
  \[H_i(\IA_n;\Q)^{\fd}\hookrightarrow
  H_i(\IA_n;\Q)\]
   are isomorphisms.
\end{conjecture}

One may even go so far as to give a conjectural picture of all of the
homology of $\I_{n,1}$ and $\IA_n$, including the infinite-dimensional
part.

\begin{conjecture}[Unstable finite generation]
  \label{conjecture:finitegen}
  For each $i\geq1$ and each $n\geq 1$:
  \begin{enumerate}
  \item The module $H_i(\I_{n,1};\Q)$ is a finitely-generated module
    over $\Sp_{2n}\Z$.
  \item The module $H_i(\IA_n;\Q)$ is a finitely-generated module over
    $\GL_n\Z$.
\end{enumerate}
\end{conjecture}

Note that Conjecture~\ref{conjecture:finitegen} is consistent with all
known computations of the homology groups of $\I_{n,1}$ and $\IA_n$,
including those that are known to be infinite-dimensional over $\Q$.  
Mess \cite[Corollary 1]{Me} proved that $H_1(\I_{2,1};\Q)$
  contains an infinite-dimensional irreducible permutation
  $\Sp_4\Z$--module, and Johnson--Millson showed that $H_3(\I_3;\Q)$
  contains an infinite-dimensional irreducible permutation
  $\Sp_6\Z$--module \cite[Proposition 5]{Me}.  The classes
  in $H_{2n-3}(\IA_n;\Q)$ found by
Bestvina--Bux--Margalit \cite{BBM1} span an infinite-dimensional
subspace, but as a $\GL_n\Z$--module this is a permutation module
generated by a single element; similarly, the classes in
 $H_{3g-2}(\I_{g,1};\Q)$ found by Bestvina--Bux--Margalit
  \cite{BBM2} span a cyclic
  $\Sp_{2g}\Z$--module. In particular, the action of $\GL_n\Z$ or $\Sp_{2g}\Z$ on such a subspace cannot be extended to an action of the corresponding $\Q$--group $\GL_n\Q$ or $\Sp_{2g}\Q$.

\section{Flag varieties, Schubert varieties, and rank-selected posets}
\label{section:flags}

The goal of this section is to demonstrate the appearance of representation stability in 
 the cohomology of various natural families of algebraic
varieties, as well as in algebraic combinatorics. These results are used in \cite{CEF2}  to compute arithmetic statistics for maximal tori in $\GL_n(\F_q)$ and Lagrangian tori in $\Sp_{2g}(\F_q)$.

\subsection{Cohomology of flag varieties}

Let $\cF_n$ be the complete flag variety parametrizing complete flags
in $\C^n$; this can be identified with $G/B$ where $G=\GL_n\C$ and $B$
is the Borel subgroup consisting of upper triangular matrices. The
inclusion $\GL_n\C\hookrightarrow\GL_{n+1}\C$ induces an inclusion of
$\cF_n$ as a closed subvariety of $\cF_{n+1}$. In terms of flags, this
amounts to regarding a complete flag $V_1<\cdots<V_n=\C^n$ as a flag
in $\C^{n+1}$ by appending $\C^{n+1}$ itself. The unitary group $U(n)$
also acts on $\cF_n$, with stabilizer a maximal torus $T$, giving an
identification of $\cF_n$ with $U(n)/T$. The normalizer $N(T)$ acts on
$U(n)/T$ on the right, which factors through an action of the Weyl
group $W=N(T)/T$. In this case $W$ can be identified with the group
$S_n$ of permutation matrices, so we obtain an $S_n$--action on
$\cF_n$, and thus an $S_n$--action on $H^i(\cF_n;\Q)$ for each $i\geq 0$.

The inclusion $\cF_n\hookrightarrow \cF_{n+1}$ induces for each $i\geq
0$ a homomorphism $H^i(\cF_{n+1};\Q)\to H^i(\cF_n;\Q)$, and the
sequence $\{H^i(\cF_n;\Q)\}$ is easily seen to be a consistent
sequence of $S_n$--representations. We will prove that this sequence
is representation stable in the sense of
Definition~\ref{definition:repstabrev}.

\begin{theorem}[Stability for the cohomology of flag varieties]
\label{thm:flag}
For each fixed $i\geq 0$, the sequence $\{H^i(\cF_n;\Q)\}$ of
$S_n$--representations is representation stable.
\end{theorem}
\begin{proof}
  The cohomology $H^*(\cF_n;\Q)$ is described as follows. The trivial
  bundle $\cF_n\times \C^n$ is filtered by $k$--dimensional subbundles
  $U_i$ for $0\leq i\leq n$, where $U_i$ over a given flag is the
  $i$th subspace of that flag. The quotients $E_i\coloneq U_i/U_{i-1}$
  are line bundles over $\cF_n$.  Let $x_i\in H^2(\cF_n;\Q)$ be the
  first Chern class $c_1(E_i)$. These classes $\{x_i\}$ generate
  $H^*(\cF_n;\Q)$, as we will see in more detail below.   
  $S_n$ acts on $H^2(\cF_n;\Q)$ by permuting the generators $x_i$.

  We are trying to prove representation stability in the sense of
  Definition~\ref{definition:repstabrev}.  First note that $x_i\in
  H^2(\cF_{n+1};\Q)$ restricts to $x_i\in H^2(\cF_n;\Q)$. A basis for
  $H^i(\cF_n;\Q)$ is given by $\mathcal{B}_n=\{x_1^{j_1}\cdots
  x_n^{j_n}|0\leq j_k<k\}$ (see \cite[Proposition 10.3]{Fu}). Thus 
  the subset of $\mathcal{B}_{n+1}$ consisting of elements with
  $j_{n+1}=0$ restricts bijectively to the basis
  $\mathcal{B}_n$.  This gives the surjectivity condition of Definition~\ref{definition:repstabrev}.    
  Now, as long as $n>i$, any element of
  $\mathcal{B}_{n+1}$ with degree $i$ can be rearranged by a
  permutation in $S_{n+1}$ to have $j_{n+1}=0$ while still satisfying
  $0\leq j_k<k$ for all $k$. This shows that for large enough $n$, the
  $S_{n+1}$--orbit of the degree $i$ terms of this subset spans
  $H^i(\cF_{n+1};\Q)$, as desired.  This gives the injectivity condition of Definition~\ref{definition:repstabrev}.

  Proving stability of multiplicities is more involved.  A general
  theorem of Borel \cite{Bo1} states that the cohomology
  $H^*(\cF_n;\Q)$ is isomorphic to the co-invariant algebra on the
  $x_i$, defined as follows. Let $\Q[x_1,\ldots,x_n]^{S_n}$ be the
  ring of symmetric polynomials, and let $I_n$ be the ideal of
  $\Q[x_1,\ldots,x_n]$ generated by all symmetric polynomials with
  zero constant term. The \emph{co-invariant algebra}
  $R[x_1,\ldots,x_n]$ is defined to be the quotient
  \[R[x_1,\ldots,x_n]:=\Q[x_1,\ldots,x_n]/I_n.\] Thus
  $R[x_1,\ldots,x_n]$ inherits a natural grading from
  $\Q[x_1,\ldots,x_n]$, and $H^*(\cF_n;\Q)$ is isomorphic to
  $R[x_1,\ldots,x_n]$ as a graded $S_n$--module (see \cite[Proposition
  10.3]{Fu} for a combinatorial proof). It is not hard to see that
  $R[x_1,\ldots,x_n]$, and thus $H^*(\cF_n;\Q)$, is in fact isomorphic
  to the regular representation $\Q S_n$, which is \emph{not} representation
  stable. However, looking at each homogeneous piece individually, we
  have the following theorem of Stanley, Lusztig, and
  Kraskiewicz--Weyman:
  
  \begin{theorem}[\cite{Re}, Theorem 8.8]
    \label{theorem:kw}
    For any partition $\lambda$, as long as $i\leq \binom{n}{2}$, the
    multiplicity of $V(\lambda)_n$ in $R_i[x_1,\ldots,x_n]$ equals the
    number of standard tableaux of shape $\lambda[n]$ with major index
    equal to $i$.
  \end{theorem}
  
  Recall that a \emph{standard tableau of shape $\lambda$} is a
  bijective labeling of the boxes of the Young diagram for $\lambda$
  by the numbers $1,\ldots,n$ with the property that in each row and
  in each column the labels are increasing. Given such a labeling, the
  \emph{descent set} is the set of numbers $i$ so that the box labeled
  $i+1$ is in a lower row than the box labeled $i$. The \emph{major
    index} of a tableau is the sum of the numbers in the descent set.

  Fix a partition $\lambda$ and a finite set $S\subset \N$.  Let
  $\mathcal{T}_n$ be the set of standard tableaux of shape
  $\lambda[n]$ with descent set exactly $S$. We will show below that
  for sufficiently large $n$, the size of $\mathcal{T}_n$ is equal to
  the size of $\mathcal{T}_{n+1}$. Since only finitely many $S\subset
  \N$ have $\sum_{j\in S}j=i$, applying Theorem~\ref{theorem:kw} once
  $n$ is sufficiently large will prove that the multiplicity of
  $V(\lambda)_n$ in $R_i[x_1,\ldots,x_n]$ is eventually constant, as
  desired.

  First we exhibit an injection from $\mathcal{T}_n$ into
  $\mathcal{T}_{n+1}$. Note that the Young diagram for $\lambda[n+1]$
  is obtained from that of $\lambda[n]$ by adding an additional box at
  the end of the first row. Our operation on tableaux will be simply
  to fill this newly-added box with $n+1$. Since neither $n$ nor $n+1$
  can be a descent in the resulting tableau, and whether any other $j$ is a
  descent remains unchanged, the descent set is unchanged by this
  operation. Thus this operation, which is clearly injective, maps
  $\mathcal{T}_n\to\mathcal{T}_{n+1}$.

  It remains to show that for sufficiently large $n$, the operation is
  also surjective. Equivalently, we must show that for sufficiently
  large $n$, any tableau of shape $\lambda[n]$ with descent set $S$
  has the label $n$ in the top row. Let $k=\max S$.  If the label $n$
  is not in the top row, then no label greater than $k$ can be in the
  top row, for otherwise at least one number between $k$ and $n-1$
  would be a descent. But exactly $|\lambda|$ boxes are not contained
  in the first row of $\lambda[n]$. Thus taking $n$ greater than
  $k+|\lambda|$, the pigeonhole principle implies that in every
  tableau some label greater than $k$ appears in the top row. It
  follows that any tableau with descent set $S$ has the label $n$ in
  the top row, as desired.

  Applying Theorem~\ref{theorem:kw}, we see that the multiplicity of
  $V(\lambda)_n$ in $H^i(\cF_n;\Q)$ is eventually independent of $n$,
  as desired.
\end{proof}

It can be seen from the proof of Theorem~\ref{thm:flag} that $H^i(\cF_n;\Q)=R_i[x_1,\ldots,x_n]$ is
in fact uniformly representation stable.

\para{Lagrangian flags}
Let $\cF'_n$ be the flag variety parametrizing pairs of a Lagrangian
subspace $L$ of $\C^{2n}$, together with a complete flag on $L$. For
$G=\Sp_{2n}\C$ and $B$ a Borel subgroup, $\cF'_n$ is identified with
$G/B$. The Weyl group in this case is the hyperoctahedral group $W_n$.  
Borel proved in \cite{Bo1} that $H^*(\cF'_n;\Q)$ is isomorphic to the
co-invariant algebra for $W_n$.

\begin{theorem}\label{thm:flagsp}
  For each fixed $i\geq 0$, the sequence $\{H^i(\cF'_n;\Q)\}$ of
  $W_n$--representations is representation stable (in the sense of
  Definition~\ref{definition:repstabrev}).
\end{theorem}
\begin{proof}
  Given a double partition $\lambda=(\lambda^+,\lambda^-)$, Stembridge
  \cite[Theorem 5.3]{Ste} generalized Stanley's theorem and proved
  that the multiplicity of $V(\lambda)_n$ in the $i^{\text{th}}$
  graded piece of the co-invariant algebra for $W_n$ is the number of
  double standard Young tableaux of shape $\lambda[n]$ whose flag
  major index is $i$, as long as $n^2\geq i$. We now summarize the
  necessary terminology. If $|\lambda^-|=k$, recall that
  $\lambda[n]=(\lambda^+[n-k],\lambda^-)$.  A \emph{double standard
    Young tableau} is a bijective labeling by the labels $1,\ldots,n$
  of the diagrams for $\lambda^+[n-k]$ and $\lambda^-$ together, which
  within each diagram is increasing on each row and column. The
  \emph{flag descent set} can be described as follows. Place the
  diagram for $\lambda^-$ above the diagram for $\lambda^+[n-k]$. Then
  the flag descent set consists of those $j$ for which $j+1$ appears
  below $j$ in the tableau, together with $n$ if and only if $n$
  appears in the diagram for $\lambda^-$. Finally, the \emph{flag
    major index} is \[2\sum j+|\lambda^-|,\] where the sum is over those
  $j$ in the flag descent set.

  As in the proof of Theorem~\ref{thm:flag}, it will suffice to prove
  that for each double partition $\lambda$ and each finite set
  $S\subset \N$, the number of double standard tableaux of shape
  $\lambda[n]$ with flag descent set $S$ is eventually
  constant. Passing from double tableaux of shape $\lambda[n]$ to
  $\lambda[n+1]$ requires adding a box to the first row of
  $\lambda^+[n-k]$; we always fill that box with $n+1$. Call this the
  \emph{main row} of the diagram. Note that the definition of flag descent
  set is such that this operation does not change the descent
  set. Thus it suffices to show that for sufficiently large $n$, every
  double standard Young tableau of shape $\lambda[n]$ having flag
  descent set $S$ has $n$ in the main row. When $n$ is larger than $\max
  S$ it cannot appear in the diagram for $\lambda^-$ above the main
  row. But there are exactly $|\lambda^+|$ boxes below the main row.
  So once $n\geq |\lambda^+|+\max S$, if $n$ were below the main row,
  some number larger than $\max S$ would appear in the descent
  set. Thus for sufficiently large $n$, the label $n$ must appear in
  the main row, as desired.

  Since only finitely many descent sets $S\subset N$ have associated
  flag major index $i$, we conclude that for each double partition
  $\lambda$, the multiplicity of $V(\lambda)_n$ in $H^i(\cF'_n;\Q)$ is
  eventually constant. Injectivity and surjectivity follow as in the
  proof of Theorem~\ref{thm:flag}, so we conclude that
  $\{H^i(\cF'_n;\Q)\}$ is representation stable.
\end{proof}

\subsection{Cohomology of Schubert varieties}
\label{section:schubert}

Recall from above that $\cF_n=G/B$ is the variety of complete flags in $\C^n$, where $G=\GL_n\C$ and $B$ is a Borel subgroup; $G$ naturally acts on $\cF_n=G/B$ by left multiplication. Choosing the standard flag in $\C^n$ as a basepoint, each permutation
$w\in S_n$ determines a flag, which can be identified with $[w]\in
G/B$. The orbits of
the flags $[w]$ under the Borel subgroup $B$ are the Bruhat cells
$BwB$. The \emph{Schubert variety} $X_w$ associated to $w$ is the
closure $\overline{B[w]}$ in $G/B$ of the Bruhat cell $BwB$. 

Let $T$
be a maximal torus in $G$.  Then the $G$--action on $G/B$ restricts to
a $T$--action and this $T$--action preserves $X_w$.  We denote by
$H^*_T(X_w;\Q)$ the equivariant cohomology with respect to $T$.  
There is an action of $S_n$ on $H^*_T(X_w;\Q)$, which is somewhat involved 
to describe; it is given in Tymoczko \cite{Ty}.

Given $w\in S_n$, we can view it as an element of $S_{n+1}$ by the usual inclusion; let $X_w[n+1]$ be the corresponding Schubert variety in $\cF_{n+1}$, and so on. Then the equivariant cohomology $\{H^*_T(X_w[n];\Q)\}$ is a consistent sequence of $S_n$--representations.

\begin{theorem}[Stability for the cohomology of Schubert varieties]
\label{theorem:schubert}
Let $w$ be any permutation.  Then for each fixed $i\geq 0$ the
sequence $\{H^i_T(X_w;\Q)\}$ of $S_n$--representations is
multiplicity stable.
\end{theorem}

\begin{proof}[Proof of Theorem~\ref{theorem:schubert}]
  For $v\in S_n$, let $\ell(v)$ denote the
  length of $v$ with respect to the standard Coxeter generators. For a
  graded ring $M$ let $M[n]$ denote the shift in grading by
  $n$. Tymoczko proved \cite[Theorem 1.1]{Ty} that
  \[H^*_T(X_w;\Q)=\bigoplus_{[v]\in X_w}\Q[t_1,\ldots,t_n][\ell(v)]\]
  as graded $S_n$--modules.  Here the sum is over those permutations
  $v\in S_n$ whose image $[v]$ lies in $X_w$. It is standard (see, e.g.,
  \cite[Proposition 10.7]{Fu}) that these are exactly the $v$ for
  which $v\leq w$ in the Bruhat partial order. The Bruhat order has
  the property that $v\leq w$ in $S_n$ if and only if $v\leq w$ when
  considered as elements of $S_{n+1}$. Thus for fixed $w$ the
  collection of $v$ in the sum is independent of $n$; similarly the
  lengths $\ell(v)$ do not change. Denote the degree
$i$ homogeneous polynomials over $\Q$ by $P_i[x_1,\ldots,x_n]$. Since
  \[H^i_T(X_w;\Q)=\bigoplus_{[v]\in
    X_w}P_{i-\ell(v)}[t_1,\ldots,t_n],\] it suffices to prove that the
  homogeneous polynomials $\{P_i[x_1,\ldots,x_n]\}$ are representation
  stable for each $i\geq 0$.

  As an aside, we remark that the surjection $H^i_T(X_w;\Q)\to
  H^i(X_w;\Q)$ is given by mapping each $t_i\mapsto 0$, so
  \[H^i(X_w;\Q)= \bigoplus_{\substack{[v]\in X_w,\\\ell(v)=i}}\Q.\]
  Combining this with the preceding discussion, we see that classical
  homological stability holds for the ordinary cohomology
  $\{H^i(X_w;\Q)\}$ of Schubert varieties.

  Note that $P_i[x_1,\ldots,x_{n+1}]$ is spanned by monomials which
 involve at most $i$ variables; thus for $n\geq i$ any such monomial
  is the image under $S_{n+1}$ of a monomial in
  $P_i[x_1,\ldots,x_n]$. This verifies surjectivity, and injectivity
  is immediate.
Let  \[\Lambda[x_1,\ldots,x_n]\coloneq \Q[x_1,\ldots,x_n]^{S_n}\] be the ring of symmetric polynomials; $\Q[x_1,\ldots,x_n]$ is a free
  $\Lambda[x_1,\ldots,x_n]$--module, and in
  fact \[\Q[x_1,\ldots,x_n]\approx
  R[x_1,\ldots,x_n]\otimes_{\Q}\Lambda[x_1,\ldots,x_n]\] as graded
  $S_n$--modules (see, e.g., the proof of \cite[Theorem 8.8]{Re}). It
  follows that \[P_i[x_1,\ldots,x_n]\approx \bigoplus_{j+k=i}
  R_j[x_1,\ldots,x_n]^{\oplus \dim \Lambda_k[x_1,\ldots,x_n]}\] as
  $S_n$--representations. We can see that the dimension $\dim\Lambda_k[x_1,\ldots,x_n]$ is eventually
  constant as follows. It is classical that the ring of symmetric
functions 
is a polynomial algebra $\Q[e_1,\ldots,e_n]$ on the elementary
symmetric polynomials $\{e_j\}$. Since the degree of $e_j$ is $j$, we
see that once $n$ is larger than $i$, the dimension of
$\Lambda_i[x_1,\ldots,x_n]$ is the number of partitions of $i$ and
thus does not depend on $n$.

For any $\lambda$ the multiplicity of
  $V(\lambda)_n$ in $R_j[x_1,\ldots,x_n]$ is eventually constant by
  Theorem~\ref{thm:flag}.  Since there are finitely many solutions to $j+k=i$ once
  $i$ is fixed, we may assume all these multiplicities have stabilized
  for $n$ large enough. We conclude that the multiplicity of
  $V(\lambda)_n$ in $P_i[x_1,\ldots,x_n]$ is eventually constant, as
  desired.
\end{proof}

\begin{remark} The results of Tymoczko quoted in the proof of
  Theorem~\ref{theorem:schubert} hold more generally for other
  semisimple groups $G$, replacing the polynomial algebra with the 
  $W$--algebra induced by the coadjoint action on the root system
  \cite[Theorem 4.10]{Ty}; here $W$ is the Weyl group of $G$.  We believe that it should be possible to
  prove representation stability for the equivariant cohomology of the
  corresponding Schubert varieties.
\end{remark}

\subsection{Rank-selected posets}
\label{section:lefschetz}

The poset $Z_n$ of subsets of the finite set $\{1,\ldots,n\}$,
ordered by inclusion, is a basic object of study in combinatorics.
The group $S_n$ acts on $\{1,\ldots, n\}$, inducing an action on
$Z_n$.  One can view this action as an analogue of the $S_n$--action
on the flag variety ${\cal F}_n$.  In this subsection we prove some
stability results for some refinements of these actions on the
associated cohomology groups.

Suppose $G$ is a group acting on an $n$--dimensional space $X$.  The \emph{Lefschetz representation} associated to this
action is the virtual $G$--representation
\[\sum_{i=0}^n(-1)^i H_i(X;\Q),\]
meaning the formal linear combination of the representations
$H_i(X;\Q)$.  The name reflects the observation that for each $g\in
G$, the associated \emph{virtual character} is the Lefschetz number
\[\sum_{i=0}^n(-1)^i  \tr\big(g_*\colon H_i(X;\Q)\to H_i(X;\Q)\big).\]

For any finite set $S\subset \N$ we may consider the
\emph{rank-selected} poset $Z_n(S)$. This is the poset consisting of
 $\emptyset$ and $\{1,\ldots,n\}$, together with those
subsets of $\{1,\ldots,n\}$ whose cardinality lies in $S$. Let
$|Z_n(S)|$ be the geometric realization of this poset. The natural
action of the symmetric group $S_n$ on $Z_n$ preserves the subposet
$Z_n(S)$, yielding an action of $S_n$ on the geometric realization
$|Z_n(S)|$. Let $L_n(S)$ be the associated Lefschetz
representation \[L_n(S)\coloneq \sum_i(-1)^i H_i(|Z_n(S)|;\Q).\] 


\begin{theorem}[Stability for Lefschetz representations of
  rank-selected posets]
  Let $S\subset \N$ be any finite set.  Then the sequence $\{L_n(S)\}$
  of virtual $S_n$--representations is multiplicity stable.
\end{theorem}

\begin{proof}
  Consider the related virtual representation
  \[L'_n(S):=(-1)^{|S|-1}(L_n(S)\oplus \Q).\] Clearly $\{L'_n(S)\}$ is
  multiplicity stable if and only if $\{L_n(S)\}$ is multiplicity
  stable. Given a partition $\lambda$, Stanley \cite[Theorem 4.3]{Sta}
  proves that the multiplicity of $V(\lambda)_n$ in $L'_n(S)$ equals
  the number of standard Young tableaux with shape $\lambda[n]$ whose
  descent set is exactly $S\cap \{1,\ldots,n-1\}$. As we saw in the
  proof of Theorem~\ref{thm:flag}, this implies that the multiplicity
  of $V(\lambda)_n$ is constant for sufficiently large $n$, as desired.
\end{proof}

Let $C_n$ be the $n$--dimensional \emph{cross-polytope}, i.e.\ the
convex hull of the set of unit coordinate vectors $\{\pm e_1,\ldots,
\pm e_n\}$ in $\R^n$.  Let $Q_n$ be the poset of \emph{faces} of
$C_n$, meaning convex hulls of subsets of vertices. For $S\subset \N$,
let $Q_n(S)$ be the rank-selected poset consisting of faces whose
dimension lies in $S$, together with $\emptyset$ and $C_n$. The hyperoctahedral group $W_n$ naturally acts on $C_n$,
and thus on the poset $Q_n(S)$ and its geometric realization
$|Q_n(S)|$. Let $L^C_n(S)$ be the associated Lefschetz representation
\[L^C_n(S)\coloneq\sum_i(-1)^i H_i(|Q_n(S)|;\Q).\]

\begin{theorem}[Stability for Lefschetz representations of
  rank-selected cross-polytopes]
  Let $S\subset \N$ be any finite set.  Then the sequence
  $\{L^C_n(S)\}$ of virtual $W_n$--representations is multiplicity
  stable.
\end{theorem}

\begin{proof}
  Given a double partition $\lambda=(\lambda^+,\lambda^-)$, Stanley
  \cite[Theorem 6.4]{Sta} shows that the multiplicity of
  $V(\lambda)_n$ in $(-1)^{|S|-1}(L^C_n(S)\oplus \Q)$ is the number of
  double standard Young tableaux of shape $\lambda[n]$ whose flag
  descent set is exactly $S\cap \{1,\ldots,n-1\}$. As we showed in the
  proof of Theorem~\ref{thm:flagsp}, this implies that the
  multiplicity of $V(\lambda)_n$ is constant for sufficiently large
  $n$, as desired.
\end{proof}

\subsection{The $(n+1)^{n-1}$ conjecture}
\label{section:nplusone}

There is a variation of the co-invariant algebra (discussed in the
proof of Theorem~\ref{thm:flag} above) that has been intensely studied
by combinatorialists. The symmetric group $S_n$ acts on
$\Q[x_1,\ldots,x_n,y_1,\ldots ,y_n]$ diagonally, permuting the $x_\bullet$
and the $y_\bullet$ separately. The \emph{diagonal co-invariant algebra} is
the $\Q$--algebra defined by:
\[R_n:=\Q[x_1,\ldots ,x_n,y_1,\ldots ,y_n]/I_n\]
\noindent
where $I_n$ denotes the ideal generated by the $S_n$--invariant
polynomials without constant term. The bigrading of $\Q[x_1,\ldots
,x_n,y_1,\ldots ,y_n]$ by total degree in $\{x_\bullet\}$ and total
degree in $\{y_\bullet\}$ descends to a bigrading $(R_n)_{i,j}$ of the
algebra $R_n$. This bigrading is preserved by the action of $S_n$ on
$R_n$.  The \emph{$(n+1)^{n-1}$ conjecture} was the conjecture
that \[\dim(R_n)=(n+1)^{n-1}.\] This conjecture was proved by Haiman
(see, e.g., the survey \cite{Hai}), using a connection between this
problem and the geometry of the Hilbert scheme of configurations of
$n$ points in $\C^2$.  Just as with the classical co-invariant
algebra, the structure of $R_n$ as an $S_n$--representation has been
determined \cite[Theorem 4.24]{Hai}.  However, the following seems to
be unknown.  It can be viewed as an ``asymptotic refinement'' of the
$(n+1)^{n-1}$ conjecture.

\begin{question}
\label{question:coinv}
  Is the sequence of $S_n$--representations
  $\{(R_n)_{i,j}\}$ representation stable for each fixed $i,j\geq 1$?  \footnote{Since this paper was first posted, this question has been answered affirmatively by the authors and Ellenberg in \cite[Section 3]{CEF1}, along with the generalization to $R_n^{(k)}$.}
\end{question}

This question has a
natural generalization to the ``$k$--diagonal co-invariant algebra''
$R_n^{(k)}$ for $k\geq 3$, by which we mean the algebra defined by the
same construction as above, with $kn$ variables partitioned into $k$
subcollections and $S_n$ acting diagonally on each subcollection
separately.  In this case the dimension of $R^{(k)}_n$ is not known.
It would be especially interesting if representation
stability as in Question
\ref{question:coinv} could be proved without knowing the irreducible
decomposition, or even the dimension, of $R^{(k)}_n$.

\section{Congruence subgroups, modular representations and stable
  periodicity}
\label{section:congruence}
Recall that a \emph{modular representation} of a finite group $G$ is an action of $G$ on a vector space over a field of positive characteristic dividing the order of $G$. Such representations need not decompose as a direct sum of irreducible representations and in general are very difficult to analyze. For finite groups of Lie type, for example $G=\SL_n(\F_p)$, the modular representation theory is significantly better understood in the \emph{defining characteristic} of $G$, meaning in this case over a field of characteristic $p$.
There are a number of important examples of groups $\Gamma$ whose cohomology $H^i(\Gamma;\F_p)$ is naturally a modular representation of a finite group of Lie type.   Examples of such $\Gamma$ include various congruence subgroups 
of arithmetic groups as well as congruence subgroups of mapping class
groups.

After explaining in detail a key motivating example, we briefly
review the modular representation theory that will be needed to
formulate representation stability in this context.  One new
phenomenon here is that natural sequences of representations arise
that do not satisfy representation stability, but instead exhibit a
form of ``stable periodicity'' as representations.  After defining
this precisely, we present several results and conjectures using this
concept.

\subsection{A motivating example}
Consider the following fundamental example from arithmetic. For any
prime $p$ the \emph{level $p$ congruence subgroup}
$\Gamma_n(p)<\SL_n\Z$ is the kernel
\[\Gamma_n(p)\coloneq \ker (\pi\colon\SL_n\Z\twoheadrightarrow\SL_n(\F_p))\]
where $\pi$ is the map reducing the entries of a matrix modulo
$p$. Charney proved in \cite{Ch} that over $\Q$ (indeed even over
$\Z[1/p]$) the sequence of groups $\{\Gamma_n(p)\}$ satisfy classical
homological stability. Furthermore, she proved that this is equivalent
to the claim that the natural action of $\SL_n(\F_p)$ on
$H^i(\Gamma_n(p);\Q)$ is trivial for large enough $n$, so that
\[H^i(\Gamma_n(p);\Q)^{\SL_n(\F_p)}
=H^i(\Gamma_n(p);\Q)=H^i(\SL_n\Z;\Q).\]

Replacing the coefficient field $\Q$ with $\F_p$ or its algebraic
closure $\Fpbar$, the situation becomes more interesting, and the
cohomology is much richer (see, e.g., \cite{Ad,As}).  First note that Charney's result is not
true in this case: the action of $\SL_n(\F_p)$ on
$H^i(\Gamma_n(p);\Fpbar)$ is certainly not trivial.  We can work this
out for $H_1(\Gamma_n(p);\F_p)$ explicity.  Each $B\in\Gamma_n(p)$ can
be written as $B=I+pA$ for some $A$.  It is easy to check that the map
$B\mapsto A\pmod{p}$ gives a surjective homomorphism
\begin{equation}
\label{eq:abelianize1}
\psi\colon  \Gamma_n(p)\to\fsl_n(\F_p)
\end{equation} 
where $\fsl_n(\F_p)$ is the abelian group of traceless $n\times n$
matrices with entries in $\F_p$.  Lee--Szczarba \cite{LSz} observed
that the proof of the Congruence Subgroup Property implies that $\psi$
yields an isomorphism
\[H_1(\Gamma_n(p);\Z)\approx H_1(\Gamma_n(p);\F_p)\approx
\fsl_n(\F_p).\] We thus see that, since the dimension of
$H_1(\Gamma(n,p);\F_p)$ increases with $n$, the sequence of groups
$\{\Gamma_n(\F_p)\}$ does not satisfy homological stability over
$\F_p$ in the classical sense.  However, it is clear from the
construction that the $\SL_n(\F_p)$--action on
$H_1(\Gamma_n(p);\F_p)\approx \fsl_n(\F_p)$ is just the usual
(modular) \emph{adjoint representation}; this is a modular
representation because $\fsl_n(\F_p)$ is a vector space over $\F_p$,
and $p$ divides the order of $\SL_n(\F_p)$.  We can thus hope to use
the modular representation theory of $\SL_n(\F_p)$ to define and study
a version of representation stability for each sequence
$\{H^i(\Gamma_n(p),\F_p)\}$ of $\SL_n(\F_p)$--representations.  For
example, $\fsl_n(\F_p)$ is an irreducible
$\SL_n(\F_p)$--representation, and so an appropriate form of
representation stability holds for $\{H_1(\Gamma_n(p);\F_p)\}$.

One can do all of the above for level $p$ congruence subgroups
$\Gamma^{\Sp}_{2g}(p)$ of $\Sp_{2g}\Z$.  As we will explain in
\S\ref{section:periodicity}, something new happens here: the sequence
of $\Sp_{2g}\F_2$--representations $\{H_1(\Gamma^{\Sp}_{2g}(2);\F_2)\}$
is only representation stable when restricted to even $g$, or to odd
$g$.  Indeed, for each $p\geq 2$ we will see below natural
examples of sequences that are ``stably periodic'' with period $p$.

\subsection{Modular representations of finite groups of Lie type}

In order to formalize the notion of representation stability in the
modular case, we need to review the pertinent representation theory.

\para{Representations of \boldmath$\SL_n(\Fpbar)$ and
  \boldmath$\Sp_{2n}(\Fpbar)$ in their defining characteristic}
Before restricting to the finite group $\SL_n(\F_p)$, we consider
representations of the algebraic group $\SL_n(\Fpbar)$ in the defining
characteristic $p$. While it is not true in this context that every
representation is completely reducible, irreducible representations of
$\SL_n(\Fpbar)$ over $\Fpbar$ are still classified by highest weights,
as follows.  We give the details for the case of $\SL_n$, but all
claims hold for $\Sp_{2n}$ as well. A nice reference for these
assertions is Humphreys \cite[Chapters 2 and 3]{Hu}.

Let $T<\SL_n(\Fpbar)$ be the maximal torus consisting of diagonal
matrices.  Let ${U<\SL_n(\Fpbar)}$ be the subgroup of strictly
upper-triangular matrices. Any representation $V$ of $\SL_n(\Fpbar)$
decomposes into eigenspaces for $T$. A vector $v\in V$ is called a
\emph{highest weight vector} if $v$ is an eigenvector for $T$ and is
invariant under $U$, in which case its \emph{weight} is the
corresponding eigenvalue $\lambda\in T^*$. Writing $T^*$ additively,
we identify $T^*$ with $\Z[L_1,\ldots,L_n]/(L_1+\cdots+L_n)$. The same
applies to $\Sp_{2n}(\Fpbar)$, with $T^*=\Z[L_1,\ldots,L_n]$. In
either case, a weight is called \emph{dominant} if it can be written
as a nonegative integral combination of the fundamental weights
$\omega_i=L_1+\cdots+L_i$.

The basics of the classification of irreducible
$\SL_n(\Fpbar)$--representations are the same as in the characteristic
0 case: every irreducible representation contains a unique highest
weight vector; the highest weight $\lambda$ determines the irreducible
representation; and every dominant weight occurs as the highest weight
of an irreducible representation. Thus we may unambiguously denote by
$V(\lambda)_n$ the irreducible representation of $\SL_n(\Fpbar)$ or
$\Sp_{2n}(\Fpbar)$ with highest weight $\lambda$. However, much less
is known about these irreducible representations than in the
characteristic $0$ case, and there is no known way to uniformly
construct all irreducible representations. Even the dimensions of the
irreducible representations are not known in general.

One approach to the construction of irreducible
$\SL_n(\Fpbar)$--representations $V(\lambda)$ is through Weyl
modules. This process starts with the irreducible representation
$V(\lambda)_\Q$ of $\SL_n\Q$ with weight $\lambda$. There is then a
special $\Z$--form $V(\lambda)_\Z\subset V(\lambda)_\Q$ so that
$\SL_n\Fpbar$ acts on the \emph{Weyl module} $W(\lambda)\coloneq
V(\lambda)_\Z\otimes \Fpbar$. The Weyl module $W(\lambda)$ is
generated by a single highest weight vector with weight $\lambda$, but
in general $W(\lambda)$ will not be irreducible. However, $W(\lambda)$
always admits a unique simple quotient, which must be the irreducible
representation $V(\lambda)$. We will see below that for fixed
$\lambda$, the question of whether $W(\lambda)$ is irreducible can
depend on the residue of $n$ modulo $p$.

\para{Restriction to finite groups of Lie type}
Given any representation of $\SL_n(\Fpbar)$, we may ``twist'' it by
precomposing with the Frobenius map $\SL_n(\Fpbar)\to
\SL_n(\Fpbar)$. This twisted representation clearly remains
irreducible; in fact for any $\lambda$ the twist of $V(\lambda)_n$ by
the Frobenius is $V(p\lambda)_n$.
A dominant weight $\lambda$ is called \emph{$p$--restricted} if it can
be written as $\lambda=\sum c_i\omega_i$ with $0\leq c_i<p$.
If $\lambda$ is $p$--restricted, then the restriction of the
irreducible representation $V(\lambda)_n$ from $\SL_n(\Fpbar)$ to
$\SL_n(\F_p)$ remains irreducible.  Every irreducible
representation of $\SL_n(\F_p)$ is of this form. Thus we have found
all $p^{n-1}$ irreducible representations of $\SL_n(\F_p)$ and all
$p^n$ irreducible representations of $\Sp_{2n}(\F_p)$.

\para{Uniqueness of composition factors} In the modular case we cannot
decompose a representation into a direct sum of irreducibles.
However, by the Jordan--H\"older theorem, the irreducible
representations that occur as the composition factors in any
Jordan--H\"older decomposition of any representation are indeed
unique.

\subsection{Stable periodicity and congruence subgroups}
\label{section:periodicity}

The definition of representation stability in the modular case needs
to be altered in a fundamental way in order to apply to several
natural examples.  One of these examples is the \emph{level $p$
  symplectic congruence subgroup} $\Gamma_{2g}^{\rm Sp}(p)<\Sp_{2g}\Z$
defined as the kernel
\[\Gamma_{2g}^{\rm Sp}(p)\coloneq \ker (\pi\colon\Sp_{2g}\Z\twoheadrightarrow\Sp_{2g}\F_p)\]
where $\pi$ is the map reducing the entries of a matrix modulo the
prime $p$.  Building on work of Sato, Putman \cite{Pu1} has shown, among many other things, that
for $g\geq 3$ and $p$ odd there is an $\Sp_{2g}\Z$--equivariant
isomorphism:
\[H_1(\Gamma_{2g}^{\rm Sp}(p),\Z)\approx H_1(\Gamma_{2g}^{\rm
  Sp}(p),\F_p)\approx \fsp_{2g}(\F_p)\] where $\fsp_{2g}(\F_p)$ is the
adjoint representation of $\Sp_{2g}(\F_p)$ on its Lie algebra.  Putman
also proved that the group $H_1(\Gamma_{2g}^{\rm Sp}(2),\F_2)$ is an
extension of $ \fsp_{2g}(\F_2)$ by $H:=H_1(S_g;\F_2)$.

Note that $\fsp_{2g}\F_p$ sits inside $\fgl_{2g}\F_p\approx H^*\otimes
H\approx H\otimes H$ as $\fsp_{2g}\F_p\approx \Sym^2 H$.  When $p$ is
odd, $\fsp_{2g}\F_p\approx \Sym^2 H$ is irreducible with highest
weight vector $a_1\cdot a_1$ and highest weight $2\omega_1$ (see
Hogeweij \cite[Corollary 2.7]{Ho}).  The situation is different for
$p=2$: the representation $\fsp_{2g}\F_2\approx \Sym^2 H$ is no longer
irreducible. Indeed since \[(x+y)^2=x^2+2x\cdot y+y^2=x^2+y^2\] there
is an embedding $H\hookrightarrow \Sym^2 H$ defined by $x\mapsto
x\cdot x$; this is a map of $\Sp_{2g}(\F_2)$--representations since
$a^2=a$ in $\F_2$.  Recalling that $a_1\cdot a_1$ has highest weight
$2\omega_1$, over $\overline{\F}_2$ we see here the isomorphism
between $V(2\omega_1)$ and the twist of $V(\omega_1)\approx H$ by the
Frobenius map $a\mapsto a^2$.  Since $x\cdot y=y\cdot x=-y\cdot x$,
the quotient $\Sym^2 H/H$ is isomorphic to $\bwedge^2 H$.  This has an
invariant contraction $\bwedge^2 H\to \F_2$ (represented by the
symplectic form) and an invariant vector $\omega=a_1\cdot b_1+\cdots
a_g\cdot b_g$ (representing the symplectic form).  These are
independent when $g$ is odd, but not when $g$ is even.  Thus
$\fsp_{2g}\F_2$ has composition factors $V(0),V(\omega_1),V(\omega_2)$
if $g$ is odd, and $V(0)^2,V(\omega_1),V(\omega_2)$ if $g$ is even (see
\cite[Lemma 2.10]{Ho}).

In order to take situations like this into account, we must build periodicity into the
definition of stability.

\begin{definition}[Stable periodicity]
  Let $G_n=\SL_n(\F_p)$ or $\Sp_{2n}(\F_p)$.  Let $\{V_n\}$ be a
  consistent (c.f.\ \S\ref{section:repstab:def}) sequence of modular
  $G_n$--representations, i.e.\ representations of vector spaces over
  $\F_p$.  The sequence $\{V_n\}$ is \emph{stably
    representation periodic}, or just \emph{stably periodic}, if
  Condition I (Injectivity) and Condition II (Surjectivity) of
  Definition~\ref{definition:repstab1} hold, together with the
  following:
  \begin{enumerate} 
  \item[{\bf PMIII.}](Stable periodicity of multiplicities): For each
    highest weight vector $\lambda$, the multiplicity of $V(\lambda)$
    as a composition factor in the Jordan--H\"older series for $V_n$
    as a $G_n$--representation is \emph{stably periodic}: there exists
    $C=C_\lambda$ so that for all sufficiently large $n$, this
    multiplicity is periodic in $n$ with period $C$.
  \end{enumerate}
 
  \medskip Similarly we have the corresponding notion of
  \emph{uniformly stably periodic}, where we additionally require that
  the eventual period $C$ does not depend on $\lambda$, and also
  \emph{mixed tensor stably periodic}.  We note that a representation
  stable sequence is also stably periodic with period $C$ for any
  $C\geq 1$.
\end{definition}

We will apply the above definition to give a conjectural picture of
the cohomology of congruence groups.

\begin{conjecture}[Modular periodic stability for congruence groups]
  \label{conjecture:modular1}
  Fix any $i\geq 0$ and any prime $p$.  Then
  \begin{enumerate}
  \item The sequence of $\SL_n(\F_p)$--representations
    $\{H_i(\Gamma_n(p);\F_p)\}$ is uniformly mixed tensor stably
    periodic with period $p$.
  
  \item The sequence of $\Sp_{2n}(\F_p)$--representations
    $\{H_i(\Gamma^{\rm Sp}_{2n}(p);\F_p)\}$ is uniformly stably periodic with
    period $p$.
  \end{enumerate}
\end{conjecture}

We note that mixed tensor representations are really needed in Part 1
of Conjecture~\ref{conjecture:modular1}, since for example
\[H_1(\Gamma_n(p))=\fsl_n\F_p=V(L_1-L_n)=V(\omega_1+\omega_{n-1})=V(1;1)_n\]
is not representation stable, but is mixed representation stable.  We
also remark that periodicity is also needed in the conjecture.  For
example, by the discussion above, the sequence $\{H_1(\Gamma^{\rm
  Sp}_{2n}(2);\F_2)\}$ is a stably periodic sequence of
$\Sp_{2n}\F_2$--representations with stable period $2$.  These
examples also verify that Conjecture~\ref{conjecture:modular1} is true
for $i=1$.

\subsection{The abelianization of the Torelli group}

Dennis Johnson computed that the abelianization of the Torelli group
$\I_{g,1}$ comes from two sources. The first is the so-called Johnson
homomorphism, which is purely algebraically defined, and captures the
action of $\I_{g,1}$ on the universal two-step nilpotent quotient of
$\pi_1(S_{g,1})$ (but see \cite{CF} for a geometric perspective); its
image is $\bwedge^3 H_1(S_{g,1};\Z)$. The second is the
Birman--Craggs--Johnson homomorphism, which views the Torelli group as
gluing maps for Heegard splittings and bundles together the Rokhlin
invariants of the resulting homology 3--spheres. Its image is
2--torsion and is isomorphic to the space $B_3$ of Boolean polynomials
on $H_1(S_{g,1};\F_2)$ of degree at most 3. Johnson showed that these
quotients exhaust the homology of the Torelli group, but with some
overlap. He concludes in \cite{Jo2} that for $g\geq 3$ there is an
isomorphism of abelian groups:
\[H_1(\I_{g,1},\Z)\approx\bwedge^3 H_1(S_{g,1};\Z)\oplus B_2,\]
where $B_2$ is the space of Boolean polynomials of degree at most 2.

The action of $\Sp_{2g}\Z$ on $H_1(\I_{g,1};\Z)$ descends to an action
of $\Sp_{2g}(\Z/2\Z)$ on the torsion subgroup
$H_1(\I_{g,1};\Z)_{\Torsion}\approx B_2$. Shvartsman \cite{Sh} has
recently determined the structure of $H_1(\I_{g,1};\Z)_{\Torsion}$ as
an $\Sp_{2g}(\F_2)$--module. From his calculation we deduce the
following.

\begin{theorem}
\label{thm:torabtor}
The torsion subgroup $H_1(\I_{g,1};\Z)_{\Torsion}$ of the
abelianization of $\I_{g,1}$ is uniformly stably periodic with period
2. The subsequence for even $g$ is uniformly representation
stable, and the subsequence for odd $g$ is uniformly representation
stable.
\end{theorem}

\begin{proof}
  The results of Shvartsman \cite{Sh} give the following list of the
  simple $\Sp_{2g}(\F_2)$--modules appearing in a composition series
  for $H_1(\I_{g,1};\Z)_{\Torsion}$ for $g\geq 3$.  We list modules by
  their highest weight.
  \begin{align*}
    &V(0),V(\omega_1),V(0),V(\omega_2)&\text{ for $g$ odd }\\
    &V(0),V(\omega_1),V(0),V(\omega_2),V(0)&\text{ for $g$ even}
  \end{align*}
  The discrepancy between $g$ even and $g$ odd arises from the same
  source as the corresponding discrepancy for $\bwedge^2
  H_1(S_{g,1};\F_2)$ discussed above.
\end{proof}

\subsection{Level $p$ mapping class groups}
The \emph{level $p$ mapping class group $\Mod_{g,1}(p)$} is the kernel
of the composition \[\Mod_{g,1}\twoheadrightarrow
\Sp_{2g}\Z\twoheadrightarrow \Sp_{2g}\F_p.\] The group $\Mod_{g,1}(p)$
is the ``mod $p$'' analogue of the Torelli group $\I_{g,1}$, since it
is the subgroup of $\Mod_{g,1}$ acting trivially on
$H_1(S_{g,1};\F_p)$.  Hain \cite[Proposition 5.1]{Ha3} proved that for
$g\geq 3$ the group $H^1(\Mod_{g,1}(p);\Z)$ is trivial, so the
abelianization $H_1(\Mod_{g,1}(p);\Z)$ consists entirely of torsion
elements.

Putman \cite{Pu1}, building on work of Sato,  recently proved that elements of
$H_1(\Mod_{g,1}(p);\Z)$ come from three sources. The first is the
abelianization of the congruence subgroup $\fsp_{2g}(\F_p)$, which we
discussed above.  The second source is a ``mod $p$'' version of the
Johnson homomorphism, which has image $\bwedge^3 H_1(S_{g,1};\F_p)$. The
third source contributes only when $p=2$, and is a quotient $B_2/\F_2$
coming from the Birman--Craggs--Johnson homomorphism. The quotient
$\Sp_{2g}\F_p$ naturally acts on $H_1(\Mod_{g,1}(p);\Z)$, and it
follows from Putman's characterization that $H_1(\Mod_{g,1}(p);\Z)$ is
in fact an $\F_p$--representation of $\Sp_{2g}\F_p$.

\begin{theorem}
  \label{thm:ModgL}
  Fix a prime $p$. Then the sequence $\{H_1(\Mod_{g,1}(p);\Z)\}$ of
  $\Sp_{2g}\F_p$--representations is periodically uniformly
  representation stable with period $p$.
\end{theorem}

\begin{proof}
  Let $H:=H_1(S_{g,1};\F_p)$ be the standard representation of
  $\Sp_{2g}\F_p$. For any prime $p$, the representation $\bwedge^3 H$
  has as composition factors the simple $\Sp_{2g}\F_p$--modules:
  \begin{align*}
    &V(\omega_1),V(\omega_3)&\text{for }g\equiv 1\bmod{p}\\
    &V(\omega_1),V(\omega_3),V(\omega_1)&\text{for }g\not\equiv
    1\bmod{p}
  \end{align*}
 
  Putman proves in \cite[Theorem 7.8]{Pu1} that for $p$ odd and $g\geq
  5$, the group $H_1(\Mod_{g,1}(p);\Z)$ is an extension of
  $\fsp_{2g}\F_p$ by $\bwedge^3 H$. Thus $H_1(\Mod_{g,1}(p);\Z)$ has
  composition factors
  \begin{align*}
    &V(\omega_1),V(2\omega_1),V(\omega_3),
    &\text{ for }g\equiv 1\bmod{p}\ \\
    &V(\omega_1)^2,V(2\omega_1),V(\omega_3)&\text{ for }g\not\equiv
    1\bmod{p}.
  \end{align*}

  For $p=2$, Putman proves that $H_1(\Mod_{g,1}(2);\Z)$ is an
  extension of $H_1(\Gamma_{2g}^{\rm Sp}(p);\F_p)$ by
  $\bwedge^3H\oplus B_2/\F_2$. The former has composition factors
  $\fsp_{2g}\F_2$ and $V(\omega_1)$, and Shvartsman describes $B_2$ as
  in Theorem~\ref{thm:torabtor}.  We conclude that for $g\geq 5$, the
  group $H_1(\Mod_{g,1}(2);\Z)$ has the following composition factors
  as an $\Sp_{2g}\F_2$--module:
  \begin{align*}
    &V(0)^2,V(\omega_1)^4,V(\omega_2)^2,V(\omega_3)&\text{ for $g$ odd  }\\
    &V(0)^4,V(\omega_1)^5,V(\omega_2)^2,V(\omega_3)&\text{ for $g$
      even}
  \end{align*}
  Thus in both cases we see that the abelianization is periodic and
  uniformly multiplicity stable with period $p$.
\end{proof}

Given Theorem~\ref{thm:ModgL}, it is natural to make the following
conjecture.

\begin{conjecture}[Modular periodic stability for $\Mod_{g,1}(p)$]
  Fix any $i\geq 0$ and a prime $p$. Then the sequence of
  $\Sp_{2g}\F_p$--representations $\{H_i(\Mod_{g,1}(p);\Z)\}$ is
  uniformly stably periodic with period $p$.
\end{conjecture}

We believe that all of the material in this section can be extended to
corresponding ``level $p$ congruence subgroups'' of $\IA_n$.

\small

\noindent
Dept.\ of Mathematics\\
Stanford University\\
450 Serra Mall\\
Stanford, CA 94305\\
E-mail: church@math.stanford.edu
\medskip

\noindent
Dept. of Mathematics\\
University of Chicago\\
5734 University Ave.\\
Chicago, IL 60637\\
E-mail: farb@math.uchicago.edu
\medskip

\end{document}